\documentclass[11pt,reqno,tbtags,a4paper]{amsart}
\usepackage{amssymb}
\usepackage[utf8]{inputenc}   % om svenska bokstäver
\usepackage{url}
\usepackage[square,numbers]{natbib}
\bibpunct[, ]{[}{]}{;}{n}{,}{,}
\usepackage[pdfencoding=auto,psdextra,breaklinks]{hyperref}
\usepackage{tikz-cd}
\usepackage{xcolor}
\usepackage{xpunctuate}
\usepackage{etoolbox} % used for Remark environment's end-box
\usepackage{siunitx}
\newcommand{\ignore}[1]{\relax}

\title%[]
{Successive minimum spanning trees}

\date{4 June 2019}
\author[Svante Janson]{Svante Janson}
\thanks{The work was partly supported by the Knut and Alice Wallenberg Foundation.}
\address[Svante Janson]{Department of Mathematics, Uppsala University, PO Box 480,
SE-751~06 Uppsala, Sweden}
\email{svante.janson@math.uu.se}
\urladdr{http://www.math.uu.se/svante-janson}

\author[Gregory B. Sorkin]{Gregory B. Sorkin}
\address[Gregory B. Sorkin]{Department of Mathematics,
The London School of Economics and Political Science,
Houghton Street, London WC2A 2AE, England}
\email{g.b.sorkin@lse.ac.uk}
\urladdr{http://personal.lse.ac.uk/sorkin/}


\numberwithin{equation}{section}

\renewcommand\le{\leqslant}
\renewcommand\ge{\geqslant}
\renewcommand\leq{\le}
\renewcommand\geq{\ge}

\allowdisplaybreaks

\definecolor{brown}{cmyk}{0, 0.72, 1, 0.45}
\definecolor{grey}{gray}{0.5}

\definecolor{lightRed}{cmyk}{0, 0.3, 0.3, 0.0}

\theoremstyle{plain}% default
\newtheorem{theorem}{Theorem}[section]
\newtheorem{lemma}[theorem]{Lemma}

\newtheorem{corollary}[theorem]{Corollary}
\newtheorem{conjecture}[theorem]{Conjecture}

\theoremstyle{definition}
\newtheorem{example}[theorem]{Example}

\newtheorem{remark}[theorem]{Remark}

\AtEndEnvironment{remark}{\null\hfill\qedsymbol}
\AtEndEnvironment{example}{\null\hfill\qedsymbol}

\newenvironment{romenumerate}[1][-10pt]{% optional argument changes indentation
\addtolength{\leftmargini}{#1}\begin{enumerate}% gives (i), (ii) etc.
 }{\end{enumerate}}

\newcounter{oldenumi}
{\setcounter{oldenumi}{\value{enumi}}
\begin{romenumerate} \setcounter{enumi}{\value{oldenumi}}}
{\end{romenumerate}}

\newenvironment{alphenumerate}[1][-10pt]{% optional argument changes indentation
\addtolength{\leftmargini}{#1}\begin{enumerate}% gives (i), (ii) etc.
 }{\end{enumerate}}

\newcounter{thmenumerate}

\newcounter{xenumerate}   %no left indentation; thus wider lines

\newcommand\pfitemx[1]{\par#1:}
\newcommand\pfitemref[1]{\pfitemx{\ref{#1}}}

\newcommand{\refT}[1]{Theorem~\ref{#1}}
\newcommand{\refC}[1]{Corollary~\ref{#1}}

\newcommand{\refL}[1]{Lemma~\ref{#1}}
\newcommand{\refR}[1]{Remark~\ref{#1}}
\newcommand{\refS}[1]{Section~\ref{#1}}
\newcommand{\refSS}[1]{Section~\ref{#1}}

\newcommand{\refE}[1]{Example~\ref{#1}}
\newcommand{\refF}[1]{Figure~\ref{#1}}

\newcommand{\refConj}[1]{Conjecture~\ref{#1}}
\newcommand{\refTab}[1]{Table~\ref{#1}}

\begingroup
  \divide\count255 by 60
  \multiply\count255 by -60
  \advance\count255 by \time
  \ifnum \count255 < 10 \xdef\klockan{\the\count1.0\the\count255}
  \else\xdef\klockan{\the\count1.\the\count255}\fi
\endgroup

\newcommand{\sumi}{\sum_{i=1}^\infty}

\newcommand{\sumik}{\sum_{i=1}^k}
\newcommand{\sumin}{\sum_{i=1}^n}

\newcommand\set[1]{\ensuremath{\{#1\}}}

\newcommand\xpar[1]{(#1)}
\newcommand\bigpar[1]{\bigl(#1\bigr)}
\newcommand\Bigpar[1]{\Bigl(#1\Bigr)}

\newcommand\bigsqpar[1]{\bigl[#1\bigr]}

\newcommand\xcpar[1]{\{#1\}}

\newcommand\abs[1]{|#1|}
\newcommand\bigabs[1]{\bigl|#1\bigr|}

\def\rompar(#1){\textup(#1\textup)}    % usage: \rompar(...)
\newcommand\xfrac[2]{#1/#2}

\newcommand\Bigparfrac[2]{\Bigpar{\frac{#1}{#2}}}

\def\xexp(#1){e^{#1}}
\newcommand\ceil[1]{\lceil#1\rceil}
\newcommand\floor[1]{\lfloor#1\rfloor}

\newcommand\ntoo{\ensuremath{{n\to\infty}}}
\newcommand\Ntoo{\ensuremath{{N\to\infty}}}

\newcommand\ktoo{\ensuremath{{k\to\infty}}}

\newcommand\ttoo{\ensuremath{{t\to\infty}}}

\newcommand\norm[1]{\|#1\|}

\newcommand\downto{\searrow}
\newcommand\upto{\nearrow}
\newcommand\punkt{.\spacefactor=1000}    % om problem!
\newcommand\iid{i.i.d\punkt}
\newcommand\ie{i.e\punkt}
\newcommand\eg{e.g\punkt}
\newcommand\viz{viz\punkt}
\newcommand\cf{cf\punkt}
\newcommand{\as}{a.s\punkt}
\newcommand{\aex}{a.e\punkt}
\newcommand\whp{w.h.p\punkt}
\newcommand\whpx{w.h.p\xperiod}

\newcommand{\tend}{\longrightarrow}

\newcommand\pto{\overset{\mathrm{p}}{\tend}}

\newcommand\eqd{\overset{\mathrm{d}}{=}}
\newcommand\eqdef{\stackrel{\operatorname{def}}{=}}

\newcommand\op{o_{\mathrm p}}
\newcommand\Op{O_{\mathrm p}}

\newcommand\bbR{\mathbb R}

\newcounter{CC}
\newcounter{cc}
\newcommand\E{\operatorname{\mathbb E{}}}
\renewcommand\P{\operatorname{\mathbb P{}}}
\newcommand\Var{\operatorname{Var}}

\newcommand\Exp{\operatorname{Exp}}
\newcommand\Po{\operatorname{Po}}
\newcommand\Bi{\operatorname{Bi}}

\newcommand\ga{\alpha}

\newcommand\gd{\delta}

\newcommand\gf{\varphi}
\newcommand\gam{\gamma}
\newcommand\gG{\Gamma}

\newcommand\kk{\kappa}
\newcommand\gl{\lambda}

\newcommand\gs{\sigma}

\newcommand\gth{\theta}
\newcommand\eps{\varepsilon}

\renewcommand\phi{\xxx}  %% WARNING

\newcommand\cC{\mathcal C}

\newcommand\cF{\mathcal F}

\newcommand\cM{\mathcal M}

\newcommand\cP{\mathcal P}

\newcommand\cS{{\mathcal S}}

\newcommand\cV{\mathcal V}

\newcommand\ett[1]{\boldsymbol1\xcpar{#1}}

\newcommand\qw{^{-1}}
\newcommand\qww{^{-2}}
\newcommand\qq{^{1/2}}

\newcommand\intoo{\int_0^\infty}

\newcommand\oi{[0,1]}
\newcommand\ooi{[0,1)}
\newcommand\ooo{[0,\infty)}

\newcommand\dd{\,\mathrm{d}}
\newcommand\ddx{\mathrm{d}}

\newcommand\rhs{right-hand side}

\newcommand\pij{p_{ij}}
\newcommand\pmij{p^-_{ij}}
\newcommand\pmx{p^-}
\newcommand\matris[1]{\xpar{#1}}
\newcommand\matrisn[1]{\matris{#1}\ijn}
\newcommand\pijx{\matris{\pij}}
\newcommand\pijxn{\matrisn{\pij}}
\newcommand\ijn{_{i,j=1}^n}

\newcommand\gnpij{\ensuremath{G(n,\pijx)}}

\newcommand\xG{\dot G}
\newcommand\xx{\mathbf x}
\newcommand\ccc{\mathfrak C^*}
\newcommand\FFm{(F_{k-1}(s))_{s\ge0}}
\newcommand\intS{\int_{\cS}}
\newcommand\gx{\gs} % {\hat\gs}
\newcommand\gxk{{\gx_k}}
\newcommand\pg{permanent giant}
\newcommand\ug{unique giant}
\newcommand\chix{\widehat\chi}
\newcommand\mst{minimum spanning tree}
\newcommand\gvnkt{G^{\cV}(n,\kk_t)}

\newcommand\fX{\mathfrak X}
\newcommand\HS{Hilbert--Schmidt}
\newcommand\oot{[0,t)}
\newcommand\normHS[1]{\norm{#1}_{\mathrm{HS}}}
\newcommand\normLL[1]{\norm{#1}_{L^2}}
\newcommand\mm{^{(m)}}
\newcommand\muki{\mu_{k-1}}
\newcommand\rhoki{\rho_{k-1}}
\newcommand\gxki{\gx_{k-1}}
\newcommand\xPoi{^{\mathrm P}}
\newcommand\xExp{^{\mathrm E}}
\newcommand\xUni{^{\mathrm U}}
\newcommand\eG{G\xExp}
\newcommand\eF{F\xExp}
\newcommand\eT{T\xExp}
\newcommand\pT{T\xPoi}
\newcommand\uT{T\xUni}
\newcommand\eX{X\xExp}
\newcommand\uX{X\xUni}
\newcommand\wi{w_{i}}
\newcommand\mmm{\mathrm{mst}}
\newcommand\hT{\hat T}
\newcommand\xT{\overline T}
\newcommand\xxdots{\dots}
\newcommand\setdiff{\bigtriangleup}
\newcommand\gsx{\gs_\infty}
\newcommand\ab{(a,b]}
\newcommand\gdimp{\bar\gd}
\newcommand\Knoo{K_n^\infty}

\newcommand\ER{Erd\H os--R\'enyi}

\newcommand{\Wk}{W_k}
\renewcommand{\ss}{s}
\newcommand{\ssv}{\mathbf{\ss}}

\renewcommand{\tt}{\tau}
\newcommand{\rr}{r}
\newcommand{\rrv}{\mathbf{\rr}}

\newcommand{\mg}{\succeq}
\newcommand\card{\abs}
\newcommand{\pu}{{\bar p}}

\newcommand{\prr}{{\hat p}}

\begin{document}

\begin{abstract}
In a complete graph $K_n$ with edge weights drawn independently from
a uniform distribution $U(0,1)$
(or alternatively an exponential distribution $\Exp(1)$),
let $T_1$ be the MST (the spanning tree of minimum weight)
and let $T_k$ be the MST after deletion of the edges of all previous trees
$T_i$, $i <k$.
We show that each tree's weight $w(T_k)$
converges in probability to a constant $\gamma_k$
with $2k-2\sqrt k <\gam_k<2k+2\sqrt k$,
and we conjecture that $\gam_k = 2k-1+o(1)$.
The problem is distinct from that of \cite{FriezeJ},
finding $k$ MSTs of combined minimum weight,
and for $k=2$ ours has strictly larger cost.

Our results also hold (and mostly are derived) in a multigraph model
where edge weights for each vertex pair follow a Poisson process;
here we additionally have $\E(w(T_k)) \to \gam_k$.
Thinking of an edge of weight $w$ as arriving at time $t=n w$,
Kruskal's algorithm defines forests $F_k(t)$,
each initially empty and eventually equal to $T_k$,
with each arriving edge added to the first $F_k(t)$ where it does not create a cycle.
Using tools of inhomogeneous random graphs we obtain structural results
including that $C_1(F_k(t))/n$,
the fraction of vertices in the largest component of $F_k(t)$,
converges in probability to a function $\rho_k(t)$, uniformly for all $t$,
and that a giant component appears in $F_k(t)$ at a time $t=\gs_k$.
We conjecture that the functions $\rho_k$ tend to time translations of
a single function,
$\rho_k(2k+x)\to\rho_\infty(x)$ as \ktoo, uniformly in $x\in\bbR$.

Simulations and numerical computations
give estimated values of $\gamma_k$ for small $k$, and support the conjectures just stated.
\end{abstract}

\maketitle

\section{Introduction}\label{S:intro}

\subsection{Problem definition and main results}

Consider the complete graph $K_n$ with edge costs
that are
\iid{} random variables, with a uniform distribution $U(0,1)$
or, alternatively, an
exponential distribution $\Exp(1)$.
A well-known problem is to find the minimum (cost) spanning tree $T_1$, and
its cost or ``weight'' $w(T_1)$.
A famous result by \citet{Frieze} shows that
as \ntoo,
$w(T_1)$ converges in probability to $\zeta(3)$,
in both the uniform and exponential cases.

Suppose now that we want a second spanning tree $T_2$, edge-disjoint from the
first, and that we do this in a greedy fashion by first finding the \mst{}
$T_1$, and then the \mst{} $T_2$
using only the remaining edges.
(I.e., the \mst{} in $K_n\setminus T_1$,
meaning the graph with edge set $E(K_n)\setminus E(T_1)$.)
We then continue and define $T_3$ as the \mst{} in $K_n\setminus(T_1\cup
T_2)$, and so on. The main purpose of the present paper is to show that the
costs $w(T_2)$, $w(T_3)$, \dots also converge in probability to some
constants.

\begin{theorem}   \label{T0}
For each $k\ge1$, there exists a constant $\gam_k$ such that,  as \ntoo,
$w(T_k)\pto\gam_k$
(for both uniform and exponential cost distributions).
\end{theorem}

The result extends easily to other distributions of the edge costs, see
\refR{Rotherdistributions}, but we consider in this paper only the uniform
and exponential cases.

A minor technical problem is that $T_2$
and subsequent trees do not always
exist; it may happen that $T_1$ is a star and then $K_n\setminus T_1$ is
disconnected. This happens only with a small probability, and \whp{} (with
high probability, \ie, with probability $1-o(1)$ as \ntoo) $T_k$ is
defined for every fixed $k$, see \refS{SotherModels}.
However, in the main part of the paper we avoid this problem completely by
modifying the model:
we assume that we have a multigraph,
which we denote by $\Knoo$,
with an infinite number of copies of
each edge in $K_n$,
and that each edge's copies' costs are
given by the points in a
Poisson process with intensity 1 on $\ooo$.
(The Poisson processes for different edges are, of course, independent.)
Note that when finding $T_1$, we only care about the cheapest copy of each
edge,  and its cost has an $\Exp(1)$ distribution, so the problem for
$T_1$ is the same as the original one. However,
on $\Knoo$ we
never run out of
edges and we can define $T_k$ for all integers $k=1,2,3,\dots$.
Asymptotically, the three models are equivalent,
as shown in \refS{SotherModels},
and \refT{T0} holds for any of the models. In particular:
\begin{theorem}
  \label{T0multi}
For each $k\ge1$,  as \ntoo,
$w(T_k)\pto\gam_k$ also for the  multigraph model with Poisson process costs.
\end{theorem}

\citet{Frieze} also proved that the expectation $\E w(T_1)$ converges
  to $\zeta(3)$.
For the multigraph model just described, this too extends.
\begin{theorem}
  \label{TE}
For the Poisson multigraph model,
$\E w(T_k)\to\gam_k$ for each $k\ge1$ as \ntoo.
\end{theorem}

It is well known that the \mst{} (with any given costs,
obtained randomly or deterministically)
can be found by
\emph{Kruskal's algorithm} \cite{Kruskal},
which processes the edges in order of increasing cost and keeps those
that join two different components in the forest obtained so far.
(I.e., it keeps each edge that does not form a cycle together with
previously chosen edges.)
As in many other  previous papers on the random \mst{} problem,
from \cite{Frieze} on, our
proofs are based on analyzing the behavior of this algorithm.

Rescale to think of an edge of weight $w$ as arriving at time $t=n w$.
Kruskal's algorithm allows us to construct all trees $T_k$
simultaneously by growing forests $F_k(t)$,
with $F_k(0)$ empty and $F_k(\infty)=T_k$:
taking the edges of $K_n$
(or $\Knoo$)
in order of time arrival (increasing cost),
an edge is added to the first forest $F_k$ where it does not create a cycle.
We will also consider a sequence of graphs $G_k(t) \supseteq F_k(t)$,
where when we add an edge to $F_k$ we also add it to all
the graphs $G_1,\ldots,G_k$;
see \refS{Smodelmodel} for details.

The proof of \refT{T0} is based on a detailed structural characterization
of the graphs $G_k(t)$, given by \refT{T1}
(too detailed to set forth here in full),
relying heavily on the theory of inhomogeneous random graphs
from \cite{SJ178} and related works.
Where $C_1(G_k(t))$ denotes the number of vertices in the largest component of $G_k(t)$
(or equivalently of $F_k(t)$, as by construction they have the same components)
\refT{T1} shows that
$C_1(G_k(t))/n$
converges in probability to some function $\rho_k(t)$, uniformly for all times $t$.
Moreover, each $G_k$ has its own giant-component threshold:
$\rho_k(t)$ is 0 until some time $\gs_k$,
and strictly positive thereafter.

The functions $\rho_k(t)$ are of central interest.
For one thing, an edge is rejected from $F_k$,
making it a candidate for $F_{k+1}$,
precisely if its two endpoints are within the same component of $F_k$,
and it is shown (see \refC{Cchi})
that this is essentially equivalent to the two endpoints both
being within the largest component.
This line of reasoning yields the constants $\gamma_k$ explicitly,
see \eqref{wlim}, albeit not in a form that is easily evaluated.
We are able, at least, to re-prove (in \refE{Egamma1})
that $\gam_1=\zeta(3)$, as first shown in \cite{Frieze}.

The functions $\rho_k$ also appear to have a beautiful structure,
tending to time-translated copies of a single universal function:
\begin{conjecture} \label{Conj-rhoInfinity}
  There exists a continuous increasing function
  $\rho_\infty(x):(-\infty,\infty)\to\ooi$ such that
$\rho_k(2k+x)\to\rho_\infty(x)$ as \ktoo, uniformly in $x\in\bbR$.
\end{conjecture}

This suggests, though does not immediately imply, another conjecture.
\begin{conjecture}\label{Conj-gamma2}
 For some $\gd$, as \ktoo,
$\gam_k=2k+\gd+o(1)$.
\end{conjecture}
If this conjecture holds, then necessarily $\gd\in[-1,0]$, see \refR{R-gamma2}.

A variety of computational results are
given in \refS{Snumerical}.
They are supportive of \refConj{Conj-rhoInfinity}
(see Figures \ref{Sim1M} and \ref{FigRho3from1to5})
and (see \refF{sim10M})
a stronger version of \refConj{Conj-gamma2}
where we take $\gd=-1$:
\begin{conjecture}\label{Conj2k-1}
As \ktoo,
$\gam_k=2k-1+o(1)$.
\end{conjecture}

Although we cannot prove these conjectures, some
bounds on $\gam_k$ are obtained in \refS{Sbounds}
by a more elementary analysis of the sequence of forests $F_k$.
In particular, \refT{Tbounds} and \refC{Cbounds} lead to the following,
implying that $\gam_k\sim 2k$ as $k\to\infty$.
\begin{corollary}
  \label{Cbound2}
For every $k\ge1$,
\begin{equation}\label{cbound2}
  2k-2k\qq <\gam_k<2k+2k\qq.
\end{equation}
\end{corollary}

See also the related Conjectures
\ref{Conj-gamma1},
\ref{Conj-gx} and
\ref{ConjImproved}.

\begin{remark}\label{RE}
For the simple graph
$K_n$ with, say, exponential costs,
there is as said above a small but
positive probability that $T_k$ does not exist for $k\ge2$.
Hence, either $\E w(T_k)$ is undefined for $k\ge2$,
or we define
$w(T_k)=\infty$ when $T_k$ does not exist,
and then $\E w(T_k)=\infty$ for $k\ge2$ and every $n$.
This is no problem for the convergence in probability in \refT{T0}, but
  it implies that
\refT{TE} does not hold for simple graphs, and the multigraph model is
essential for studying the expectation.
\end{remark}

\begin{remark}
For the \mst{} $T_1$, various further results are known, including refined
estimates for the expectation of the cost $w(T_1)$ \cite{SJ277},
a normal limit law \cite{SJ110},
and asymptotics for the variance \cite{SJ110,SJ110-addendum,Wastlund}.
It seems challenging to show corresponding results for $T_2$ or later trees.
\end{remark}

\subsection{Motivations}
\citet{FriezeJ} recently considered a related problem, where instead of
choosing spanning trees $T_1,T_2,\dots$ greedily one by one, they
choose $k$ edge-disjoint spanning trees with minimum total cost.
It is easy to see, by small examples, that selecting $k$ spanning trees
greedily one by one
does not always give a set of $k$ edge-disjoint spanning trees with minimum
cost, so the problems are different.

We show in \refT{Tdiff}
that, at least
for $k=2$, the two problems also asymptotically have different answers, in
the sense that the limiting values of the minimum cost
--- which exist for both problems  --- are different.
(Also, as discussed in \refS{Sfirstk},
we improve on the upper bound from \cite[Section~3]{FriezeJ}
on the cost of the {net cheapest} $k$ trees,
since our upper bound \eqref{kbounds} on the cost of the {first} $k$ trees
is smaller.)

Both our question and that of \citet{FriezeJ} are natural,
both seem generally relevant to questions of robust network design,
and both have mathematically interesting answers.

\medskip

Another motivation for our question comes from Talwar's ``frugality ratio''
characterizing algorithmic mechanisms (auction procedures) \cite{Talwar}.
The frugality ratio is the
cost paid by a mechanism for a cheapest structure,
divided by the nominal cost of the second-cheapest structure (in the sense of our $T_2$).
Talwar showed that for any matroidal structure
(such as our spanning trees),
in the worst case over all cost assignments,
the frugality ratio of the
famous Vickrey--Clarke--Groves (VCG) auction is 1:
the VCG cost lies between the nominal costs of $T_1$ and $T_2$.
It is natural to wonder, in our randomized setting, how these three costs compare.

\citet{CFMS} show that in
the present setting (MST in $K_n$ with \iid{} $U(0,1)$ edge costs),
the VCG cost is \emph{on average}
exactly 2 times the nominal cost of $T_1$,
and
\citet{JS} show that the VCG cost \emph{converges in probability}
to a limit, namely 2 times the limit $\zeta(3)$ of the cost of $T_1$
(with further results given for all graphs, and all matroids).
\citet{FriezeJ} show that the combined cost of the cheapest pair of
trees converges in expectation to a constant which is numerically about
$4.1704288$,
and Theorem \ref{Tdiff} shows that the cost of $T_1+T_2$
converges in probability to a value that is strictly larger.
It follows that in this average-case setting,
the frugality ratio converges in probability to some value smaller than
$(2 \zeta(3))/ (4.1704288 - \zeta(3))$, about $0.80991$.
So where Talwar found that the (worst-case) frugality ratio
was at most 1 for matroids and could be larger in other cases,
in the present setting it is considerably less than 1.

\subsection{Some notation}

If $x$ and $y$ are real numbers, then $x\vee y:=\max(x,y)$ and $x\land
y:=\min(x,y)$. Furthermore, $x_+:= x\vee 0$.
These operators bind most strongly, e.g., $t-\tau(i)\vee\tau(j)$ means $t-(\tau(i)\vee\tau(j))$.

We use $:=$ as defining its left-hand side, and $\eqdef$ as a reminder that equality
of the two sides is by definition.
We write $\doteq$ for numerical approximate equality,
and $\approx$ for approximate equality in an asymptotic sense
(details given where used).

We use ``increasing'' and ``decreasing'' in their weak senses; for example,
a function $f$ is increasing if $f(x)\le f(y)$ whenever $x\le y$.

Unspecified limits are as \ntoo.
As said above, \whp{} means with probability $1-o(1)$.
Convergence in probability is denoted $\pto$.
Furthermore, if $X_n$ are random variables and $a_n$ are positive constants,
$X_n=\op(a_n)$ means, as usual, $X_n/a_n\pto0$; this is also equivalent to:
for every $\eps>0$, \whp{} $|X_n|<\eps a_n$.

Graph means, in general, multigraph. (It is usually clear from the context
whether we consider a multigraph or simple graph.)
If $G$ is a multigraph, then $\xG$ denotes the
simple graph obtained by merging parallel edges and deleting loops. (Loops
do not appear in the present paper.)
The number of vertices in a graph $G$ is denoted by $|G|$, and the number of
edges by $e(G)$.

For a graph $G$, let $\cC_1(G)$,  $\cC_2(G)$, \dots{} be
the largest component, the second largest
component, and so on, using any rule to break ties.
(If there are less than $k$ components, we define
$\cC_k(G)=\emptyset$.)
Furthermore, let $C_i(G):=|\cC_i(G)|$; thus $C_1(G)$ is the
the number of vertices in the largest component,
and so on.
We generally regard components of a graph $G$ as sets of vertices.

\section{Model and main structural results} \label{Smodel}
\subsection{Model} \label{Smodelmodel}
We elaborate the multigraph model in the introduction.

We consider (random)
(multi)graphs on the vertex set $[n]:=\set{1,\dots,n}$; we usually
omit $n$ from the notation.
The graphs will depend on time, and are denoted
by $G_k(t)$ and $F_k(t)$, where $k=1,2,3,\dots$ and $t\in[0,\infty]$;
they all start  as empty at time $t=0$ and grow as time increases.
We will have $G_k(t)\supseteq G_{k+1}(t)$ and $F_k(t)\subseteq
G_k(t)$ for all $k$ and $t$. Furthermore, $F_k(t)$  will be a forest.
As \ttoo, $F_k(t)$ will eventually become a spanning tree,
$F_k(\infty)$, which is the
$k$th spanning tree $T_k$
produced by the greedy algorithm in the introduction,
operating on the multigraph $G_1(\infty)$.

Since the vertex set is fixed, we may when convenient
identify the multigraphs with sets of edges.
We begin by defining $G_1(t)$ by letting edges arrive as independent Poisson
processes with rate $1/n$ for each pair $\set{i,j}$ of vertices; $G_1(t)$
consists of all edges that have arrived at or before time $t$.
(This scaling of time turns out to be natural and useful.
In essence this is because what is relevant is the cheapest edges on each vertex,
and these have expected cost $\Theta(1/n)$ and thus appear at expected time $\Theta(1)$.)
We define the cost of an edge arriving at time $t$ to be $t/n$, and note
that in $G_1(\infty)$, the costs of the edges joining two vertices form a
Poisson process with rate $1$. Hence, $G_1(\infty)$ is the multigraph model
defined in \refS{S:intro}.

Thus, for any fixed $t\ge0$, $G_1(t)$ is a multigraph where the number of
edges between any two fixed vertices is $\Po(t/n)$, and these numbers are
independent for different pairs of vertices. This is a natural multigraph
version of the \ER{} graph $G(n,t)$.
(The process $G_1(t)$, $t\ge0$, is a continuous-time version of the
multigraph process in \eg{}
\cite{BollobasFrieze} and \cite[Section 1]{SJ97},
ignoring loops.)
Note that $\xG_1(t)$, \ie,  $G_1(t)$
with multiple edges merged,
is simply the random graph $G(n,p)$ with $p=1-e^{-t/n}$.

Next, we let $F_1(t)$ be the subgraph of $G_1(t)$ consisting of every edge that
has arrived at some time $s\le t$ and at that time joined two different
components of $G_1(s)$. Thus, this is a subforest of $G_1(t)$, as stated
above, and it is precisely the forest constructed by Kruskal's algorithm
(recalled in the introduction) operating on $G_1(\infty)$,
at the time all edges with cost $\le t/n$ have been considered.
Hence, $F_1(\infty)$ is the minimum spanning tree $T_1$ of $G_1(\infty)$.

Let $G_2(t):=G_1(t)\setminus F_1(t)$, \ie, the subgraph of $G_1(t)$
consisting of all edges rejected from $F_1(t)$; in other words
$G_2(t)$ consists of the edges that, when they
arrive to $G_1(t)$, have their endpoints in the same component.

We continue recursively. $F_k(t)$ is the subforest of $G_k(t)$ consisting of
all edges in $G_k(t)$ that, when they arrived at some time $s\le t$, joined two
different components in $G_k(s)$. And $G_{k+1}(t):=G_k(t)\setminus F_k(t)$,
consisting of the edges rejected from $F_k(t)$.

Hence, the $k$th spanning tree $T_k$ produced by Kruskal's algorithm
equals $F_k(\infty)$, as asserted above.

Note that $F_k(t)$ is a spanning subforest of $G_k(t)$, in other words, the
components of $F_k(t)$ (regarded as vertex sets) are the same as the
components of $G_k(t)$; this will be used frequently below.
Moreover, each edge in $G_{k+1}(t)$ has endpoints in the same component of
$G_k(t)$; hence, each component of $G_{k+1}(t)$ is a subset of a component
of $G_k(t)$. It follows that an edge arriving to $G_1(t)$ will be passed
through $G_2(t), \dots,G_k(t)$ and to $G_{k+1}(t)$ (and possibly further)
if and only if its endpoints belong to the same component of $G_k(t)$, and thus
if and only if its endpoints belong to the same component of $F_k(t)$.

\subsection{More notation}\label{SSmore}

We say that a component $\cC$ of a graph $G$ is the \emph{\ug}
of $G$ if $|\cC|>|\cC'|$ for every other component $\cC'$;
if there is no such component (\ie, if the maximum size is tied), then we
define the \ug{} to be $\emptyset$.

We say that a component $\cC$ of $F_k(t)$ is the \emph{\pg}
of $F_k(t)$ (or of $G_k(t)$) if it is the \ug{} of $F_k(t)$ and,
furthermore, it is a subset of the \ug{} of $F_k(u)$ for every $u>t$;
if there is no such component then the \pg{} is
defined to be $\emptyset$.

Let $\ccc_k(t)$
denote the \pg{} of
$F_k(t)$. Note that the \pg{} either is empty or the largest component; thus
$|\ccc_k(t)|$ is either 0 or $C_1(F_k(t)) =C_1(G_k(t)) $.
Note also that the \pg{} $\ccc_k(t)$ is an increasing function of $t$:
$\ccc_k(t)\subseteq\ccc_k(u)$ if $t\le u$.
Furthermore, for sufficiently large $t$
(\viz{} $t$ such that $G_k(t)$ is
connected, and thus $F_k(t)$ is the spanning tree $T_k$),
$\ccc_k(t)=\ccc_k(\infty)=[n]$.

\subsection{A structure theorem}

The basis of our proof of Theorems \ref{T0} and \ref{T0multi}
is the following theorem on the
structure of the components of $G_k(t)$.
Recall that $F_k(t)$ has the same components as $G_k(t)$, so the theorem
applies as well to $F_k(t)$.
The proof is given in \refS{Spf}.

For $k=1$, the theorem collects various known results for $G(n,p)$.
Our proof includes this case too, making the proof more
self-contained.

\begin{theorem}\label{T1}
With the definitions above, the following hold for every fixed $k\ge1$ as \ntoo.
  \begin{romenumerate}
\item \label{T1C1}
There exists a continuous increasing function $\rho_k:\ooo\to\ooi$
such that
\begin{equation}\label{t1}
C_1(G_k(t))/n\pto \rho_k(t),
\end{equation}
uniformly in $t\in\ooo$; in other words, for any $\eps>0$, \whp{}, for all
$t\ge0$,
\begin{equation}\label{t1b}
\rho_k(t)-\eps \le C_1(G_k(t))/n\le \rho_k(t)+\eps.
\end{equation}

\item \label{T1C2}
$\sup_{t\ge0}C_2(G_k(t))/n\pto 0$.
\item \label{T1gx}
There exists a threshold $\gx_k>0$ such that $\rho_k(t)=0$ for $t\le\gx_k$,
but $\rho_k(t)>0$ for $t>\gx_k$.
Furthermore, $\rho_k$ is strictly increasing on $[\gx_k,\infty)$.
\item \label{T1rholim}
There exist constants $b_k,B_k>0$ such that
\begin{equation}
  \label{t1rholim}
1-\rho_k(t)\le B_k e^{-b_kt},
\qquad t\ge0.
\end{equation}
In particular,
$\rho_k(t)\to1$ as \ttoo.
\item  \label{T1pg}
If $t>\gx_k$, then \whp{} $G_k(t)$ has a non-empty \pg.
Hence, for every $t\ge0$,
\begin{equation}\label{t1pg}
|\ccc_k(t)|/n
\pto\rho_k(t).
\end{equation}
  \end{romenumerate}
\end{theorem}

  We note also a formula for the number of edges in $G_k(t)$, and
  two simple inequalities relating different $k$.
\begin{theorem}\label{Tlang}
  For each fixed $k\ge1$
  and uniformly for $t$ in any finite interval $[0,T]$,
  \begin{equation}\label{langfredag}
	e(G_k(t))/n \pto \frac12\int_0^t \rho_{k-1}(s)^2\dd s.
  \end{equation}
\end{theorem}

\begin{theorem}\label{Tkk-1}
$\rho_{k}(t)\le\rho_{k-1}(t)$ for every $t\ge0$, with strict inequality when
  $\rho_{k-1}(t)>0$
(equivalently, when $t>\gx_{k-1}$).
Furthermore,
\begin{equation}\label{gxdiff}
\gx_{k} \ge \gx_{k-1}+1.
\end{equation}
\end{theorem}

Note that inequality \eqref{gxdiff} is weak in that we
conjecture that as \ktoo,
$\gx_k =\gx_{k-1}+2+o(1)$, see \refConj{Conj-gx}.

\section{Bounds on the expected cost}\label{Sbounds}
\subsection{Total cost of the first
 \texorpdfstring{$\boldsymbol{k}$}{k} trees}  \label{Sfirstk}
% \unichar{"1D458}   𝑘       f09d9198

The following theorem gives lower and upper bounds on the total cost
of the first $k$ spanning trees.

\begin{theorem} \label{Tbounds}
Letting $\Wk = \sum_{i=1}^k w(T_k)$ be the total cost of the first $k$
spanning trees,
for every $k\ge1$,
\begin{align} \label{kbounds}
 k^2 \frac{n-1}{n}
  & \leq \E \Wk \leq
 k(k+1) \frac{n-1}{n}
< k^2+k.
\end{align}
\end{theorem}

Comparing with \citet[Section~3]{FriezeJ},
our upper bound is smaller than their $k^2+3k^{5/3}$
despite the fact that they considered a more relaxed minimization problem
(see \refS{Sunion});
as such ours is a strict improvement.
In both cases the lower bound is simply the expected total cost of
the cheapest $k(n-1)$ edges in $G$,
with \eqref{ewk} matching \cite[(3.1)]{FriezeJ}.

\begin{proof}
The minimum possible cost of the $k$ spanning trees is the cost of
the cheapest $k(n-1)$ edges.
Since each edge's costs (plural, in our model) are given by
a Poisson process of rate 1,
the set of all edge costs is given by a Poisson process of rate $\binom n 2$.
Recall that in a Poisson process of rate $\gl$,
the interarrival times are independent exponential random variables
with mean $1/\gl$,
so that the $i$th arrival, at time $Z_i$, has $\E Z_i = i/\gl$.
It follows in this case that
$ \Wk  \geq \sum_{i=1}^{k(n-1)} Z_i $ and
\begin{align}\label{ewk}
 \E \Wk & \geq \sum_{i=1}^{k(n-1)} \frac i {\binom n 2}
  = \frac{ (k(n-1)) (k(n-1)+1)} {n(n-1)}
  \geq k^2 \frac{n-1}{n} .
\end{align}

We now prove the upper bound.
An arriving edge is rejected from $F_i$ iff both endpoints lie within
its ``forbidden'' set $B_i$ of edges,
namely those edges with both endpoints in one component.
The nesting property of the components means that
$B_1 \supseteq B_2 \supseteq \cdots$.
An arriving edge $e$ joins $F_k$ if
it is rejected from all previous forests, i.e.,
$e \in B_{k-1}$
(in which case by the nesting property, $e$ also belongs to all earlier $B$s)
but can be accepted into $F_k$, i.e., $e \notin B_k$.
The idea of the proof is to show that the first $k$ forests fill reasonably
quickly with $n-1$ edges each,
and we will do this by coupling the forest-creation process
(Kruskal's algorithm) to a simpler, easily analyzable random process.

Let $\ssv(\tt) = \set{\ss_k(\tt)}_{k=0}^{\infty}$
denote the vector of the sizes
(number of edges) of each forest after arrival of the $\tt$'th edge;
we may drop the argument $\tt$ when convenient.
Let $p_k=\card{B_k}/\binom n 2$, the rejection probability for $F_k$.
For any $\tt$, by the nesting property of the components and in turn of the $B_k$,
\begin{align} \label{sdecr}
 s_1 & \geq s_2 \geq \cdots
 & \text{and} &&
 p_1 & \geq p_2 \geq \cdots .
\end{align}
The MST process can be simulated by using
a sequence of \iid{} random variables $\ga(\tt) \sim U(0,1)$,
incrementing $s_k(\tt)$
if both $\ga(\tt) \leq p_{k-1}(\tt)$
(so that $e$ is rejected from $F_{k-1}$ and thus from all previous forests too)
and $\ga(\tt) > p_k(\tt)$
(so that $e$ is accepted into $F_k$).
We take the convention that $p_0(\tau)=1$ for all $\tau$.
For intuition,
note that when $\ss_k=0$ an edge is never rejected from in $F_k$
($p_k=0$, so $\ga \sim U(0,1)$ is never smaller);
when $\ss_k=1$ it is rejected with probability $p_k = 1/\tbinom n 2$;
and when $\ss_k=n-1$ it is always rejected
($\card{B_k}$ must be $\tbinom n 2$, so $p_k=1$).

Given the size
$s_k = \sum_{i=1}^\infty (C_i(F_k)-1)$
of the $k$th forest,
$\card{B_k}=\sumi\binom{C_i(F_k)}2$ is maximized (thus so is $p_k$)
when all the edges are in one component, i.e.,
\begin{align}
p_k &\leq \binom {s_k+1} 2 \left/ \binom n 2 \right. \label{pkub1}
 \\ &\leq \frac{s_k}{n-1}   \label{pkub2}
 =: \pu_k .
\end{align}
The size vector $\ssv(\tt)$ thus determines the values $\pu_k(\tt)$ for all $k$.

Let $\rrv(\tt)$ denote a vector analogous to $\ssv(\tt)$,
but with $\rr_k(\tt)$ incremented
if $\prr_k(\tt)<\ga(\tt)\le\prr_{k-1}(\tt)$,
with
\begin{align}\label{prr}
\prr_k
 & :=
 \frac{\rr_k}{n-1} .
\end{align}
By construction,
\begin{align} \label{rdecr}
 \rr_1 & \geq \rr_2 \geq \cdots
 & \text{and} &&
 \prr_1 & \geq \prr_2 \geq \cdots .
\end{align}
For intuition, here
note that when $\rr_k=0$ an arrival is never rejected from $\rr_k$
($\pu_k=0$);
when $\ss_k=1$ it is rejected with probability
$\pu_k=1/(n-1) > p_k = 1/\tbinom n 2$;
and when $\ss_k=n-1$ it is always rejected
($\pu_k=1$).

Taking each $F_i(0)$ to be an empty forest ($n$ isolated vertices, no edges)
and accordingly $\ssv(0)$ to be an infinite-dimensional 0 vector,
and taking $\rrv(0)$ to be the same 0 vector,
we claim that for all $\tt$,
$\ssv(\tt)$ majorizes $\rrv(\tt)$,
which we will write as $\ssv(\tt) \mg \rrv(\tt)$.
That is, the prefix sums of $\ssv$ dominate those of $\rrv$:
for all $\tt$ and $k$,
$\sum_{i=1}^k \ss_i(\tt) \geq \sum_{i=1}^k \rr_i(\tt)$.

We first prove this;
then use it to argue that edge arrivals to the first $k$ forests, i.e., to $\ssv$,
can only precede arrivals to the first $k$ elements of $\rrv$;
and finally analyze the arrival times of all $k(n-1)$ elements
to the latter to arrive at an upper bound
on the total cost of the first $k$ trees.

We prove $\ssv(\tt) \mg \rrv(\tt)$ by induction on $\tt$,
the base case with $\tt=0$ being trivial.
Figure \ref{kcoupling} may be helpful in illustrating the structure of
this inductive proof.
Suppose the claim holds for $\tt$.
The probabilities $p_k(\tt)$ are used to determine the forests $F_k(\tt+1)$
and in turn the size vector $\ssv(\tt+1)$.
Consider an intermediate object $\ssv'(\tt+1)$,
the size vector that would be given
by incrementing $\ssv(\tt)$ using the upper-bound values $\pu_k(\tt)$
taken from $\ssv(\tt)$ by \eqref{pkub2}.
Then, $\ss_i(\tt+1)$ receives the increment if
$p_{i-1} \geq \ga > p_i$,
and $\ss'_j(\tt+1)$ receives the increment if
$\pu_{j-1} \geq \ga > \pu_j$;
hence,
from $\pu_{i-1} \geq p_{i-1} \geq \ga$ it is immediate that $i \leq j$
and thus $\ssv(\tt+1) \mg \ssv'(\tt+1)$.

\begin{figure}
\begin{tikzcd}[column sep=huge,font=\normalsize,labels={font=\small}]
% arrow label size: \scriptsize; \small; or \normalsize.
 \textbf{F}(\tt) \arrow[r,"p"] \arrow{d}
    & \textbf{F}(\tt+1) \arrow{d}
  \\
  \ssv(\tt) \arrow[rd,"\pu"]
    \arrow[dotted,-]{dd}{\mg}
    \arrow[r,"p"]
     & \ssv(\tt+1) \arrow[dotted,-]{d}{\mg}
  \\
     & \ssv'(\tt+1)  \arrow[dotted,-,font=\Huge]{d}{\mg}
  \\
  \rrv(\tt) \arrow[r,"\prr"]
      & \rrv(\tt+1)
\end{tikzcd}
\caption{\label{kcoupling}%
Coupling of the forests' sizes $\ssv(\tt)$
to a simply analyzable random process $\rrv(\tt)$,
showing the structure of the inductive proof (on $\tt$) that
$\ssv(\tt)$ majorizes $\rrv(\tt)$.
}
\end{figure}
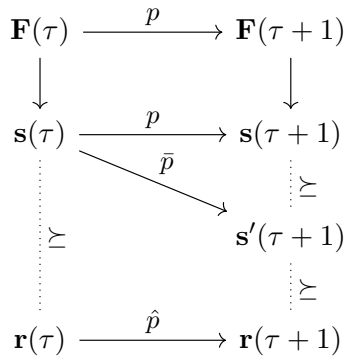

It suffices then to show that $\ssv'(\tt+1) \mg \rrv(\tt+1)$.
These two vectors are obtained respectively from $\ssv(\tt)$ and $\rrv(\tt)$,
with $\ssv(\tt) \mg \rrv(\tt)$ by the inductive hypothesis,
using probability thresholds
$\pu_k(\tt) = f(\ss_k(\tt))$ and
$\prr_k(\tt) = f(\rr_k(\tt))$
respectively,
applied to the common random variable $\ga$,
where $f(s) = s/(n-1)$
(but all that is important is that $f$ is a monotone function of $s$).
Suppose that
\begin{align} \label{splitpoint}
f(\ss_{i-1}) & \geq \ga > f(\ss_i)
& \text{ and } &&
f(\rr_{j-1}) & \geq \ga > f(\rr_j) ,
\end{align}
so that elements $i$ in $\ssv$ and $j$ in $\rrv$ are incremented.
If $i \leq j$, we are done.
(Prefix sums of $\ssv(\tt)$ dominated those of $\rrv(\tt)$,
and an earlier element is incremented in $\ssv'(\tt+1)$ than $\rrv(\tt+1)$,
thus prefix sums of $\ssv'(\tt+1)$ dominate those of $\rrv(\tt+1)$.)
Consider then the case that $i>j$.
In both processes the increment falls between indices $j$ and $i$,
so the $k$-prefix sum inequality continues to hold for $k<j$ and $k \geq i$.
Thus, for $j \leq k < i$,
\begin{align} \label{partsums}
\begin{aligned}
\sum_{\ell=1}^k \ss'_{\ell}(\tt+1)
 &=
 \sum_{\ell=1}^{j-1} \ss_{\ell}(\tt) + \sum_{\ell=j}^k \ss_{\ell}(\tt)
 \\
\sum_{\ell=1}^k \rr_{\ell}(\tt+1)
 &=
  \sum_{\ell=1}^{j-1} \rr_{\ell}(\tt) + 1 + \sum_{\ell=j}^k \rr_{\ell}(\tt) .
\end{aligned}
\end{align}
From $j<i$, \eqref{splitpoint}, and \eqref{sdecr} and \eqref{rdecr} we have
that when $j\le \ell \leq i-1$,
\begin{align}
\ss_{\ell} \geq \ss_{i-1}
 & \geq f^{-1}(\ga) >  \rr_j \geq \rr_{\ell} , \notag
\end{align}
implying
\begin{align}
\ss_\ell \geq \rr_\ell+1 . \label{alphaineq}
\end{align}
In \eqref{partsums},
we have
$\sum_{\ell=1}^{i-1} \ss_{\ell}(\tt)
 \geq
  \sum_{\ell=1}^{i-1} \rr_{\ell}(\tt)$
from the inductive hypothesis that $\ssv(\tt) \mg \rrv(\tt)$,
while using \eqref{alphaineq} gives
\begin{align*}
 \sum_{\ell=j}^k \ss_{\ell}(\tt)
 & \geq \sum_{\ell=j}^k (1+\rr_{\ell}(\tt))
 \geq
  1 + \sum_{\ell=j}^k \rr_{\ell}(\tt) ,
\end{align*}
from which it follows that $\ssv'(\tt+1) \mg \rrv(\tt+1)$,
completing the inductive proof that $\ssv(\tt)\mg \rrv(\tt)$.

Having shown that the vector $\ssv(\tt)$ of component sizes
majorizes $\rrv(\tt)$,
it suffices to analyze the latter.
Until this point we could have used
\eqref{pkub1} rather than \eqref{pkub2}
to define $\pu_k$, $\prr_k$, and the function $f$,
but now we take advantage of the particularly simple nature of the
process governing $\rrv(\tt)$.
Recall that a new edge increments $\rr_i$
for the first $i$ for which the $U(0,1)$
``coin toss'' $\ga(\tt)$ has $\ga(\tt) > \prr_i \eqdef \rr_i/(n-1)$.
Equivalently,
consider an array of cells $n-1$ rows high and infinitely many columns wide,
generate an ``arrival'' at a random row or ``height'' $X(\tt)$ uniform on $1,\ldots,n-1$,
and let this arrival occupy the first unoccupied cell $i$ at this height,
thus incrementing the occupancy $\rr_i$ of column $i$.
This is equivalent because if
$\rr_i$ of the $n-1$ cells in column $i$ are occupied,
the chance that $i$ is rejected
--- that $X(\tt)$ falls into this set and thus the arrival moves along to
test the next column $i+1$ ---
is $\rr_i/(n-1)$, matching \eqref{prr}.

Recalling that the cost of an edge arriving at time $t$ is $t/n$
in the original graph problem, the combined cost $\Wk$ of the first $k$
spanning trees is $1/n$ times
the sum of the arrival times of their $k(n-1)$ edges.
The majorization $\sum_{i=1}^k \ss_i(\tt) \geq \sum_{i=1}^k \rr_i(\tt)$
means that the $\ell$'th arrival to
the first $k$ forests
comes no later than the $\ell$'th arrival to the first $k$ columns
of the cell array.
Thus, the cost $\Wk$ of the first $k$ trees is
at most $1/n$ times the sum of the times of the
$k(n-1)$ arrivals to the array's first $k$ columns.

The continuous-time edge arrivals are a Poisson process
with intensity $1/n$ on each of the $\tbinom n 2$ edges,
thus intensity $(n-1)/2$ in all;
it is at the Poisson arrival times that the discrete time $\tt$ is incremented and
$X(\tt)$ is generated.
Subdivide the ``$X$'' process into the $n-1$ possible values that
$X$ may take on,
so that arrivals at each value (row in the cell array)
are a Poisson process of intensity $\gl = \tfrac12$.
The sum of the first $k$ arrival times in a row
is the sum of the first $k$
arrival times in its Poisson process.
The $i$th such arrival time is the sum of $i$ exponential random variables,
and has expectation $i / \gl$.
The expected sum of $k$ arrival times of a line is thus
$\binom {k+1} 2 / \gl=k(k+1)$,
and (remembering that cost is time divided by $n$),
the expected total cost of all $n-1$ lines is
\begin{align*}
 \frac{n-1}n k(k+1),
\end{align*}
yielding the upper bound in \eqref{kbounds}
and completing the proof of the theorem.
\end{proof}

\begin{corollary}
  \label{Cbounds}
Let $\gG_k:=\sumik\gam_i$. Then, for every $k\ge1$,
\begin{equation}\label{cbounds}
  k^2\le\gG_k=\sumik\gam_i\le k^2+k.
\end{equation}
\end{corollary}

\begin{proof}
  Immediate from Theorems \ref{Tbounds} and \ref{TE}.
\end{proof}

\begin{example}
  In particular, Corollary \ref{Cbounds} gives $1\le\gam_1\le2$ and
  $4\le\gam_1+\gam_2\le6$.
In fact, we know that $\gam_1=\zeta(3)\doteq1.202$ \cite{Frieze} and
$\gam_1+\gam_2>4.1704$ by \cite{FriezeJ} and \refS{Sunion}, see
\refC{Cdiff}.
Numerical estimates suggest
a value of about 4.30; see \refS{Snumerical},
including \refTab{sim10},
for various estimates
of $\gam_2$.
\end{example}

\subsection{Corollaries and conjectures for the \texorpdfstring{$\boldsymbol{k}$}{k}th tree}
Turning to individual $\gam_k$ instead of their sum $\gG_k$,
we obtain
\refC{Cbound2},
namely that $2k-2k\qq <\gam_k<2k+2k\qq$.

\begin{proof}[Proof of \refC{Cbound2}]
  For the upper bound, we note that obviously $\gam_1\le\gam_2\le\dots$, and
  thus, for any $\ell\ge1$, using both the upper and lower bound in
  \eqref{cbounds},
  \begin{equation}
	\begin{split}
\ell \, \gam_k
&\le\sum_{i=k}^{k+\ell-1}\gam_i
= \gG_{k+\ell-1}-\gG_{k-1}
\le (k+\ell-1)(k+\ell)-(k-1)^2
\\&
=\ell^2+\ell(2k-1)+k-1
	\end{split}
  \end{equation}
and hence
  \begin{equation}\label{cb+}
	\begin{split}
\gam_k
\le2k-1+\ell+\frac{k-1}{\ell}.
	\end{split}
  \end{equation}
Choosing $\ell=\ceil{\sqrt{k}}$ gives the upper bound in \eqref{cbound2}.

For the lower bound we similarly have, for $1\le\ell\le k$,
\begin{equation}
  \begin{split}
	\ell\gam_k\ge\gG_k-\gG_{k-\ell}\ge k^2-(k-\ell)(k-\ell+1)
=-\ell^2-(2k+1)\ell-k
  \end{split}
\end{equation}
and hence
  \begin{equation}\label{cb-}
	\begin{split}
\gam_k
\ge2k+1-\ell-\frac{k}{\ell}.
	\end{split}
  \end{equation}
Choosing, again, $\ell=\ceil{\sqrt{k}}$ gives the lower bound in
\eqref{cbound2}.
\end{proof}

\begin{remark}
  For a specific $k$, we can improve \eqref{cbound2} somewhat by instead
  using \eqref{cb+} and \eqref{cb-} with $\ell=\floor{\sqrt k}$ or
$\ell=\ceil{\sqrt{k}}$.
For example, for $k=2$, taking $\ell=1$ yields
$2\le\gam_2\le5$.
For $k=3$, taking $\ell=2$ yields
$3.5\le\gam_3\le8$.
\end{remark}

Besides these rigorous results, taking increments of the left and right-hand sides of
\eqref{cbounds} also suggests
the following conjecture.

\begin{conjecture}\label{Conj-gamma1}
  For $k\ge1$,
$2k-1\le\gam_k\le2k$.
\end{conjecture}

\begin{remark}\label{R-gamma2}
Moreover, if
$\gam_k=2k+\gd+o(1)$ holds, as conjectured
  in Conjecture \ref{Conj-gamma2},
then
$\gG_k=k^2+k(\gd+1)+o(k)$, and
thus
necessarily $\gd\in[-1,0]$ as a consequence of \refC{Cbounds}.
In fact, the numerical estimates in \refS{Snumerical},
specifically Table \ref{sim10} and Figure \ref{sim10M},
suggest that $\gd=-1$;
see \refConj{Conj2k-1}.
\end{remark}

\subsection{Improved upper bounds} \label{SUB2}
The upper bounds in \refT{Tbounds} and \refC{Cbounds} were proved using the
bound \eqref{pkub2}. A stronger, but less explicit, bound can be proved
by using instead the sharper \eqref{pkub1}.
That is,
we consider the random vectors $\rrv(\tt)$
defined as above
but with \eqref{prr} replaced by
\begin{align}\label{prr+}
\prr_k
 & :=
\binom{\rr_k+1}2\Big/\binom{n}2 .
\end{align}
As remarked before \eqref{pkub1}, this approximation comes from imagining
all edges in each $F_k$ to be in a single component;
this overestimates the probability that an arriving edge is rejected from $F_k$
and, as developed in the previous subsection, gives
$\ssv(\tt) \mg \rrv(\tt)$
just as when $\prr_k$ was defined by \eqref{pkub2}.

Using for consistency our usual time scaling
in which edges arrive at rate $(n-1)/2$,
by a standard martingale argument one can show that,
for each $k\ge1$,

\begin{equation}\label{rlim}
  \frac1n \rr_k(\floor{\tfrac12 nt})\pto g_k(t), \qquad\text{uniformly for $t\ge0$},
\end{equation}
for some continuously differentiable functions $g_k(t)$ satisfying
the differential equations, with $g_0(t):=1$,
\begin{equation}\label{pg}
  g_k'(t)= \tfrac12 \left( g_{k-1}(t)^2-g_k(t)^2 \right),
\qquad g_k(0)=0,
\qquad k\ge1.
\end{equation}
Moreover, using $\ssv(\tt) \mg \rrv(\tt)$ and taking limits,
it can be shown that
\begin{equation}\label{gam+}
  \gG_k:=\sumik \gam_i
\leq \frac12 \intoo  t\bigpar{1-g_k(t)^2}\dd t .
\end{equation}
We omit the details,
but roughly,
in time $dt$, $\frac12 n \dd t$ edges arrive, all costing about $t/n$,
and a $g_k(t)^2$ fraction of them pass beyond the first $k$ graphs
(to the degree that we are now modeling graphs).
Compare \eqref{gam+} with \eqref{lwoo},
with reference to \eqref{wtk}.

For $k=1$, \eqref{pg} has the solution
$g_1(t) = \tanh(t/2)$,
and \eqref{gam+}
yields the bound $\gG_1=\gam_1\le2\ln2\doteq1.386$.
This is better than the bound 2
given by \eqref{cbounds}, but still far from precise since
$\gam_1=\zeta(3)\doteq1.202$.

For $k\ge2$ we do not know any exact solution to \eqref{pg},
but numerical solution of \eqref{pg} and calculation of \eqref{gam+}
(see \refS{Srho2Sim}) suggests that $\gG_k < k^2 + 1$.
We leave the proof of this as an open problem.
If proved, this would be a marked improvement on $\gG_k \leq k^2+k$,
which was the exact expectation of the random process given by \eqref{pkub2}
(that part of the analysis was tight).
In particular, it would establish that $2k-2 \leq \gamma_k \leq 2k$;
see \refConj{ConjImproved}.

For $k=2$,
the numerical calculations in \refS{Srho2Sim} give
$\gam_1+\gam_2\le \num{4.5542}\ldots$ % 4.55423272186262
(see \refTab{TgGimproved})
and thus $\gam_2\le \num{3.3521}\ldots$.
The same value was also obtained using Maple's numerical differential equation solver, with Maple giving greater precision but the two methods agreeing in the digits shown here.
 % 3.35217581886262

\section{Preliminaries and more notation}\label{Sprel}

\subsection{Some random graphs} \label{Gpij}
For a symmetric array $\pijxn$ of probabilities in $\oi$, let
$\gnpij$ be the random (simple) graph on the vertex set
$[n]:=\set{1,\ldots,n}$
where the edge
$ij$ appears with probability $\pij$, for $i<j$, and these $\binom n2$
events are independent.
We extend this (in a trivial way) by defining
$G(n,A)=G(n,\matris{a_{ij}}):=G(n,\matrisn{a_{ij}\land 1})$
for any symmetric non-negative $n\times n$ matrix $A=\matrisn{a_{ij}}$.
Moreover, the matrix $A$ can be a random, in which case $G(n,A)$ is defined
by first conditioning on $A$. (Hence, edges appear conditionally
independently, given $A$.)

Note that we do not allow loops, so the diagonal entries $p_{ii}$ or
$a_{ii}$ are
ignored, and may be assumed to be 0 without loss of generality.

\subsection{Susceptibility}

The \emph{susceptibility} $\chi(G)$ of a (deterministic or
random) graph $G$ of order $n=|G|$
is defined by
\begin{equation}\label{chi}
  \chi(G):=\frac1n\sumi C_i(G)^2.
\end{equation}
$\chi(G)$ can be interpreted
as the mean size of the component containing a random
vertex, see \cite{SJ232}.
We also exclude the first term in the sum and define
\begin{equation}\label{chix}
  \chix(G):=\frac1n\sum_{i=2}^\infty C_i(G)^2.
\end{equation}
(This is particularly interesting
for a graph $G$ with a single giant component of order $\Theta(n)$,
when the sum in \eqref{chi} is dominated by the first term.)

Viewing each term in the sums \eqref{chi}--\eqref{chix} as $(C_i(G)/n) \, C_i(G)$,
since $\sumi C_i(G)/n=1$
each sum can be viewed as a weighted sum of the sizes $C_i(G)$ with the weights summing to at most 1, so
\begin{align}
\chi(G)&\le C_1(G),   \label{chiC1}
\\
\chix(G)&\le C_2(G).   \label{chixC2}
\end{align}

Let $\pi(G)$ be the probability that two randomly chosen distinct vertices
in $G$ belong to the same component. Then
\begin{equation}\label{pi}
  \pi(G)=\frac{\sumi C_i(G)\bigpar{C_i(G)-1}}{n(n-1)}
=\frac{n\chi(G)-n}{n(n-1)}=\frac{\chi(G)-1}{n-1}.
\end{equation}

\subsection{Kernels and an integral operator}\label{SSintop}

A \emph{kernel}, or \emph{graphon},
is a non-negative symmetric measurable function $\kk:\cS^2\to\ooo$, where
(in a common abuse of notation)
$\cS=(\cS,\cF,\mu)$ is a probability space.

Given a kernel $\kk$ on a probability space $\cS$,
let $T_{\kk}$ be the integral operator defined by (for suitable
functions $f$ on $\cS$),
\begin{equation}\label{tk}
 \bigpar{T_\kk f}(x)=\intS \kk(x,y)f(y)\dd\mu(y),
\end{equation}
and let $\Phi_{\kk}$ be the non-linear operator
\begin{equation}\label{Phik}
  \Phi_\kk f := 1-e^{-T_\kk f}.
\end{equation}

In our cases, the kernel $\kk$ is bounded, and then $T_\kk$ is a compact (in
fact, \HS) operator on $L^2(\cS,\mu)$. Since furthermore $\kk\ge0$,
it follows that there exists an eigenfunction $\psi\ge0$ on $\cS$ with
eigenvalue $\norm{T_k}$, see \cite[Lemma 5.15]{SJ178},
where
$\norm{T_\kk}$ denotes the operator norm of $T_k$ as an operator on $L^2(\cS,\mu)$.

\subsection{Branching processes}\label{SSbranching}
Given a kernel $\kk$ on a probability space $(\cS,\cF,\mu)$,
as in \cite{SJ178} let
$\fX_\kk(x)$ be the multi-type Galton--Watson branching process with type
space $\cS$, starting with a single particle of type $x\in\cS$, and where
in each generation
a particle of type $y$
is replaced by its children, consisting of
a set of particles distributed as a
Poisson process on $\cS$ with intensity $\kk(y,z)\dd\mu(z)$.
Let further $\fX_\kk$ be the same branching process started with a particle of
random type, distributed as $\mu$.

Let
\begin{equation}\label{rhokdef}
  \rho_\kk(x) \quad \text{ and } \quad \rho(\kk)=\rho(\kk;\mu)
\end{equation}
be the probabilities that
$\fX_\kk(x)$ and $\fX_\kk$, respectively, survive for ever.
Thus,
\begin{equation} \label{rhokappa}
  \rho(\kk)=\int_{\cS}\rho_\kk(x)\dd\mu(x).
\end{equation}

By \cite[Theorem 6.1]{SJ178},
assuming e.g.\ that $\kk$ is bounded as in our cases
(much less will do),
the following hold.
\begin{romenumerate}
\item
The function $\rho_\kk$ is a fixed point of $\Phi_\kk$, \ie,
it satisfies the equation
\begin{equation}\label{rhoPhi}
\rho_\kk=  \Phi_\kk\rho_\kk \eqdef 1-e^{-T_\kk\rho_\kk}.
\end{equation}
Moreover, $\rho_\kk$ is the largest non-negative solution of this equation.
\item
If $\norm{T_\kk}\le1$, then $\rho_\kk(x)=0$ for every $x$, and thus
$\rho(\kk)=0$.
\item
If $\norm{T_\kk}>1$, then
$\rho(\kk)>0$.
\end{romenumerate}

\section{Proof of \refT{T1}}\label{Spf}

\subsection{Basics, and an approximating inhomogeneous
 random graph}
The proof of \refT{T1} is based on induction;  we assume
throughout this section that, for $k \geq 1$,
\refT{T1} holds for $k-1$ and show that it holds for $k$.

For convenience, we define $F_0(t):=G_0(t):=K_n$ for every $t\ge0$;
this enables us to
consider $G_1(t)$ together with $G_{k}(t)$ for $k>1$.
(Alternatively, we could refer to known results for the random graph process
$G_1(t)$.)
Note that \refT{T1} then trivially holds for $k=0$, with $\rho_0(t)=1$ for
all $t$ and $\gx_0=0$, except that \ref{T1gx} has to be modified
(since $\rho_0$ is constant).
There are some trivial modifications
below in the case $k=1$
(and also some, more or less important, simplifications);
we leave these to the reader.

Thus,
fix $k\ge1$, assume that \refT{T1} holds for $k-1$ and
consider the evolution of $G_{k}(t)$.
Essentially everything in this proof depends
on $k$, but we often omit it from the notation. (Recall that we also
usually omit $n$.)

We condition on the entire process $(F_{k-1}(s))_{s\ge0}$.
For two distinct vertices $i,j\in[n]$,
let $\tau(i,j)=\tau_{k-1}(i,j)$ be the time that $i$ and $j$ become members of
the same
component in $F_{k-1}(t)$. This is the time when edges $ij$ start to be
passed to $G_{k}(t)$, and it follows that,
conditionally on $(F_{k-1}(s))_{s\ge0}$,
the process $G_{k}(t)$, $t\ge0$,
can be described as $G_1(t)$ above (\refS{Smodel}),
except that for
each pair \set{i,j} of vertices, edges appear according to a Poisson process
on $(\tau(i,j),\infty)$.
In particular, for a fixed time $t$ (a value, independent of $n$) and conditioned on
$\FFm$,
in the multigraph $G_{k}(t)$, the number of edges $ij$ is
$\Po\bigpar{(t-\tau(i,j))_+/n}$, and these numbers are (conditionally)
independent for different pairs \set{i,j}.
Hence, if we merge multiple edges and obtain the simple graph $\xG_k(t)$,
we see that
\begin{equation}\label{xGkt}
\xG_k(t)=\gnpij,
\end{equation}
the random graph defined in \refSS{Gpij}
with
\begin{equation}\label{pij}
  \pij=\pij(t):=1-e^{-(t-\tau(i,j))_+/n}
=\frac{(t-\tau(i,j))_+}{n} +O\bigpar{n\qww}
\end{equation}
when $i\neq j$, and (for completeness) $p_{ii}=0$.
Note that the probabilities $\pij$ depend on $\FFm$ and thus are random,
and recall that therefore
$\gnpij$ is defined by first conditioning on $(\pij)_{i,j}$ (or on $\FFm$).

Define
\begin{equation}\label{tau}
\tau(i)= \tau_{k-1}(i):=\inf\set{t\ge0:i\in\ccc_{k-1}(t)},
\qquad i=1,\dots,n,
\end{equation}
\ie, the first time that vertex $i$ belongs to the \pg{} of $F_{k-1}$.
Note that
\begin{equation}\label{tautau}
  \tau(i,j)\le \tau(i)\vee\tau(j),
\end{equation}
but strict inequality is possible since $i$ and $j$ may both belong to a
component of $F_{k-1}(t)$ that is not the \pg. We shall see that
this does not happen very often, and one of the ideas in the proof is that
we may
regard the inequality \eqref{tautau} as an approximate equality.
This is formalized in the following lemma,
and leads to a more tractable graph defined in
\eqref{gvnkt}
and compared with $\xG_k(t)$ in \refL{L3}.

\begin{lemma}\label{L1}
For any fixed $t>0$,
  \begin{equation}\label{l1}
	\sum_{i\neq j}
\bigpar{(t-\tau(i,j))_+ -(t-\tau(i)\vee\tau(j))_+}
= \op(n^2).
  \end{equation}
\end{lemma}

\begin{proof}
Fix $\eps>0$,
let $L:=\ceil{t/\eps}$ and let $t_\ell:=\gx_{k-1}+\ell \, \eps$, $\ell=1,\dots,L$.

Say that the pair $(i,j)$ is \emph{bad} if
\begin{align}
\tau(i,j)\le t
\qquad\text{and}\qquad
\tau(i)\vee\tau(j)-\tau(i,j)>2\eps.
\end{align}
Note that, using \eqref{tautau},
for any pair $(i,j)$,
\begin{align}\label{winbad}
0\le (t-\tau(i,j))_+ -(t-\tau(i)\vee\tau(j))_+ &\le t
\intertext{and for a good pair (\ie, a pair that is not bad),}
(t-\tau(i,j))_+ -(t-\tau(i)\vee\tau(j))_+ &\le 2\eps.\label{wingood}
\end{align}

By the induction hypothesis \refT{T1}\ref{T1pg},
\whp{}
$G_{k-1}(t_1)$ has a \pg, so we may assume that this holds.
(Failures contribute $\op(n^2)$ to the \rhs{} of \eqref{l1}.)

If $(i,j)$ is bad,
then either $\tau(i,j)\le \gx_{k-1}-\eps$, or there exists $\ell\in [1,L]$ such that
$\tau(i,j)\le t_\ell< \tau(i)\vee\tau(j)$.
In the first case, $i$ and $j$ belong to the same component in
$G_{k-1}(\gx_{k-1}-\eps)$, and in the second case they belong to the same component
in $G_{k-1}(t_\ell)$,
but not to the largest one, since that is assumed to be the \pg.
Hence, for any $t$, the number of bad pairs $(i,j)$
is at most,
using the definitions \eqref{chi}--\eqref{chix},
\begin{equation}\label{eleo}
  n\chi\bigpar{G_{k-1}(\gx_{k-1}-\eps)}
+ \sum_{\ell=1}^{L} n\chix\bigpar{G_{k-1}(t_\ell)}.
\end{equation}
By \eqref{chiC1} and the induction hypothesis \ref{T1C1} and \ref{T1gx},
\begin{equation}\label{eleon}
  \chi\bigpar{G_{k-1}(\gx_{k-1}-\eps)} \le
C_1\bigpar{G_{k-1}(\gx_{k-1}-\eps)}
=n\rho_{k-1}(\gx_{k-1}-\eps) +\op(n)
=\op(n)
\end{equation}
and similarly for every $\ell$,
by \eqref{chixC2} and the induction hypothesis \ref{T1C2},
\begin{equation}\label{eleono}
\chix\bigpar{G_{k-1}(t_\ell)}
\le C_2\bigpar{G_{k-1}(t_\ell)}
=\op(n).
\end{equation}
By \eqref{eleo}--\eqref{eleono}, the number of bad pairs is $\op(n^2)$.
Hence, using \eqref{winbad} and \eqref{wingood}, we obtain
  \begin{equation}
	\sum_{i\neq j}
\bigpar{(t-\tau(i,j))_+ -(t-\tau(i)\vee\tau(j))_+}
\le 2\eps n^2 + t\op(n^2),
  \end{equation}
and the result follows since $\eps$ is arbitrary.
\end{proof}

We use the machinery and notation in \citet[in particular Section 2]{SJ178},
and make the following definitions:
\begin{itemize}
\item
$\cS$ is the set
\begin{equation}\label{Sdef}
 \cS:=\ooo .
\end{equation}
\item
$\mu_{k-1}$ is the probability measure on $\cS$ with distribution function
  \begin{equation}\label{muk}
\mu_{k-1}([0,x])=\rho_{k-1}(x). 	
  \end{equation}
\item
$\xx_n:=(x_1,\dots,x_n)$ where
  \begin{equation} \label{x=tau}
  x_i=\tau_{k-1}(i) .
  \end{equation}
\item
$\nu_n$ is a probability measure given by
\begin{equation}\label{nun}
 \nu_n:=\frac{1}n\sumin \gd_{x_i},
\end{equation}
where $\gd_x$ is the point mass (Dirac
  delta) at $x$. (In other words,
  $\nu_n$ is the
  empirical distribution of $\set{x_1,\dots,x_n}$.
  Put yet another way, for any set $A \subset \cS$,
  $\nu_n(A) := \frac{1}n\bigabs{\set{i \colon x_i \in A}}$.)
\end{itemize}

Note that $\xx_n$ and $\nu_n$ are random, and determined by $\FFm$.

\begin{lemma}\label{L2}
  $\nu_n\pto\mu_{k-1}$ in the space $\cP(\cS)$ of probability measures on $\cS$.
\end{lemma}

\begin{proof}
  The claim is equivalent to
  \begin{equation}
	\label{cp1}
\nu_n[0,x]\pto\mu_{k-1}[0,x]
  \end{equation}
for every continuity point $x$ of $\mu_{k-1}$, see \eg{}
\cite[Lemma A.2 and	Remark A.3]{SJ178} or \cite[Section 3]{Billingsley}.
(In our case, for $k>1$, $\mu_{k-1}$ is a continuous measure, by \eqref{muk}
and \ref{T1C1}, so we should consider all $x$.)
However, \eqref{cp1} follows from \eqref{tau} and the induction hypothesis
\ref{T1pg}, which yield,
for any $x\ge0$,
\begin{equation}
  \nu_n[0,x]=\frac{1}n\bigabs{\set{i:\tau_{k-1}(i)\le x}} \label{xcount}
=\frac{1}n\bigabs{\ccc_{k-1}(x)}
\pto\rho_{k-1}(x) =\mu_{k-1}[0,x].
\end{equation}
\end{proof}

In the terminology of \cite[Section 2]{SJ178},
$(\cS,\mu_{k-1})$ is a \emph{ground space} and, by \refL{L2},
\begin{equation}\label{V}
  \cV:=(\cS,\mu_{k-1},(\xx_n)_{n\ge1})
\end{equation}
is a \emph{vertex space},
meaning that the number of vertices $x_i$ appearing by time $t$
is governed by $\mu_{k-1}$, as made precise by \eqref{xcount}.
We define also, for every $t\ge0$, the \emph{kernel}
\begin{equation}\label{kkt}
  \kk_t(x,y):=(t-x\vee y)_+ = (t-x)_+\land (t-y)_+,
\qquad x,y\in\cS=\ooo.
\end{equation}
Note that, for fixed $t$, the kernel $\kk_t$ is bounded and continuous;
hence $\kk_t$ is a \emph{graphical kernel}
\cite[Definition 2.7, Remark 2.8 and Lemma 8.1]{SJ178}.
Furthermore, $\kk_t$ is strictly positive, and thus irreducible, on
$[0,t)\times[0,t)$, and 0 on the complement
$\cS^2\setminus[0,t)^2$; hence, $\kk_t$ is
\emph{quasi-irreducible}
\cite[Definition 2.11]{SJ178}, provided $t>\gx_{k-1}$ so $\rho_{k-1}(t)>0$.
(If $t\le\gx_{k-1}$, then $\mu_{k-1}[0,t]=\rho_{k-1}(t)=0$, and thus
$\kk_t=0$ $\mu_{k-1}^2$-\aex{} on $\cS^2$.)

\bigskip
As detailed in \cite[Section 2]{SJ178},
specifically near its (2.3),
these ingredients define a random graph
\begin{equation}
  \gvnkt .
\end{equation}
Recall that in our case the kernel
$\kk_t$ is given by \eqref{kkt} while the vertex space
$\cV$ is given by \eqref{V}, in turn
with $\cS$ and $\mu_{k-1}$ given by \eqref{Sdef} and \eqref{muk},
and $\xx_n$ given by \eqref{x=tau}.

In general, $\gvnkt$
denotes a random graph with vertices arriving at random times $\xx_n$,
vertices $i$ and $j$ joined with probability $\kk_t(x_i,x_j)$,
and \cite{SJ178} describes the behavior of such an inhomogeneous random graph.
It suffices to think of $\gvnkt$ in terms of $\cS$, $\mu_{k-1}$, and $\kk_t$,
because as shown in \cite{SJ178}
the particulars of $\xx_n$ are irrelevant as long as $\xx_n$ is consistent with $\mu_{k-1}$ in the sense of \eqref{xcount} and \eqref{nun},
this consistency
following from the fact that $\cV$ is a vertex space (see \eqref{V} and the line following it).

Here, $\gvnkt$ is the random graph
alluded to after \eqref{tautau},
a proxy for $G_k(t)$
with the difference that it is
based on the times $\tau(i)$
of vertices joining the permanent giant
of $F_{k-1}$,
rather than the more complicated two-variable times $\tau(i,j)$
of two vertices first belonging to a common component.
Concretely,
\begin{equation}\label{gvnkt}
\gvnkt \eqdef G(n,\matris{\pmij}),
\end{equation}
the \rhs{} being the random graph defined in \refSS{Gpij} with
(recalling \eqref{kkt} and \eqref{x=tau})
\begin{equation}\label{pmij}
  \pmij:=\frac{1}n \kk_t(x_i,x_j)
=\frac{1}n\bigpar{t-x_i\vee x_j}_+
=\frac{1}n\bigpar{t-\tau(i)\vee \tau(j)}_+
\end{equation}
when $i\neq j$, and (for completeness) $\pmx_{ii}=0$.
We assume throughout that $n \geq t$,
so that $\pmij \in [0,1]$;
this is not an issue since $t$ is fixed while $n \to \infty$.
Note that by \eqref{pij} and \eqref{l1},
\begin{equation}\label{ee0}
  \begin{split}
\sum_{i, j}|\pij-\pmij|
=
\sum_{i\neq j} \frac{(t-\tau(i,j))_+ -(t-\tau(i)\vee\tau(j))_+}n +O(1)
= \op(n).
  \end{split}
\end{equation}

Recall that both $\pij$ and $\pmij$ depend on $\FFm$, and thus are random.
By \eqref{xGkt} and \eqref{gvnkt},
$\xG_{k}(t)=G(n,\matris{\pij})$
and $\gvnkt=G(n,\matris{\pmij})$, so
by making the obvious maximal coupling of $G(n,\matris{\pij})$ and
$G(n,\matris{\pmij})$
conditionally on $\FFm$,
we obtain a coupling of $\xG_k(t)$ and $\gvnkt$ such
that if $e\bigpar{\xG_k(t)\setdiff\gvnkt}$ is the number of edges that are
present in one of the graphs but not in the other, then
\begin{equation}\label{ee}
  \E\bigpar{ e\bigpar{\xG_k(t)\setdiff\gvnkt} \mid\FFm}
=
\sum_{i< j}|\pij-\pmij|.
\end{equation}

\begin{lemma}
  \label{L3}
For every fixed $t\ge0$,
\begin{equation}\label{l3}
  e\bigpar{\xG_k(t)\setdiff\gvnkt} = \op(n).
\end{equation}
\end{lemma}

\begin{proof}
  Let $\eps>0$.
For convenience, let
$X_n:= e\bigpar{\xG_k(t)\setdiff\gvnkt}$
and
$Y_n:=\E\bigpar{X_n \mid\FFm}$.
Then,
using Markov's inequality,
for any $\gd>0$,
\begin{equation}
  \P\bigpar{X_n>\eps n}
\le \P(Y_n>\gd n)+ \E\P(X_n>\eps n\mid Y_n\le \gd n)
\le\P(Y_n>\gd n)+
\frac{\gd n}{\eps n}.
\end{equation}
But $Y_n = \op(n)$ by \eqref{ee} and \eqref{ee0}, so $\P(Y_n>\gd n)=o(1)$.
Hence $\P(X_n>\eps n)\le \gd/\eps+o(1)$. Since $\gd$ is arbitrary, this
shows $\P(X_n>\eps n)=o(1)$, \ie, $0\le X_n\le \eps n$ \whp, which completes the
proof.
\end{proof}

\subsection{Towards part \ref{T1C1}}
The following lemma establishes \eqref{t1} of \refT{T1}\ref{T1C1}
  for any fixed $t \geq 0$; doing so
uniformly for all $t \geq 0$, as the theorem states, follows later.
Here we rely on \cite[Theorem 3.1]{SJ178}, which, roughly speaking,
relates the size of the largest component of a random graph $\gvnkt$,
to the survival probability of the branching process defined by the same kernel $\kk_t$
and the measure (here $\mu_{k-1}$) comprised by the vertex space $\cV$.
By \refL{L3}, the graph $\xG_k(t)$ of interest differs from $\gvnkt$ in only $\op(n)$ edges,
and the stability theorem \cite[Theorem 3.9]{SJ178}
shows that the size of the largest component of $\xG_k(t)$ is about the same as that of $\gvnkt$.

Let $\fX_{t}=\fX_{t,k}:=\fX_{\kk_t}$ be the branching process
defined in \refSS{SSbranching} for the kernel $\kk_t$ and the measure $\mu_{k-1}$,
and (recalling \eqref{rhokdef}) let
$\rho(\kk_t) \eqdef \rho(\kk_t;\mu_{k-1})$ be its survival probability.

\begin{lemma}
  \label{LC1}
For every fixed $t\ge0$, \eqref{t1} holds with
\begin{equation}\label{rhok}
  \rho_k(t):=\rho(\kk_t;\mu_{k-1}) \eqdef \rho(\kk_t),
\end{equation}
the survival probability of the branching process $\fX_t$.
\end{lemma}

Do not confuse $\rho_k$ with $\rho_\kk$,
respectively the $\rho$ and $\rho_\kk$ of \eqref{rhokdef}.

\begin{proof}
Fix $t\ge0$.  First, if $t\le\gx_{k-1}$, then, by the induction hypothesis,
$C_1(G_{k-1}(t))/n\pto\rho_{k-1}(t)=0$. Since each component of $G_k(t)$ is
a subset of a component of $G_{k-1}(t)$, we have
$C_1(G_k(t))\le C_1(G_{k-1}(t))$, and thus also
$C_1(G_{k}(t))/n\pto0$, which shows \eqref{t1} with $\rho_k(t)=0$.
The survival probability $\rho(\kk_t)=0$ here as well, establishing \eqref{rhok},
as indeed no particle has any children:
recalling the definitions in \refS{SSbranching},
the total number of children of a particle of any type $y$
is  $\Po \left( \int_{0}^{\infty}\kk_t(y,z)\dd\mu_{k-1}(z)dz \right) = \Po(0)$,
since for $z < \gx_{k-1}$ we have $\mu_{k-1}(z)=0$ thus $\dd\mu_{k-1}(z)=0$,
while for $z \geq \gx_{k-1} \ge t$ we have $\kk_t(y,z)=0$.

Hence we may in the rest of the proof assume $t>\gx_{k-1}$ and thus
$\mu_{k-1}(t)=\rho_{k-1}(t)>0$. As noted after \eqref{kkt} above, the kernel $\kk_t$ then is
quasi-irreducible.
Hence, it follows from \cite[Theorem 3.1]{SJ178} that
\begin{equation} \label{C1Gnu}
  C_1(\gvnkt)/n\pto \rho(\kk_t).
\end{equation}

We have shown in \refL{L3} that $\xG_k(t)$ differs from
$\gvnkt$ by only $\op(n)$ edges, and we appeal to
the stability theorem \cite[Theorem 3.9]{SJ178}
to show that largest components of these two graphs have essentially the same size.
(Alternatively, we could use \cite[Theorem 1.1]{SJ223}.)
A minor technical problem is that this theorem is stated for irreducible
kernels, while $\kk_t$ is only quasi-irreducible. We can extend the theorem
(in a standard way)
by considering only the vertices $i$ with $x_i=\tau_{k-1}(i)\le t$,
\ie, the vertices $i$ in the \pg{} $\ccc_{k-1}(t)$ of $G_{k-1}(t)$, see \eqref{tau}.
This defines a \emph{generalized vertex space} \cite[Section 2]{SJ178}
$\cV'=(\cS',\mu_{k-1}',(\xx_n')_{n\ge1})$, where
$\cS':=[0,t]$, $\mu'_{k-1}$ is the restriction of $\mu_{k-1}$ to $\cS'$, and
$\xx_n'$ is the subsequence of $\xx_n=(x_1,\dots,x_n)$ consisting of all
$x_i\in\cS'$. The kernel $\kk_t$ is strictly positive \aex{} on
$\cS'\times\cS'$, and is thus irreducible.

Thus, we may take $G_n$ in \cite[Theorem 3.9]{SJ178} to be
\begin{equation} \label{Gn}
   G_n := G^{\cV'}(n,\kk_t),
\end{equation}
which may
be thought of as the restriction of $G^{\cV}(n,\kk_t)$ to $\ccc_{k-1}(t)$.
Take the theorem's
$G_n'$ to be
\begin{equation} \label{Gn'}
   G_n' := \xG_k(t) \left[ \ccc_{k-1}(t) \right],
\end{equation}
the restriction of $\xG_k(t)$ to $\ccc_{k-1}(t)$.
For any $\gd>0$,
from \refL{L3}, \whp\ $e\bigpar{\xG_k(t)\setdiff\gvnkt} \le \gd n$.
Restricting each of these graphs to $\ccc_{k-1}(t)$,
it follows that \whp\
\begin{equation} \label{Gn-Gn'}
  e(G_n'\setdiff G_n) \le \gd n .
\end{equation}
Thus, $G_n$ and $G_n'$ fulfill the theorem's hypotheses.
For any $\eps>0$, we may choose $\gd>0$ per the theorem's hypotheses,
and it follows from
the theorem and \eqref{C1Gnu}
that \whp{}
\begin{equation}\label{jul}
 \bigpar{\rho(\kk_t)-\eps}n\le
  C_1(G_n')
  \le \bigpar{\rho(\kk_t)+\eps}n.
\end{equation}

Our aim is to establish \eqref{t1b}, which is \eqref{jul} with $C_1(G_k(t))$ in lieu of $C_1(G_n')$.
Each component $\cC$ of $G_k(t)$
(or equivalently of $\xG_k(t)$)
is a subset of some component of $G_{k-1}(t)$,
either $\cC_1(G_{k-1}(t))$ or some other component.
Since $t > \gx_{k-1}$ and
by the induction hypothesis \ref{T1pg} of \refT{T1},
\whp{} $\ccc_{k-1}(t) \neq \emptyset$
and thus $\cC_1(G_{k-1}(t)) = \ccc_{k-1}(t)$.
Thus, components of $G_k(t)$ contained in $\cC_1(G_{k-1}(t))$ are also
contained in
$G_n'$, and the largest such component is governed by \eqref{jul}.
Components of $G_k(t)$
contained in a smaller component of $G_{k-1}(t)$
have size at most $C_2(G_{k-1}(t))$,
which by the induction hypothesis \ref{T1C2} is \whp{} smaller than any constant times $n$,
and thus smaller than the component described by \eqref{jul}.
Consequently,
\whp{} $C_1(G_k(t))= C_1(G_n')$,
and thus
\eqref{jul} implies \eqref{t1b} \whp, for every $\eps>0$, which is equivalent to
\eqref{t1}.
\end{proof}

\subsection{Towards part \ref{T1C2}}
The next lemma establishes something like \refT{T1} \ref{T1C2},
but only for any fixed $t \ge 0$; extending this to the supremum follows later.

\begin{lemma}\label{LC2}
  For every fixed $t\ge0$,
$C_2(G_k(t))=\op(n)$.
\end{lemma}

\begin{proof}
 We use the notation  of the proof of \refL{LC1},
specifically \eqref{Gn} and \eqref{Gn'}.
Let $G_n^\dag$ be the graph $G_n'$  with a
single edge added such that the two largest components $\cC_1(G_n')$ and
$\cC_2(G_n')$ are joined and let $\eps>0$.
(If $\cC_2(G_n')=\emptyset$, let $G_n^\dag:=G_n'$.)
Since \whp{}
the analog of \eqref{Gn-Gn'} holds also for $G_n^\dag$,
\cite[Theorem 3.9]{SJ178} applies also to $G_n$ and
$G_n^\dag$
and shows that \whp{}
\begin{equation}\label{juldag}
  C_1(G_n')+C_2(G_n')=C_1(G_n^\dag) \le \bigpar{\rho(\kk_t)+\eps}n.
\end{equation}
This and \eqref{jul} imply that \whp{}
\begin{equation}\label{annandag}
C_2(G_n') \le 2\eps n.
\end{equation}
Furthermore, as shown in the proof of \refL{LC1}, \whp{} every component of $G_k(t)$
that is \emph{not} part of $G_n'$ has size at most $C_2(G_{k-1}(t))$, which \whp{}
is $\le\eps n$
by the induction hypothesis. Consequently, \whp{}
\begin{equation}\label{tredjedag}
C_2(G_k(t)) \le 2\eps n,
\end{equation}
which completes the proof.
\end{proof}

Let $T_t=T_{t,k}:=T_{\kk_t}$ be the integral operator
defined by
\eqref{tk} with the measure $\mu_{k-1}$.
We regard $T_t$ as an operator on $L^2(\cS,\mu_{k-1})$, and recall that
(since $\kk_t$ is bounded) $T_t$ is a bounded and compact operator for every
$t\ge0$.

\begin{lemma}\label{LTnorm}
The operator norm $\norm{T_t}$
is a continuous function of $t\ge0$.
Furthermore,
\begin{alphenumerate}
\item \label{LT0}
$\norm{T_t}=0\iff t\le\gx_{k-1}$
\item \label{LT>}
$\norm{T_t}$ is strictly  increasing on $[\gx_{k-1},\infty)$.
\item \label{LToo}
$\norm{T_t}\to\infty$ as \ttoo.
\end{alphenumerate}
\end{lemma}

\begin{proof}
  If $0\le t\le u<\infty$, then
by \eqref{kkt},
for any $x,y\in\cS=\ooo$,
\begin{equation}
0\le\kk_u(x,y) -\kk_t(x,y)\le u-t,
\end{equation}
and consequently,
with $\normHS{\cdot}$ the \HS{} norm,
  \begin{equation}
	\begin{split}
\bigabs{\norm{T_t}-\norm{T_u}}
\le \norm{T_t-T_u}
\le\normHS{T_t-T_u}
=\norm{\kk_t-\kk_u}_{L^2(\cS^2,\mu_{k-1}^2)}
\le |t-u|.
	\end{split}
  \end{equation}
Hence, $t\mapsto\norm{T_t}$ is continuous.

\pfitemref{LT0}
By \eqref{kkt}, $\kk_t(x,y)>0\iff (x,y)\in\oot^2$.
If $t\le\gx_{k-1}$, then
$\mu_{k-1}\oot=\rho_{k-1}(t)=0$, and thus $\kk_t=0$ a.s., so $T_t=0$.

Conversely, if $t>\gx_{k-1}$, then $\kk_t>0$ on a set of positive measure,
and thus $T_t1> 0$ on a set of positive measure (where $1$ denotes the
function that is constant 1); hence $\norm{T_t}\ge\normLL{T_t1}>0$.

\pfitemref{LT>} Assume $u>t>\gx_{k-1}$.
Then, as just shown, $\norm{T_t}>0$. Moreover, as said in \refSS{SSintop},
there exists an eigenfunction $\psi_t\ge0$ with eigenvalue $\norm{T_t}$,
which we can assume is normalized: $\normLL{\psi_t}=1$.
Since $\kk_t>0$ on $[0,t)\times[0,t)$, and 0 elsewhere, it follows
from $\psi_t = T_t \psi_t/\norm{T_t}$ and \eqref{tk} that
$\psi_t=0$ \aex{} on $[t,\infty)$, and
thus nonzero somewhere on $[0,t)$,
and thus $\psi_t>0$ \aex{} on $[0,t)$.
As $u>t$, then $\kk_u(x,y)>\kk_t(x,y)$ on $[0,t)\times[0,t)$, and it follows
	that for $x\in\oot$,
\begin{equation}
	  \begin{split}
(T_u\psi_t)(x)&
=
\intoo\kk_u(x,y)\psi_t(y)\dd\mu_{k-1}(y)	
\\&
>
\intoo\kk_t(x,y)\psi_t(y)\dd\mu_{k-1}(y)
=T_t\psi_t(x)
=\norm{T_t}\psi_t(x).		
	  \end{split}  \label{Tpsix}
	\end{equation}
Thus,
\begin{equation}
  \norm{T_u}\ge\normLL{T_u\psi_t}>\norm{T_t}\,\normLL{\psi_t}
=\norm{T_t}.
\end{equation}
Consequently, \ref{LT>} holds.

\pfitemref{LToo}
By \eqref{kkt}, $\kk_t(x,y)\upto\infty$ as \ttoo{} for every
$x,y\in\ooo$. Hence,
using monotone convergence,
$T_t1(x)=\intoo\kk_t(x,y)\dd\mu_{k-1}(y)\upto\infty$ for every $x\in\ooo$,
and thus by monotone convergence again,
$\normLL{T_t1}\upto\infty$. Consequently, $\norm{T_t}\ge\normLL{T_t1}\to\infty$.
\end{proof}

\subsection{Proofs of parts \ref{T1gx} and \ref{T1rholim}}
\begin{proof}[Proof of \refT{T1}\ref{T1gx}]
By \refL{LTnorm}, there exists a unique $\gx_k>0$ such that
\begin{equation}\label{tgx1}
\norm{T_{\gx_k}}=1.
\end{equation}
Furthermore, $\norm{T_t}<1$ if $t<\gx_k$ and $\norm{T_t}>1$ if $t>\gx_k$.
Thus,
\begin{equation}\label{tgx2}
    t>\gx_k
\iff
    \norm{T_t}>1
\iff
    \rho_k(t)>0 ,
\end{equation}
where the last equivalence follows from \cite[Theorem 3.1]{SJ178},
establishing $\norm{T_t}>1$ as a necessary and sufficient condition for
the existence of a giant component, and providing its size.
In order to see that $\rho_k$ is strictly increasing on $[\gx_k,\infty)$,
  let $\gx_k<t<u$. Since $\kk_u(x,y)\ge\kk_t(x,y)$ for all $x,y\in\cS$,
we may couple the branching processes $\fX_t=\fX_{\kk_t}$ and
$\fX_u=\fX_{\kk_u}$ such that
$\fX_{u}$ is obtained from $\fX_{t}$ by adding
extra children to some individuals.
(Each individual of type $x$ gets extra children of type $y$
distributed as a Poisson process with intensity
$(\kk_u(x,y)-\kk_t(x,y))\dd\mu_{k-1}(y)$, independent of everything else.)
Then clearly
$\fX_{u}$ survives if
$\fX_{t}$ does, so $\rho_k(u):=\rho(\kk_u)\ge\rho(\kk_t) = \rho_k(t)$.
(See \cite[Lemma 6.3]{SJ178}.)
Moreover, there is a positive probability that
$\fX_{t}$ dies out but
$\fX_{u}$ survives, for example
because the initial particle has no children in
$\fX_{t}$ but at least one in $\fX_{u}$, and this child starts a
surviving branching process.
Hence $\rho_k(u)>\rho_k(t)$.
\end{proof}

\bigskip
We next prove \refT{T1}\ref{T1rholim}.
A simple lemma will be useful here and subsequently.

Consider the process defined in \refS{Smodel}
of all graphs $G_j(t)$, $j\ge1$ and $t\ge0$,
under some edge-arrival process;
consider also a similar set of graphs $G'_j(t)$
coming from a second arrival process thicker than the first
(\ie, containing the same arrivals and possibly others).

\begin{lemma} \label{thicker}
The thicker process yields larger graphs, \ie, $G_j(t) \subseteq G'_j(t)$ for all $j$ and $t$.
Also, any edge $e$ present in both arrival processes, if contained in $F'_1(t) \cup\dots\cup F'_j(t)$,
is also contained in $F_1(t) \cup\dots\cup F_j(t)$.
\end{lemma}

\begin{proof}
It is easy to see that adding edges can only make $G_1$ larger, \ie, that $G'_1(t) \supseteq G_1(t)$.
Thus any edge originally passed on to $G_2(t)$ will still be passed on, plus
perhaps some others;
by induction on $j$, any $G_j(t)$ can only increase, \ie, $G'_j(t) \supseteq G_j(t)$.
This proves the first assertion.
The second assertion
follows from
the first.
  If $e$ is not contained in $F_1(t) \cup\dots\cup F_j(t)$
then it is passed on to $G_{j+1}(t)$, and hence, as just shown,
it belongs also to $G'_{j+1}(t)$ and therefore
not to $F'_1(t) \cup\dots\cup F'_j(t)$.
\end{proof}

Let $q_n(t)$ be the probability that two fixed, distinct, vertices in
$G_k(t)$ belong to the same component. By symmetry, this is the same for
any pair of vertices, and thus also for a random pair of distinct vertices.
Hence, recalling \eqref{pi},
\begin{equation}\label{qn}
  q_n(t)=\E \pi(G_k(t))
 .
\end{equation}

\begin{lemma}\label{Lqn}
  There exist constants $b_k,B_k>0$ such that, for every $n\ge2$,
\begin{equation}
  \label{lqn}
q_n(t)\ge 1-B_k e^{-b_kt},
\qquad t\ge0.
\end{equation}
\end{lemma}

\begin{proof}
Fix some $t_0>\gx_k$; thus $\rho_k(t_0)>0$ by \ref{T1gx}.
Then, \cf{}
\eqref{pi},
writing $C_i:=C_i(G_k(t_0))$,
\begin{equation}\label{NT}
  q_n(t_0) = \E \frac{\sumi C_i(C_i-1)}{n(n-1)}
\ge \E \frac{ C_1(C_1-1)}{n(n-1)}.
\end{equation}
By \refL{LC1}, $C_1/n\pto\rho_k(t_0)$ as \ntoo, and thus
$C_1(C_1-1)/(n(n-1))\pto\rho_k(t_0)^2$.
Hence, by
\eqref{NT} and dominated convergence
(see \eg{} \cite[Theorems 5.5.4 and 5.5.5]{Gut}),
\begin{equation}\label{NTP}
\liminf_\ntoo  q_n(t_0)
\ge \lim_\ntoo \E \frac{ C_1(C_1-1)}{n(n-1)}=\rho_k(t_0)^2>0.
\end{equation}
Let $q:=\rho_k(t_0)^2/2$, say. Then \eqref{NTP} shows that if $n$ is large
enough,  $q_n(t_0)>q$. By reducing $q$, if necessary, we may assume that
this holds for every $n$, since obviously $q_n(t_0)>0$ for every fixed $n\ge2$.

For an integer $m\ge0$, consider the process defined in \refS{Smodel}
of all graphs $G_j(t)$, $j\ge1$ and $t\ge0$,
but erase all edges and restart at $mt_0$;
denote the resulting random graphs by $G_j\mm(t)$ and note that
$G_j\mm(t+mt_0)\eqd G_j(t)$. In particular, let
$G_{k,m}:=G_k\mm((m+1)t_0)$.
Then $G_{k,m}\eqd G_k(t_0)$; furthermore, the random graphs $G_{k,m}$,
$m=0,1,\dots$, are independent, since they depend on edges arriving in
disjoint time intervals.

Consider the process at times $i \, t_0$ for integers $i$.
By \refL{thicker},
$G_k(i \, t_0)$ dominates what it would have been had no edges arrived by $(i-1) t_0$,
which, for $i\ge2$, is
simply an independent copy of $G_k(t_0)$
(that is, independent of $G_k(t_0)$ but identically distributed).
Consequently, for any integer $M$, vertices $x$ and $y$ can be in different components of $G_k(M t_0)$
only if they are in different components in each of the $M$ copies of $G_k(t_0)$.
Thinking of all values $i \ge1$ at once, these copies of $G_k(t_0)$ are all independent,
as they depend on edge arrivals in disjoint time intervals $\big( (i-1)t_0, i\,t_0 \big]$.
Thus,
\begin{equation}
1- q_n(Mt_0)\le (1-q_n(t_0))^M\le (1-q)^M\le e^{-qM}.
\end{equation}
Thus, for any $t\ge0$, taking $M:=\floor{t/t_0}$,
\begin{equation}
q_n(t)\ge q_n(Mt_0)
\ge 1-e^{-q\floor{t/t_0}}
\ge 1-e^{q-(q/t_0)t},
\end{equation}
which shows \eqref{lqn}.
(In fact, we get $B_k=e^q<e$; we can take $q$ arbitrarily small and thus
$B_k$ arbitrarily close to 1, at the expense of decreasing $b_k$.)
\end{proof}

\begin{proof}[Proof of \refT{T1}\ref{T1rholim}]
Let $\cC^1(t)$ be the component of $G_k(t)$ that contains vertex 1.
Then, by \refL{Lqn},
\begin{equation}\label{kkl}
  \E C_1(G_k(t))
\ge\E \cC^1(t)
=(n-1) q_n(t)
\ge (n-1)\bigpar{1-B_ke^{-b_kt}}.
\end{equation}
Furthermore, by \refL{LC1} and dominated convergence,
$\E C_1(G_k(t))/n\to\rho_k(t)$ as \ntoo. Hence, \eqref{kkl} implies
$\rho_k(t) \ge 1-B_ke^{-b_kt}$, which is \eqref{t1rholim}.
\end{proof}

\subsection{Proofs of parts \ref{T1C1}, \ref{T1C2}, and \ref{T1pg}}
\begin{proof}[Proof of \refT{T1}\ref{T1C1}]
If $t_n$ is a sequence such that either $t_n\upto t$ or $t_n\downto t$,
then $\kk_{t_n}(x,y)\to\kk_t(x,y)$ for all $x$ and $y$ by \eqref{kkt}, and
thus by \cite[Theorem 6.4]{SJ178}, recalling \eqref{rhok},
$\rho_k(t_n) =\rho(\kk_{t_n}) \to\rho(\kk_{t}) =\rho_k(t)$.
Hence, $\rho_k$ is continuous.
That $\rho_k$ is increasing was shown above in
\refT{T1}\ref{T1gx}.

Let $\eps>0$ and let $N$ be an integer with $N\ge\eps\qw$.
Since $\rho_k$ is continuous, and $\rho_k(t)\to1$ as \ttoo{}
by \refT{T1}\ref{T1rholim}, we can choose
$t_0=0<t_1<\dots<t_{N-1}<t_N=\infty$ with
\begin{equation}\label{tj}
  \rho_k(t_j)=j/N .
\end{equation}
By \refL{LC1}, \whp{}
\begin{equation}\label{trettiosju}
\rho_k(t_j)-\eps\le \xfrac{ C_1\bigpar{G_k(t_j)}}n
\le \rho_k(t_j)+\eps
\end{equation}
for every $j\le N$. (The case $j=N$ is trivial, since  $G_k(\infty)$ \as{}
is connected.)
Then, for every $j=1,\dots, N$ and every $t\in[t_{j-1},t_j]$,
\begin{equation}\label{tretton}
\xfrac{C_1\bigpar{G_k(t)}}n \le
\xfrac{C_1\bigpar{G_k(t_j)}}n
\le \rho_k(t_j)+\eps
\le \rho_k(t)+\frac{1}N+\eps
\le \rho_k(t)+2\eps,
\end{equation}
which together with a similar lower bound shows that \whp{}
$|\xfrac{C_1\bigpar{G_k(t)}}n -\rho_k(t)|\le 2\eps$ for all $t\ge0$.
Since $\eps$ is arbitrary, this shows \eqref{t1b}.
\end{proof}

\begin{proof}[Proof of \refT{T1}\ref{T1C2}]
Let $\eps$, $N$, and $t_j$, $j=0,\dots,N$,
be as
in the proof of \refT{T1}\ref{T1C1} above.
Again, \whp, \eqref{trettiosju} holds for every $j\le N$.
Moreover, by \refL{LC2}, for every $j\le N-1$, and trivially when $j=N$, \whp{}
\begin{equation}\label{c2}
  C_2\bigpar{G_k(t_j)}\le \eps n.
\end{equation}
Assume \eqref{trettiosju} and \eqref{c2} for every $j\le N$, and also
that $C_2\bigpar{G_k(t)}>3\eps n$ for some $t\ge0$. Choose $j$ with $1\le
j\le N$ such
that $t\in[t_{j-1},t_j]$.
If $\cC_2(G_k(t))$ has \emph{not} merged with $\cC_1(G_k(t))$ by
time $t_j$,
then
\begin{equation}
  C_2\bigpar{G_k(t_j)}
\ge   C_2\bigpar{G_k(t)}>3\eps n,
\end{equation}
which contradicts \eqref{c2}.
If on the other hand these two
components \emph{have} merged, then,
using \eqref{trettiosju}
and (from \eqref{tj}) that $\rho_k(t_{j-1}) \ge \rho_k(t_j)-\eps$,
\begin{equation}
  \begin{split}
  C_1\bigpar{G_k(t_j)}&
\ge   C_1\bigpar{G_k(t)}+  C_2\bigpar{G_k(t)}
 \,>\,   C_1\bigpar{G_k(t_{j-1})}+  3\eps n
\\&
\ge \rho_k(t_{j-1})n -\eps n+ 3\eps n
 \, \ge \, \rho_k(t_{j})n + \eps n,
  \end{split}
\end{equation}
which contradicts \eqref{trettiosju}.
Consequently, \whp{} $\sup_t C_2\bigpar{G_k(t)}\le3\eps n$.
\end{proof}

\begin{proof}[Proof of \refT{T1}\ref{T1pg}]
  If $t>\gx_k$, then $\rho_k(t)>0$ by \refT{T1}\ref{T1gx}.
Let $\gd=\rho_k(t)/2$.
Then, by \ref{T1C1} and \ref{T1C2}, \whp{}
$C_1(G_k(t))>\gd n$,
and, simultaneously for every $u\ge0$,
$C_2(G_k(u))<\gd n$.
Assume that these inequalities hold.
Then, in particular, the largest component of $G_k(t)$ is a \ug.
(Recall the definition from \refS{SSmore}.)
Moreover, for every $u\ge t$, the component $\cC$ of $G_k(u)$
that contains the largest component of $G_k(t)$ then satisfies
\begin{equation}
|\cC|\ge C_1(G_k(t))>\gd n > C_2(G_k(u)),
\end{equation}
showing that $\cC$ is the \ug{} of $G_k(u)$.
Hence, the largest component of $G_k(t)$ is \whp{} a \pg.

Consequently, if $t>\gx_k$, then \whp{}
$|\ccc_k(t)|=C_1(G_k(t))$ and \eqref{t1pg} follows from \eqref{t1}.
On the other hand, if $t\le\gx_k$, then \eqref{t1} and \ref{T1gx} yield
\begin{equation}
|\ccc_k(t)|/n\le C_1(G_k(t))/n \pto \rho_k(t)=0,
\end{equation}
and \eqref{t1pg} follows in this case too.
\end{proof}

This completes the proof of \refT{T1}.

\subsection{A corollary}
We note the following corollary.

\begin{corollary}
  \label{Cchi}
For $k\ge1$,
uniformly for all $t\in\ooo$,
\begin{align}
  \chi\bigpar{G_k(t)}/n&\pto \rho_k(t)^2, \label{cchi}
\\
  \chix\bigpar{G_k(t)}/n&\pto0.\label{cchix}
\end{align}
\end{corollary}

\begin{proof}
First, \eqref{cchix} follows immediately from \eqref{chixC2} and
\refT{T1}\ref{T1C2}.

Next, by the definitions \eqref{chi}--\eqref{chix},
\begin{equation}
  \chi\bigpar{G_k(t)}/n
=   \chix\bigpar{G_k(t)}/n + C_1\bigpar{G_k(t)}^2/n^2,
\end{equation}
and \eqref{cchi} follows by \eqref{cchix} and \eqref{t1}.
\end{proof}

\begin{remark}
Using \cite[Theorem 4.7 and Lemma 2.2]{SJ232}
together with results above
(in particular \eqref{xGkt}--\eqref{pij} and
Lemmas \ref{L1} and \ref{L2}),
it is not difficult to prove the much stronger
results that if $t<\gx_k$ is fixed,
then there exists a finite constant $\chi_{k}(t)$ such that
$
  \chi(G_k(t))\pto \chi_{k}(t)
$,
and
if $t\neq\gx_k$ is fixed,
then there exists a finite constant $\chix_{k}(t)$ such that
$
  \chix(G_k(t))\pto \chix_{k}(t)
$.
Furthermore, these limits can be calculated from the branching process
$\fX_t=\fX_{\kk_t}$ on $(\cS,\mu_{k-1})$:
if we let $|\fX_t|$ be the total population of the
branching process, then
$  \chi_k(t)=\E(|\fX_t|)$
and $\chix_k(t)=\E\bigpar{|\fX_t| \, \ett{|\fX_t|<\infty}}$.
We omit the details.
\end{remark}

\begin{proof}[Proof of \refT{Tkk-1}]
  Since $\kk_t(x,y)=0$ for every $y$ when $x\ge t$, a particle of type
$x\ge  t$ will not get any children at all in the branching process
  $\fX_{t,k}=\fX_{\kk_t}$,
hence has survival probability $\rho_\kappa(x)=0$. Thus,
recalling \eqref{rhokappa} and \eqref{muk},
the survival probability
  \begin{equation}\label{kk1}
	\rho_k(t)
  = \int_{0}^{\infty} \rho_{\kappa}(x) d\mu_{k-1}(x)
  \leq \int_{0}^{t} 1 \, d\mu_{k-1}(x)
  = \mu_{k-1}\oot =\rho_{k-1}(t).
  \end{equation}
Moreover, even if $x<t$, there is a positive probability that $x$ has no
children in $\fX_{t,k}$, and thus there is strict inequality in \eqref{kk1}
whenever $\rho_{k-1}(t)>0$.

Turning to the threshold, we note that
$\mu[0,\gx_{k-1})=\rho_{k-1}(\gx_{k-1})=0$, and thus $\mu_{k-1}$-\aex{} $x$
satisfies $x\ge\gx_{k-1}$,
in which case $\kk_t(x,y)\le (t-\gx_{k-1})_+$.
In particular, if $t<\gx_{k-1}+1$, then
$\norm{T_t}\le\normHS{T_t}
=\normLL{\kk_t}
\le(t-\gx_{k-1})_+<1$ and hence $t< \gx_k$,
see after $\eqref{tgx1}$.
Consequently, $\gx_k\ge\gx_{k-1}+1$.

An alternative view of the last part is that, asymptotically,
no edges arrive in $G_k(t)$ until $t=\gx_{k-1}$, and even if
all edges were passed on to $G_k(t)$ from that instant,
$G_k(t)$ would thenceforth evolve as a simple \ER\ random graph,
developing a giant component only 1 unit of time later, at $t=\gx_{k-1}+1$.
\end{proof}

\section{Proof of \refT{T0multi}}\label{SpfT0}

For $a$ and $b$ with $0\le a< b\le\infty$, let $N_k(a,b)$ be the number of
edges that arrive to $G_1(t)$
during the interval $\ab$ and are \emph{not} passed on to $G_{k+1}(t)$;
furthermore, let $W_k(a,b)$ be their total cost.
In other words, we consider the edges, arriving in $\ab$, that end up in
one of $T_1=F_1(\infty), \dots, T_k=F_k(\infty)$.
In particular, for $0\le t\le\infty$,
\begin{align}
  N_k(0,t)&=\sum_{i=1}^k e(F_i(t)),\label{nk}\\
 W_k(0,t)&=\sum_{i=1}^k w(F_i(t))\label{wk}
\intertext{and thus}
    W_k(0,\infty)&=\sum_{i=1}^k w(T_i).\label{wtk}
\end{align}
Since an edge arriving at time $t$ has cost $t/n$, we have
\begin{equation}\label{nw}
  \frac{a}{n}N_k(a,b) \le W_k(a,b)\le \frac{b}{n}N_k(a,b).
\end{equation}

\begin{lemma}
  \label{LN}
Let $0\le a<b\le \infty$ and $k\ge1$.
For any $\eps>0$, \whp
\begin{equation}\label{ln}
\tfrac12  (b-a)\bigpar{1-\rho_k(b)^2-\eps}n
\le N_k(a,b)
\le \tfrac12  (b-a)\bigpar{1-\rho_k(a)^2+\eps}n.
\end{equation}
\end{lemma}
\begin{proof}
Let $\cF_t$ be the $\gs$-field generated by everything that has happened up
to time $t$.
At time $t$, the fraction of edges arriving to $G_1(t)$ that are
rejected by all of $F_1(t),\dots,F_k(t)$ is simply the fraction lying within a component of $F_k(t)$,
namely $\pi(G_k(t))$ (see \eqref{pi}).
Since edges arrive to $G_1(t)$ at a total rate $\frac{1}n\binom n2=\frac{n-1}2$,
conditioned on $\cF_t$,
edges are added to $F_1(t)\cup\dots\cup F_k(t)$ at a rate,
using \eqref{pi},
\begin{equation}\label{mia}
  r_k(t):=\frac{n-1}2\bigpar{1-\pi(G_k(t))}
=\frac{n-\chi(G_k(t))}{2}.
\end{equation}
By \refC{Cchi}, for every fixed $t$,
\begin{equation}\label{rklim}
  r_k(t)/n\pto \bigpar{1-\rho_k(t)^2}/2.
\end{equation}

Condition on the event
$r_k(a)\le \bigpar{1-\rho_k(a)^2+\eps}n/2$, which by \eqref{rklim} occurs \whpx.
Then, since $r_k(t)$ is a decreasing function of $t$,
the process of edges that are added to $F_1(t)\cup\dots\cup F_k(t)$ can for
$t\ge a$ be coupled with a Poisson process with constant intensity
$\bigpar{1-\rho_k(a)^2+\eps}n/2$ that is thicker (in the sense defined just
before \refL{thicker}).
Thus, letting $Z$ be the number arriving in the latter process in $\ab$,
we have \whp{}
\begin{equation}\label{nk1}
  N_k(a,b) \le Z \sim \Po \bigpar{(b-a)\xpar{1-\rho_k(a)^2+\eps}n/2}.
\end{equation}
Furthermore,  by the law of large numbers, \whp
\begin{equation}\label{nkz}
Z \le (b-a)\bigpar{1-\rho_k(a)^2+2\eps}n/2.
\end{equation}
Combining \eqref{nk1} and \eqref{nkz} yields
the upper bound
in \eqref{ln} (with $2\eps$ in place of $\eps$).

For the lower bound, we stop the entire process
as soon as
$r_k(t)<\frac12\bigpar{1-\rho_k(b)^2-\eps}n$.
Since $r_k(t)$ is decreasing, if the stopping condition does not hold at time $t=b$
then it also does not hold at any earlier time, so by
\eqref{rklim}, \whp{} we do not stop before $b$.
As long as we have not stopped,
we can couple with a Poisson process with constant intensity
$\bigpar{1-\rho_k(b)^2-\eps}n/2$ that is thinner (\ie, opposite to thicker),
and we obtain the lower bound in
\eqref{ln} in an
analogous way as the upper bound.
\end{proof}

\begin{lemma}
  \label{LW1}
If\/  $0<b<\infty$, then
\begin{equation}
  \label{lw1}
W_k(0,b) \pto \frac12\int_0^b \bigpar{1-\rho_k(t)^2}t\dd t.
\end{equation}
\end{lemma}

\begin{proof}
Let $N\ge1$ and define $t_j:=jb/N$.
By \eqref{nw} and \refL{LN}, for every $j \in [N]$,
\whp
\begin{equation}
 W_k(t_{j-1},t_j)
\le t_j \tfrac12(t_j-t_{j-1})  \bigpar{1-\rho_k(t_{j-1})^2+N\qw}
=\int_{t_{j-1}}^{t_j} f_N(t)\dd t,
\end{equation}
where we define the piecewise-constant function $f_N$ by
\begin{equation}\label{fN}
f_N(t):=  \tfrac12 t_j  \bigpar{1-\rho_k(t_{j-1})^2+N\qw},
\qquad t_{j-1}\le t<t_j.
\end{equation}
Consequently, \whp,
\begin{equation}\label{sopor}
  W_k(0,b)=\sum_{j=1}^N W_k(t_{j-1},t_j)
\le\sum_{j=1}^N \int_{t_{j-1}}^{t_j} f_N(t)\dd t
=\int_0^b f_N(t)\dd t.
\end{equation}

Now let $N\to\infty$. The functions $f_N(t)$ are uniformly bounded on $[0,b]$,
and $f_N(t)\to f(t):=\frac12t\bigpar{1-\rho_k(t)^2}$ as \Ntoo{} for every
$t$ by \eqref{fN} and the continuity of $\rho_k(t)$.
Hence, dominated convergence yields
$\int_0^b f_N(t)\dd t\to \int_0^b f(t)\dd t$. Given $\eps>0$, we may thus
choose $N$ such that
$\int_0^b f_N(t)\dd t< \int_0^b f(t)\dd t+\eps$, and then \eqref{sopor}
shows that \whp
\begin{equation}
  W_k(0,b) <  \int_0^b f(t)\dd t+\eps.
\end{equation}
We obtain a corresponding lower bound
similarly, using the lower bounds in \eqref{nw} and \eqref{ln}.
Consequently, $W_k(0,b)\pto\int_0^b f(t)\dd t$, which is \eqref{lw1}.
\end{proof}

We want to extend \refL{LW1} to $b=\infty$.
This will be \refL{LWoo}, but to prove it we need the following lemma.
\begin{lemma}\label{LWin}
For any $k\ge1$ there exist constants $b_k',B_k'>0$
such that, for all $t\ge0$,
\begin{equation}  \label{lwin}
  \E W_k(t,\infty)
\le B_k' e^{-b_k't}.
\end{equation}
  \end{lemma}

\begin{proof}
For any $t$, recalling that $N_k$ counts edges arriving at rate $r_k(t)$
and that $r_k(t)$ is a decreasing function,
 we obtain by \eqref{mia},
 \eqref{qn}, and \refL{Lqn},
\begin{equation}
  \E N_k(t,t+1)
  \leq \E r_k(t)
  = \frac{n-1}{2} \bigpar{1-q_n(t)}
\le {n\,} B_k e^{-b_k t} .
\end{equation}
Thus, by \eqref{nw}, for $b_k':=b_k/2$ and some $B_k''<\infty$,
\begin{equation}
  \E W_k(t,t+1) \le \frac{t+1}n\E N_k(t,t+1)
\le (t+1) B_k e^{-b_k t}
\le B_k'' e^{-b_k' t}.
\end{equation}
Hence, for some $B_k'<\infty$ and all $t\ge0$,
\begin{equation}
  \E W_k(t,\infty) = \sum_{j=0}^\infty \E W_k(t+j,t+j+1) \le B_k' e^{-b_k't}.
\end{equation}
\end{proof}

\begin{lemma}  \label{LWoo}
\begin{equation}
  \label{lwoo}
W_k(0,\infty) \pto \frac12\int_0^\infty \bigpar{1-\rho_k(t)^2}\,t\dd t
\: < \: \infty.
\end{equation}
\end{lemma}

\begin{proof}
  First,
  $1-\rho_k(t)^2\le 2(1-\rho_k(t))\le 2B_ke^{-b_kt}$ by
  \refT{T1}\ref{T1rholim},
  establishing that
  the integral converges.

Let $\eps>0$. We may choose $b<\infty$ such that
\begin{equation}
\frac12\int_0^b\bigpar{1-\rho_k(t)^2}t\dd t
> \frac12\int_0^\infty\bigpar{1-\rho_k(t)^2}t\dd t-\eps,
\end{equation}
and then \refL{LW1} shows that \whp{}
\begin{equation}\label{winb}
  W_k(0,\infty)\ge W_k(0,b)
\ge \frac12\int_0^\infty\bigpar{1-\rho_k(t)^2}t\dd t-\eps.
  \end{equation}

For an upper bound, we note that
by \refL{LWin},
given $\gd,\eps>0$, we may choose $b$ such that
$\E W_k(b,\infty)<\gd\eps$ (for all $n$), and thus with probability
$\ge1-\gd$,
$W_k(b,\infty)<\eps$.
It then follows from \refL{LW1} that with probability $\ge 1-\gd-o(1)$
(failure probabilities at most $\gd$ for the first inequality and $o(1)$ for the second),
\begin{equation}\label{winc}
  \begin{split}
  W_k(0,\infty)&<W_k(0,b)+\eps
\le
 \frac12\int_0^b\bigpar{1-\rho_k(t)^2}t\dd t+2\eps
\\&
\le
 \frac12\int_0^\infty\bigpar{1-\rho_k(t)^2}t\dd t+2\eps. 	
  \end{split}
\end{equation}
Since $\gd$ is arbitrary, \eqref{winc} holds \whp, which together with
\eqref{winb} shows \eqref{lwoo}.
\end{proof}

\begin{proof}[Proof of \refT{T0multi}]
  By \eqref{wtk} and \refL{LWoo} (with $W_0(0,\infty)=0$),
  \begin{equation}\label{wlim}
	w(T_k)=W_k(0,\infty)-W_{k-1}(0,\infty)
\pto
\gam_k:=
\frac12\int_0^\infty \bigpar{\rho_{k-1}(t)^2-\rho_k(t)^2}\,t\dd t.
  \end{equation}
\end{proof}

\begin{example}\label{Egamma1}
  The limit $\gam_k$ in \refT{T0multi} is thus given by the integral in
  \eqref{wlim}. Unfortunately, we do not know how to calculate this, even
  numerically, for $k\ge2$. However, we can illustrate the result with the
  case $k=1$. In this case, $\rho_1(t)$ is the asymptotic relative size of
  the giant component in $G(n,t/n)$, and as is well-known, and follows from
  \eqref{rhok} and \eqref{rhoPhi}
  noting that $\kk_t(x,y)=t$,
$\gx_1=1$ and for $t>1$, $\rho_1(t)=1-e^{-t\rho_1(t)}$. The latter function
  has the inverse $t(\rho)=-\log(1-\rho)/\rho$, $\rho\in(0,1)$. Hence, by an
  integration by parts and two changes of variables, with $\rho=1-e^{-x}$,
  \begin{equation}\label{kings}
	\begin{split}
\gam_1
&=
\frac12\int_0^\infty \bigpar{1-\rho_1(t)^2}\,t\dd t
=\frac{1}4 \bigsqpar{t^2(1-\rho_1(t)^2)}_0^\infty
+\frac12\int_0^\infty t^2 \rho_1(t)\dd \rho_1(t)	
\\&
=\frac12\int_0^1 t(\rho)^2 \rho\dd \rho	
=\frac12\int_0^1 \frac{\log^2(1-\rho)}{\rho}\dd \rho	
\\&
=\frac12\intoo \frac{x^2e^{-x}}{1-e^{-x}}\dd x
=\zeta(3),
	\end{split}
\raisetag{1.2\baselineskip}
  \end{equation}
where the final integral can be evaluated using a series expansion.
Hence we recover the limit $\zeta(3)$ found by \citet{Frieze}.
\end{example}

\begin{remark}
An argument similar to the proofs of Lemmas \ref{LW1} and \ref{LWoo}
shows that
\begin{equation}
  N_k(0,\infty)/n\pto  \frac12\int_0^\infty\bigpar{1-\rho_k(t)^2}\dd t.
\end{equation}
However, since $T_k$ has $n-1$ edges, we trivially
have $N_k(0,\infty)=k(n-1)$ a.s.
Hence, for any $k\ge1$,
\begin{equation}\label{intk}
 \frac12\int_0^\infty\bigpar{1-\rho_k(t)^2}\dd t = k.
\end{equation}
(This is easily verified for the case $k=1$, by calculations similar to
\eqref{kings}.)
Equivalently, for any $k\ge1$ (since \eqref{intk} holds trivially for $k=0$
too),
\begin{equation}\label{intkk}
 \frac12\int_0^\infty\bigpar{\rho_{k-1}(t)^2-\rho_k(t)^2}\dd t = 1.
\end{equation}
\end{remark}
Equation \eqref{intkk} is weakly supportive of \refConj{Conj-rhoInfinity}.

\begin{proof}[Proof of \refT{TE}]
It follows from \eqref{nk} that $N_k(0,t)\le k(n-1)$ and
thus, using also \eqref{nw}, $W_k(0,b)\le kb$. Consequently, \refL{LW1} and
dominated convergence yield,
for every $b<\infty$,
\begin{equation}
  \label{elw1}
\E W_k(0,b) \to \frac12\int_0^b \bigpar{1-\rho_k(t)^2}t\dd t.
\end{equation}

\refL{LWin} shows that
$\E W_k(0,b) = \E W_k(0,\infty)-\E W_k(b,\infty)\to \E W_k(0,\infty)$
uniformly in $n$ as $b\to\infty$.
Hence,
\eqref{elw1} holds for $b=\infty$ too by the following routine
three-epsilon argument:
We have
\begin{align}
  &\E W_k(0,\infty)- \frac12\int_0^\infty \bigpar{1-\rho_k(t)^2}t\dd t
 \notag \\ &\hskip4em
=  \Bigpar{\E W_k(0,b)- \frac12\int_0^b \bigpar{1-\rho_k(t)^2}t\dd t}
  \notag\\&\hskip6em
  + \E W_k(b,\infty)- \frac12\int_b^\infty \bigpar{1-\rho_k(t)^2}t\dd t,
\end{align}
where, for any $\eps>0$, we can make all three terms on the \rhs{} less than
$\eps$ (in absolute value) by choosing first $b$ and then $n$ large enough.

The result follows since $w(T_k)=W_k(0,\infty)-W_{k-1}(0,\infty)$, \cf{}
\eqref{wlim}.
\end{proof}

We can now prove \refT{Tlang}.

\begin{proof}[Proof of \refT{Tlang}]
  Let $0\le a<b\le\infty$, and
let $N(a,b)$ be the total number of edges arriving to $G_1(s)$ in
  the interval $s\in\ab$.
Then $N(a,b)\sim\Po\bigpar{\binom n2\frac{1}n(b-a)}$, and by
  the law of large numbers,
for any $\eps>0$, \whp
\begin{equation}\label{matteus}
  \tfrac{1}2(b-a-\eps)n\le N(a,b)\le \tfrac{1}2(b-a+\eps)n.
\end{equation}
The number of edges passed to $G_k(s)$ in $\ab$ is $N(a,b)-N_{k-1}(a,b)$,
  and thus it follows from \eqref{matteus} and \eqref{ln} that
for any $\eps>0$, \whp
\begin{equation*}\quad
e(G_k(b))-e(G_k(a))=N(a,b)-N_{k-1}(a,b)
\;
\begin{cases}
\ge \tfrac{1}2(b-a)\bigpar{\rho_{k-1}(a)^2-\eps}n,
\\
\le \tfrac{1}2(b-a)\bigpar{\rho_{k-1}(b)^2+\eps}n.
\end{cases}
\end{equation*}

For any fixed $t>0$, we obtain \eqref{langfredag}
by partitioning the interval $[0,t)$ into small
subintervals and taking limits as in the proof of \refL{LW1}.
Uniform convergence for $t\in[0,T]$ then follows as in the proof of
\refT{T1}\ref{T1C1}, \cf{} \eqref{tretton}.
\end{proof}

\section{Simple graphs and proof of \refT{T0}}\label{SotherModels}
In the \emph{Poisson (process) model} studied so far, we have a multigraph
with an
infinite number of parallel edges (with increasing costs) between each pair
of vertices.

It is also of interest to consider the simple graph $K_n$ with a single edge
(with random cost) between each pair of vertices, with the costs \iid{}
random variables.
We consider two cases, the
\emph{exponential model} with costs $\Exp(1)$ and
the \emph{uniform model} with costs $U(0,1)$.
When necessary, we distinguish the three models by superscripts
P, E, and U.

We use the standard coupling of the exponential and uniform models:
if $\eX_{ij}\sim\Exp(1)$ is the cost of edge $ij$ in the exponential model,
then the costs
\begin{equation}
  \label{xue}
\uX_{ij}:=1-\exp(-\eX_{ij})
\end{equation}
are \iid{} and $U(0,1)$, and thus
yield the uniform model.
Since the mapping $\eX_{ij}\mapsto \uX_{ij}$ is monotone,
the Kruskal algorithm selects the same set of edges for both models, and
thus the trees $T_1,T_2,\dots$ (as long as they exist) are the same for both
models; the edge costs are different, but since we select edges with small
costs,
$\uX_{ij}\approx \eX_{ij}$ for all edges in $T_k$ and thus $w(T\xUni_k)\approx
w(T\xExp_k)$; se \refL{LTB} for a precise statement.

\begin{remark}\label{Rotherdistributions}
  We can in the same way couple the exponential (or uniform) model with a
  model with \iid{} edge costs with any given distribution.
It is easily seen,
by the proof below and
arguments as in \citet{Frieze} or \citet{Steele1987} for $T_1$,
that
\refT{T0} extends to any edge costs $X_{ij}$
that have
a continuous distribution on $\ooo$ with the distribution
function
$F(x)$ having a right derivative $F'(0+)=1$ (for example,
an absolutely continuous distribution with a density function $f(x)$ that
is right-continuous at $0$ with $f(0+)=1$);
if $F'(0+)=a>0$, we obtain instead $w(T_k)\pto \gam_k/a$.
This involves no new arguments, so we confine ourselves to the
important models above as an illustration, and leave the general case to the
reader.
\end{remark}

Moreover, we obtain the exponential model from the Poisson model by
keeping only the first (cheapest) edge for each pair of vertices. We
assume throughout the section this coupling of the two models.
We regard also the exponential model as evolving in time, and define
$\eG_k(t)$ and $\eF_k(t)$ recursively as we did $G_k(t)$ and $F_k(t)$ in
\refS{Smodel}, starting with $\eG_1(t):=\xG_1(t)$, the simple subgraph of $K_n$
obtained by merging parallel edges and giving the merged edge the smallest cost
of the edges (which is the same as
keeping just the first edge
between each pair of vertices).

Recall from the introduction
that
while in the Poisson model every $T_k$ exists a.s.,
in the exponential and uniform models there is
a positive probability that $T_k$ does not exist,
for any $k\ge2$
and any $n\ge2$.
(In this case we define $w(T_k):=\infty$.)
The next lemma shows, in particular, that this
probability is $o(1)$ as \ntoo.
(The estimates in this and the following lemma are not best possible and can
easily be improved.)

\begin{lemma}\label{L2k}
In any of the three models and for any fixed $k\ge1$,
\whp{} $T_k$ exists and, moreover,
uses only edges of costs $\le2k\log n/n$.
\end{lemma}

\begin{proof}
Consider the exponential model;
the result for the other two models is an immediate consequence by the
couplings above (or by trivial modifications of the proof).
The result then says that \whp{} $\eG_k(2k\log n)$ is connected.

By induction, we may for $k\ge1$
assume that the result holds for $k-1$.
Thus, \whp, $\eG_{k-1}\bigpar{2(k-1)\log n}$ is connected, and then all later
edges are passed to $\eG_k(t)$.

Consider now the edges arriving in $(2(k-1)\log n,2k\log n]$.
They form a random graph $G(n,p)$ with
\begin{equation}
  \begin{split}
  p&=e^{-2(k-1)\log n/n}-e^{-2k\log n/n}
=e^{-2k\log n/n}\bigpar{e^{2\log n/n}-1}
\\&
=\bigpar{2+o(1)}\xfrac{\log n}n.	
  \end{split}
\end{equation}
As is well-known since the beginning of random graph theory
\cite{ER1959}, see \eg{} \cite{Bollobas},
such a random graph $G(n,p)$ is \whp{} connected. We have also seen that
\whp{} this graph $G(n,p)$ is a subgraph of $\eG_k(2k\log n)$.
Hence, $\eG_k(2k\log n)$ is \whp{} connected, which completes the induction.
\end{proof}

\begin{lemma}\label{LTA}
  For each fixed $k$,
  \begin{equation}\label{ltapprox}
	0\le w(T\xExp_k)-w(T\xPoi_k) = \Op\Bigpar{\frac{\log^3n}n}.
  \end{equation}
\end{lemma}

\begin{proof}
Since the exponential model is obtained from the Poisson model by deleting
some edges,
we have by \refL{thicker} that every edge contained in both processes
and contained in $F_1(t)\cup\dots\cup F_k(t)$
is also contained in $\eF_1(t)\cup\dots\cup \eF_k(t)$;
the only edges ``missing'' from the latter are those that were repeat edges in the Poisson model.

At time $t_k := 2k\log n$, how many repeat edges are there?
For two given vertices $i$ and $j$, the number of parallel edges is $\Po(t_k/n)$,
so the probability that it is two or more is
$p_2(t_k):=\P(\Po(t_k/n)\ge 2)\le (t_k/n)^2/2$.
(We use that the $k$th factorial moment of $\Po(\lambda)$ is $\lambda^k$,
and Markov's inequality.)
Hence, the number of pairs
$\set{i,j}$ with more than one edge is $\Bi(\binom n2,p_2(t_k))$, which is
stochastically smaller than $\Bi(n^2,(t_k/n)^2)$, which by Chebyshev's
inequality \whp{} is $\le 2t_k^2=8k^2\log^2n$.
Similarly,
the probability that $i$ and $j$ have three or more parallel edges is
$\leq (t_k/n)^3/6$ and thus
\whp{} there are no triple edges in $G_1(t_k)$.

By \refL{L2k}, \whp{}
$\pT_1\cup \dots\cup \pT_k
=F_1(t_k)\cup\dots\cup F_k(t_k)$,
and we have just established that
\whp{} all but at most
$2t_k^2$
of the edges in
$F_1(t_k)\cup\dots\cup F_k(t_k)$ are also in
$\eF_1(t_k)\cup\dots\cup \eF_k(t_k)\subseteq \eT_1\cup\dots\cup\eT_k$.
Since each spanning tree has exactly $n-1$ edges, the missing edges are
replaced by the same number of other edges, which by \refL{L2k} \whp{} also
have cost $\le t_k/n$
each, thus total cost at most $2t_k^3 = 16k^3\log^3n/n$.
Consequently, \whp,
\begin{equation}\label{ltab}
  w\bigpar{\eT_1\cup\dots\cup \eT_k}
\le
  w\bigpar{\pT_1\cup\dots\cup \pT_k}+16k^3\log^3 n/n.
\end{equation}
Having additional edges can never hurt (in this matroidal context),
so
\begin{equation}
  \label{ltabc}
  w(\pT_j)\le w(\eT_j),
  \qquad j\ge1.
\end{equation}
This yields the first inequality in \eqref{ltapprox},
while the second follows from \eqref{ltab} together with \eqref{ltabc} for
$j\le k-1$.
\end{proof}

\begin{lemma}
  \label{LTB}
  For each fixed $k$,
  \begin{equation}\label{ltb}
	0\le w(T\xExp_k)-w(T\xUni_k) = \Op\Bigpar{\frac{\log^2n}n}.
  \end{equation}
\end{lemma}

\begin{proof}
  As said above, $\eT_k$ and $\uT_k$ consist of the same edges, with edge
  costs related by \eqref{xue}. Since \eqref{xue} implies
$0\le \eX_{ij}-\uX_{ij}\le \frac12(\eX_{ij})^2$, it follows that, using
  \refL{L2k}, \whp{}
  \begin{equation}
0\le w\bigpar{\eT_k}-w\bigpar{\uT_k}
\le \sum_{e\in E(T_k)}(\eX_e)^2
\le n \Bigparfrac{2k\log n}{n}^2.
  \end{equation}
\end{proof}

\begin{proof}[Proof of \refT{T0}]
  It follows  from \refT{T0multi} and Lemma \ref{LTA}
that for each fixed $k$, $w(\eT_k)\pto\gam_k$, and
then from \refL{LTB} that
$w(\uT_k)\pto\gam_k$,
which is \refT{T0}.
\end{proof}

Recall that the corresponding statement for the expectation is false,
as $\E w(\eT_k)=\E w(\uT_k)=\infty$
for $k\ge2$;
see \refR{RE}.

\section{The second threshold}\label{Ssecond}
As noted in \refE{Egamma1} we do not know how to
calculate the limit $\gam_2$.
However, we can find the threshold $\gx_2$.
In principle, the method works for $\gs_k$ for any $k\ge2$,
provided we know $\rho_{k-1}$, so we will explain the method for general
$k$. However, we will assume the following:
\begin{align}\label{Arho}
  &\text{$\rhoki(x)$ is continuously differentiable
  on $(\gxki,\infty)$, with $\rhoki'(x)>0$,}
 \notag \\&\text{and $\rhoki'(x)$ is continuous as $x\downto\gxki$.}
\end{align}
This is, we think, not a serious restriction, for the following reasons.
First, \eqref{Arho} is easily verified for $k=2$, since we know $\rho_1$
explicitly (see \refE{Egamma1}),
so the calculation of $\gs_2$ is rigorous.
Second, we conjecture that \eqref{Arho} holds for all $k\ge2$, although we
have not proved this. (Cf.\ what we have proved in \refT{T1}.)
Third, even if this conjecture is wrong and
\eqref{Arho} does not hold for some $k$, we believe that the result below is
true, and can be shown
by suitable modifications of the argument and perhaps replacing
$\rhoki$ by suitable approximations.

By \eqref{tgx1}, $\gx_k$ is defined by $\norm{T_{\gx_k}}=1$.
As said in \refSS{SSintop},
since $\kk_{\gx_k}$ is bounded,
there exists an eigenfunction $\psi(x)\ge0$ of $T_{\gx_k}$ with eigenvalue
$\norm{T_{\gx_k}}=1$, \ie,
\begin{equation}\label{ev1}
  \psi(x)
=T_{\gx_k}\psi(x)
\eqdef
\intoo \kk_{\gx_k}(x,y)\psi(y)\dd \mu_{k-1}(y)
.
\end{equation}
By \eqref{kkt}, $\kk_\gxk(x,y)=0$ when
$x\ge\gxk$, and thus $\psi(x)=0$ when $x\ge\gxk$; furthermore, we can write
\eqref{ev1} as
\begin{equation}
  \label{ev2}
\psi(x)=\int_0^\gxk \bigpar{\gxk- x\vee y}\psi(y)\dd \mu_{k-1}(y),
\qquad x\in[0,\gxk].
\end{equation}
Since $\psi\in L^2(\muki)$, it follows from \eqref{ev2} that $\psi$ is
bounded and continuous on $[0,\gxk]$, and thus on $\ooo$. Moreover, it is
easily verified by dominated convergence
that we can differentiate under the
integral sign in \eqref{ev2}; thus $\psi(x)$ is differentiable on $(0,\gxk)$
with, using the definition \eqref{muk} of $\muki$,
\begin{equation}
  \label{ev'}
\psi'(x)=-\int_0^x \psi(y)\dd \mu_{k-1}(y)
=-\int_0^x \psi(y)\dd \rho_{k-1}(y),
\qquad x\in(0,\gxk).
\end{equation}

We now use our assumption in \eqref{Arho}
that $\rhoki(x)$ is continuously differentiable
on $(\gxki,\infty)$.
Then \eqref{ev'} yields, for $ x\in(\gxki,\gxk)$,
\begin{equation}
  \label{ev''}
\psi''(x)=-  \rho'_{k-1}(x)\psi(x),
\end{equation}
Furthermore, we have the boundary conditions
$\psi(\gxk)=0$ and, by \eqref{ev'}, $\psi'(\gxki)=0$.
We can thus find $\gxk$ by solving the
Sturm--Liouville equation \eqref{ev''} for $x\ge\gxki$
with the initial conditions $\psi(\gxki)=1$, say,
and $\psi'(\gxki)=0$; then $\gxk$ is the first zero of $\psi$ in
$(\gxki,\infty)$. (Note that the solution $\psi$ of \eqref{ev''} with these
initial conditions is unique,
again using \eqref{Arho}.)

We can transform the solution further as follows, which is advantageous
for $k=2$, when we know the inverse of $\rhoki$ explicitly.
Since $\rhoki$ is strictly increasing on $[\gxki,\infty)$, there exists an
  inverse
\begin{equation}\label{rhoki}
  \gf:=\rhoki\qw:[0,1)\to[\gxki,\infty) .
\end{equation}
Let
\begin{equation} \label{skdef}
  s_k:=\rhoki(\gxk)\in(0,1)
\end{equation}
and, for $x\in[0,s_k]$, let
\begin{equation} \label{h}
  h(x):=\psi(\gf(x))
\end{equation}
and
\begin{equation}\label{H}
 H(x):=\int_0^x h(y)\dd y.
\end{equation}
Then $h$ is continuous on $[0,s_k]$, and on the interior $(0,s_k)$,
\begin{equation}  \label{H'}
  H'(x)=h(x).
\end{equation}
The assumption \eqref{Arho} implies that
$\gf$ is differentiable on $(0,1)$. Thus, for $x\in(0,s_k)$,
\eqref{H'} and \eqref{h} yield
\begin{equation}
  H''(x)=h'(x)=\psi'(\gf(x))\gf'(x).    \label{H''1}
\end{equation}

We now make a change of variables in \eqref{ev'}.
In \eqref{change1} we use that $\rho_{k-1}(y)=0$ for $y \leq \gs_{k-1}$
from \refT{T1}\ref{T1gx};
in \eqref{change2} that $\gf(\rho_{k-1}(y))=y$ for $y\ge\gxki$ by definition \eqref{rhoki};
and in \eqref{change3} again that $\rho_{k-1}=\gf^{-1}$, as well as
$\rho_{k-1}(\gxki)=0$.
\begin{align}
\psi'(\gf(x))
&
=-\int_{0}^{\gf(x)} \psi(y)\dd \rho_{k-1}(y)
\notag
\\
 &= -\int_{\gxki}^{\gf(x)} \psi(y) \frac{\dd \rho_{k-1}(y)}{\dd y} \dd y
\label{change1}
\\
& = -\int_{\gxki}^{\gf(x)} \psi(\gf(\rho_{k-1}(y))) \rho'_{k-1}(y)  \dd y
\label{change2}
\\&
=-\int_0^{x} \psi(\gf(\rho))\dd \rho
\label{change3}
\\&
=-\int_0^{x} h(\rho)\dd \rho
=-H(x) .  \label{negH}
\end{align}
By \eqref{H''1} and \eqref{negH},
for $x\in(0,s_k)$,
\begin{equation}\label{H''}
  H''(x)=-\gf'(x)H(x).
\end{equation}
Note that $H(0)=0$ by \eqref{H} while, by \eqref{H'}, \eqref{h}, \eqref{skdef}, \eqref{rhoki}, and \eqref{ev2},
\begin{equation} \label{H'0}
  H'(s_k)
   \eqdef h(s_k)
   \eqdef \psi(\gf(s_k))
   = \psi(\gf(\rhoki(\gxk)))
   =\psi(\gxk)=0 .
\end{equation}
Furthermore, \eqref{ev2} yields $\psi(x)>0$ for $x\in[0,\gs_k)$,
  and thus \eqref{h} and \eqref{rhoki}--\eqref{skdef}
  yield
\begin{align}\label{h>0}
  h(x) & \eqdef \psi(\gf(x)) >0 \text{ for } x\in[0,s_k),
\end{align}
as $\gf(0)=0$ while $\gf(s_k)=\gs_k$.

Consequently, if we solve the Sturm--Liouville equation \eqref{H''} with
\begin{align}\label{H0H'0}
  H(0)=0 \quad\text{and}\quad H'(0)=1
\end{align}
(the former, repeating, by \eqref{H},
the latter by \eqref{H'}, \eqref{h>0}, and convenient scaling,
going back to the arbitrary scale of $\psi$ in \eqref{ev1}),
then $s_k$ is the smallest positive root of $H'(s_k)=0$,
and by \eqref{skdef} and \eqref{rhoki},
$\gxk=\gf(s_k)$.

We can reduce \eqref{H''} to a first-order equation by the substitution
\begin{align}\label{HRgth}
  H=R\sin\gth\quad\text{and}\quad H'=R\cos\gth,
\end{align}
which yields
  \begin{align}
    R'(x) \sin \gth(x) + R(x)\gth'(x)\cos \gth(x)&=H'(x)=R(x)\cos\gth(x),
    \\
   R'(x) \cos \gth(x) - R(x)\gth'(x)\sin \gth(x)&=H''(x)=-\gf'(x)R(x)\sin\gth(x).
  \end{align}
Multiplying the first equation by $\cos\gth(x)$ and the second by
$\sin\gth(x)$ and subtracting yields, after division by $R(x)$,
\begin{equation}\label{ior}
  \gth'(x)=\cos^2(\gth(x))+\gf'(x)\sin^2(\gth(x)).
\end{equation}
Since $\gf'(x)>0$ on $(0,1)$, by \eqref{rhoki} and the assumption
  $\rho_{k-1}'>0$ on $(\gxki,\infty)$ in \eqref{Arho},
  it follows from \eqref{ior} that
  $\gth'(x)>0$, and thus the function $\gth(x)$ is strictly increasing and
  thus invertible. Moreover,
  taking reciprocals,
  \eqref{ior} shows that the inverse function $x(\gth)$ satisfies
\begin{equation}\label{dxgth}
  \frac{\ddx x}{\ddx\gth} =\frac{1}{\cos^2\gth+\gf'(x(\gth))\sin^2\gth}.
\end{equation}
Now, $H(0)=0$ and $H'(0)=1$ (see \eqref{H0H'0}),
so from \eqref{HRgth} $\gth(0)$ is a multiple of $2\pi$, and we take
$\gth(0)=0$.
Also, by \eqref{h>0} $H$ is increasing up to $s_k$ and
thus $H(s_k)>0$, while $H'(s_k)=0$ by \eqref{H'0}, so
\eqref{HRgth} yields
$\gth(s_k)=\pi/2$. Thus
\begin{equation} \label{skx}
  s_k=x(\pi/2) ,
\end{equation}
where
$x(\gth)$
is the solution of \eqref{dxgth} with $x(0)=0$
(since $\gth(0)=0$ as just shown).

For $k=2$, as said in \refE{Egamma1} we have
$\rho_1(t)=1-e^{-t\rho_1(t)}$ and thus
the inverse $\gf(x)=-\log(1-x)/x$.
A numerical solution of \eqref{dxgth} (with Maple) yields
from \eqref{skx} that
$s_2\doteq \num{0.91511}$ and
thus
\begin{equation}\label{gx2}
\gx_2=\gf(s_2)\doteq \num{2.69521}.
\end{equation}

\section{A related problem by Frieze and Johansson}\label{Sunion}

As said in the introduction,
\citet{FriezeJ} recently considered the problem of finding the minimum
total cost of $k$ edge-disjoint spanning trees in $K_n$, for a fixed integer
$k\ge2$.
(They used random costs with the uniform model, see \refS{SotherModels}; we
may consider all three models used above.)
We denote this minimum cost by $\mmm_k$, %=\mmm_k(K_n)
following \cite{FriezeJ} (which uses $\mmm_k(K_n,\mathbf{X})$
for the random variable,
where $\mathbf{X}$ is the vector of random edge costs, and uses $\mmm_k(K_n)$ for its expectation).
Trivially,
\begin{equation}
  \mmm_k\le \sumik w(T_i),
\end{equation}
and as said in the introduction, it is easy to see that strict inequality
may hold when $k\ge2$, \ie,
that our greedy procedure of choosing $T_1,T_2,\dots$ successively does not
yield the minimum cost set of $k$ disjoint spanning trees.

We assume in this section that $n\ge 2k$; then $k$ edge-disjoint spanning
trees exist and thus $\mmm_k<\infty$.
(Indeed, $K_{2k}$ can be decomposed into $k$ Hamilton paths,
as shown in 1892 by Lucas \cite[pp.~162--164]{Lucas}
using a construction he attributes to Walecki.%
\footnote{%
Lucas introduces the problem as one of ``Les Jeux de Demoiselles'', namely ``Les Rondes Enfantines'',
a game of children holding hands in a circle repeatedly, never repeating a partner.
The conversion between the Hamilton cycles of the game and the Hamilton paths serving as our spanning trees is simple, and Walecki's construction is more naturally viewed in terms of Hamilton paths.
For much stronger recent results on Hamilton decompositions, see for example \cite{KO}.
})

\begin{remark} \label{RunionMatroid}
As observed by \citet{FriezeJ}, the problem is equivalent to finding the
minimum
cost of a basis in the matroid $\cM_k$,
defined as the union matroid
of $k$ copies of the cycle matroid of $K_n$.
This means that the elements of
$\cM_k$ are the edges in $K_n$, and a set of edges is
independent in $\cM_k$ if and only if it can be written as the union of $k$
forests, see \eg{} \cite[Chapter 8.3]{Welsh}.
(Hence, the bases, i.e., the maximal independent sets, are
precisely the unions of $k$ edge-disjoint spanning trees.
For the multigraph version in the Poisson model, of course we use instead the
union matroid of $k$ copies of the cycle matroid of $\Knoo$; we use the same
notation $\cM_k$.)
We write $r_k$ for rank in this matroid.
\end{remark}

Kruskal's algorithm, recapitulated in the introduction, is valid for finding a minimum cost
basis in any matroid; see \eg{} \cite[Chapter 19.1]{Welsh}.
In the present case it means that we process the edges
in order of increasing cost and keep the ones that are not dependent
(in $\cM_k$)
on the ones already selected; equivalently, we keep the next
edge $e$ if
$r_k(S\cup\set{e})>r_k(S)$, where $r_k$ is the rank function in $\cM_k$ and
$S$ is the set of edges already selected.

\begin{remark}\label{RFJmodels}
  It follows that the largest individual edge cost
for the optimal set of $k$ edge-disjoint spanning trees is at most the
largest edge cost for any given set of $k$ edge-disjoint spanning trees.
Hence, it follows from \refL{L2k} that for the random models studied here,
the optimal $k$ spanning trees \whp{} use only edges of cost $\le 2k\log n/n$.
It follows, with only minor modifications of the proofs,
that analogues of Lemmas \ref{LTA} and \ref{LTB} hold for
$\mmm_k$ for the three different models.
Hence, for limits in probability, the three models are equivalent for
$\mmm_k$ too.

Moreover, one can similarly show that for any $b>0$ there is a constant $B$
such that with probability at least $1-n^{-b}$,
the optimal $k$ spanning trees \whp{} use only edges of cost $\le Bk\log n/n$.
One can then argue as for the minimum spanning tree, see \eg{}
\cite{FriezeMcDiarmid}, \cite[Section 4.2.3]{McDiarmid1998} or
\cite[Example 3.15]{BoucheronEtAl}, and obtain
strong concentration of $\mmm_k$ for any of the three models; in particular
$\Var(\mmm_k)=o(1)$, and thus convergence of the expectation $\E(\mmm_k)$ is
equivalent to convergence in probability of $\mmm_k$.

\citet{FriezeJ} stated their results for the expectation $\E\mmm_k$
(for the uniform model), but the
results thus hold also for convergence in probability (and for any of the
three models).
\end{remark}

For $k=2$,
\citet{FriezeJ} show
that the expectation
\begin{equation} \label{mu2}
\E\mmm_2\to\mu_2\doteq 4.170\xxdots .  %4.17042881$
\end{equation}
This is strictly smaller than our estimate
for the total cost of two edge-disjoint spanning trees chosen successively,
$\gam_1+\gam_2 \doteq 1.202\ldots + 3.09\ldots > 4.29$;
see Tables \ref{sim10} and \ref{Tgamma10}.
This would show that choosing minimum spanning trees one by one is not optimal,
even asymptotically, except that our
estimates are not rigorous.
The following theorem is less precise but establishes rigorously
(subject to the numerical solution to \eqref{dxgth} giving $\gx_2$
as in \eqref{gx2})
that the values are indeed different.

\begin{theorem}\label{Tdiff}
There exists $\gd>0$ such that,
for any of the three models,
 \whp{}
$w(T_1)+w(T_2)\ge \mmm_2+\gd$.
\end{theorem}

With $\mu_2$ defined by the limit in \eqref{mu2}, this can be restated in the following equivalent form.

\begin{corollary}\label{Cdiff}
  $\gam_1+\gam_2>\mu_2$.
\end{corollary}
\begin{proof}
The equivalence of the statements in \refT{Tdiff} and \refC{Cdiff}
is immediate since $w(T_1)\pto\gam_1$ and $w(T_2)\pto\gam_2$
by Theorem \ref{T0} or \ref{T0multi} (depending on the choice of model), and $\mmm_2\pto\mu_2$ by
\cite{FriezeJ} and \refR{RFJmodels}.
\end{proof}

The proof of the theorem is based on the fact that many edges are rejected from $T_1$ and $T_2$
after time $\gx_2$, but none is rejected from the union matroid until a time $c_3$,
and $c_3$ (which we will show is the threshold for appearance of a 3-core
in a random graph) is later than $\gx_2$.

We begin with three elementary lemmas
that are deterministic, and do not assume any particular distribution
of edge costs; nevertheless,
we use the same scaling of time as before, and say that an edge with cost
$w$ is born at time $nw$.
(\refL{LF1}  has been used in several works,
including \cite{FriezeJ}, in the study of minimum spanning trees.)

\begin{lemma}\label{LF1}
Suppose that we select $N$ edges $e_1,\dots,e_N$, by any procedure, and that
$e_i$ has cost $w_i$.
Let $N(t):=|\set{i:\wi\le t/n}|$, the number of selected
edges born at or before time
$t$. Then the total cost is
\begin{equation}
  \sum_{i=1}^N\wi = \frac{1}n\intoo (N-N(t))\dd t.
\end{equation}
\end{lemma}
\begin{proof}
  \begin{equation*}
	\begin{split}
  \sum_{i=1}^N\wi
&=
 \sum_{i=1}^N \frac{1}n\intoo \ett{t/n<\wi}\dd t
=
 \frac{1}n\intoo \sum_{i=1}^N\bigpar{1- \ett{\wi\le t/n}}\dd t
\\&
= \frac{1}n\intoo (N-N(t))\dd t.
	\end{split}
\qedhere
  \end{equation*}
\end{proof}

For the next lemma, recall from \refR{RunionMatroid} that $r_k$ is rank in
the union matroid $\cM_k$.
We consider several (multi)graphs with the same vertex set $[n]$, and we
  define the intersection $G\cap H$ of two such graphs by
  $E(G\cap  H):=E(G)\cap E(H)$.
(We regard the multigraphs as having labelled edges, so parallel edges are
distinguishable.)
Note too that the trees $\xT_i$ in the lemma are arbitrary, not necessarily
the trees $T_i$ defined in \refS{Smodelmodel}.
\begin{lemma}\label{LF2}
Consider $\Knoo$
with any costs $w_e\ge 0$.
Suppose that $\xT_1,\dots,\xT_k$ are any $k$ edge-disjoint spanning trees.
For $t\ge0$, let $G(t)$ be the graph with edge set
$\set{e\in E(\Knoo):w_e\le t/n}$,
and let $N(t) :=
  e\bigpar{G(t) \cap (\xT_1 \cup \dots \cup \xT_k)}$.
Then,
  $N(t)\le r_k(G(t))$ for every $t$,
and
\begin{equation}\label{lf2}
  \sum_{i=1}^k w(\xT_i)-\mmm_k = \frac{1}n\intoo \bigpar{r_k(G(t))-N(t)}\dd t.
\end{equation}
\end{lemma}

\begin{proof}
First,
$N(t)$ is by definition
the number of edges in
$E\bigpar{G(t) \cap (\xT_1 \cup \dots \cup \xT_k)}$,
an independent
(with respect to $\cM_k$)
subset of $E(G(t))$, and thus $N(t)\le r_k(G(t))$,
as asserted.

Now apply \refL{LF1}, taking $N=k(n-1)$,  taking the edges $e_1,\dots,e_N$
to be the $N$ edges in $\xT_1\cup\dots\cup \xT_k$,
and noting that the definition of
$N(t)$ in \refL{LF1} matches that here.
This yields
\begin{equation}\label{lf2a}
  \sum_{i=1}^kw(\xT_i) = \frac{1}n\intoo (N-N(t))\dd t.
\end{equation}
Next, as a special case, consider a collection of $k$ spanning trees $\hT_1,\ldots,\hT_k$
with minimum total cost.
(Since we are in a deterministic setting, such a collection may not be unique.)
We may assume that they are found
  by Kruskal's algorithm,
and thus,
  for every $t$, the set of edges in
$G(t) \cap(\hT_1\cup\dots\cup\hT_k)$ is a maximal set of independent
edges in $G(t)$ (independent with respect to $\cM_k$),
hence the number of these edges is $N(t) = r_k(G(t))$.
Consequently, \refL{LF1} yields
\begin{equation}\label{lf2b}
\mmm_k=  \sum_{i=1}^kw(\hT_i) = \frac{1}n\intoo \bigpar{N-r_k(G(t))}\dd t.
\end{equation}
The result \eqref{lf2} follows by subtracting \eqref{lf2a} from \eqref{lf2b}.
\end{proof}

\begin{lemma}  \label{Lcore}
Let the multigraph $G$ be a subgraph of $\Knoo$
and assume that the $(k+1)$-core of $G$ is empty for some $k\ge1$.
Then
the edge set $E(G)$
is a union of $k$ disjoint forests.
In other words,
$r_k(G)=e(G)$.
\end{lemma}

The properties in two last sentences are equivalent by \refR{RunionMatroid}.

\begin{proof}
We use induction on  $|G|$; the base case $|G|=1$ is trivial.

If $|G|>1$ and $|G|$ has an empty $(k+1)$-core, then there exists a vertex $v$
  in $G$ of degree $d(v)\le k$. Let $G'$ be $G$ with $v$ and its incident edges
  deleted. By the induction hypothesis, $E(G')$ is the union of $k$
  edge-disjoint forests $F_1',\dots,F_k'$.
  These forests do not contain any edge with $v$ as an endpoint,
so we may simply add the first edge of $v$ to $F'_1$,
the second to $F'_2$, and so on, to
obtain the desired decomposition of $E(G)$.

Alternatively, the lemma follows easily from a multigraph version of
a theorem of \citet{Nash-Williams1964},
appearing also as \cite[Theorem 8.4.4]{Welsh};
specifically, the matroidal proof in \cite{Welsh} extends to multigraphs.
This theorem hypothesizes that $G$ is ``sparse'',
meaning that for every vertex subset $A$, $e(G[A]) \leq k(|A|-1)$,
but this follows from our hypothesis.
If $G$ has empty core, so does $G[A]$,
thus $G[A]$ has a vertex $v$ of degree $\leq k$,
whose deletion leaves another such vertex,
and so on until there are no edges, showing that $e(G[A]) \leq k(|A|-1)$.
\end{proof}

\begin{proof}[Proof of \refT{Tdiff}]
  By Lemmas \ref{LTA}--\ref{LTB} and \refR{RFJmodels},
the choice of model does not matter; for convenience we again take the Poisson model.

We use \refL{LF2}, with
the first and second minimum spanning trees,
\ie, $k=2$ and $\xT_j=T_j$, $j=1,2$.
Then the lemma's $G(t) = G_1(t)$ as defined in \refS{Smodelmodel} and used
throughout, and
\begin{equation}\label{eff}
  N(t)  \eqdef e\bigpar{G_1(t) \cap (T_1 \cup T_2)} =
  \sum_{i=1}^2 e\bigpar{G_1(t)\cap T_i}
=e(F_1(t))+e(F_2(t)).
\end{equation}

Now, by \citet{PittelSpencerWormald}, see also \cite{SJ184},
the 3-core threshold of a random graph is
$c_3:=\min_{\gl>0} \left( \gl/\P(\Po(\gl)\ge 2) \right) \doteq3.35$,
so that for any $b<c_3$, \whp{} the 3-core of $G(n,b/n)$ is empty.
In our context this says that for all $t\le b$,
$\xG_1(t)$ has an empty 3-core.
This holds for the multigraph $G_1(t)$ too, \eg{} by the proof in \cite{SJ184}
which uses random multigraphs. (Alternatively, we can use $\xG_1(t)$ in
\eqref{ex} below, with a negligible error.)
Hence, \refL{Lcore} shows that \whp, for all $t\le b$,
$G(t)=G_1(t)$ has full rank, \ie, $r_2(G(t))=e(G(t))$.

Furthermore, by \eqref{gx2}, $\gx_2<c_3$.
Choosing any $a$ and $b$ with $\gx_2<a<b<c_3$,
by \eqref{lf2} and \eqref{eff},
\whp,
\begin{equation}\label{ex}
  \begin{split}
    w(T_1)+w(T_2)-\mmm_2
&\ge
\frac{1}n\int_0^b \bigpar{r_k(G(t))-N(t)}\dd t
\\&
= \frac{1}n\int_0^b \bigpar{e(G(t))-e(F_1(t))-e(F_2(t))}\dd t	
\\&
= \frac{1}n\int_0^b {e(G_3(t))}\dd t	
\ge\frac{1}n (b-a) e(G_3(a)).
  \end{split}
\end{equation}
But by \refT{Tlang},
  \begin{equation}\label{ex2}
    (b-a)  e(G_3(a))/n \pto (b-a)\cdot\frac12\int_0^a\rho_2(s)^2\dd s
    =:2\gd,
  \end{equation}
  where $\gd>0$ since
  $\rho_2(s)>0$ for $s>\gx_2$ by \refT{T1}.
  Thus the theorem is established by \eqref{ex} and \eqref{ex2}.
\end{proof}

\section{Conjectured asymptotics of
\texorpdfstring{$\rho_k(t)$}{\unichar{"03C1}\unichar{"2096}(t)}}
As discussed in \refS{Sbounds}, $\gam_k\sim 2k$ for large $k$, see
for example Corollary \ref{Cbound2}.
Moreover, simulations (see \refS{Snumerical})
suggest that the functions
$\rho_k(t)$ converge, after suitable translations. If so, and
assuming suitable tail bounds, \eqref{intk} implies that
the translations should be by $2k$, up to an arbitrary constant plus $o(1)$;
this is formalized in \refConj{Conj-rhoInfinity}.

It is easy to see that this, together with suitable tail bounds justifying
dominated convergence, by \eqref{wlim}
and \eqref{intkk} would imply
  \begin{equation}\label{gamk2}
	\begin{split}
\gam_k-2k&=
\frac12\int_0^\infty \bigpar{\rho_{k-1}(t)^2-\rho_k(t)^2}\,(t-2k)\dd t
\\&
=
\frac12\int_{-2k}^\infty \bigpar{\rho_{k-1}(x+2k)^2-\rho_k(x+2k)^2}\,x\dd x
\\&
\to
\frac12\int_{-\infty}^\infty
 \bigpar{\rho_{\infty}(x+2)^2-\rho_\infty(x)^2}\,x\dd x
,
	\end{split}
  \end{equation}
which would show \refConj{Conj-gamma2},
with
$\gd=\frac12\int_{-\infty}^\infty
\bigpar{\rho_{\infty}(x+2)^2-\rho_\infty(x)^2}\,x\dd x$
(and necessarily $\gd\in[-1,0]$,
see \refR{R-gamma2}).

Recall that $\rho_k(t)$ is given by Lemma \ref{LC1}
as the survival probability
of the branching process $\fX_t$ defined in \refSS{SSbranching}
with kernel $\kk_t(x,y)$ on the probability space $(\bbR_+,\mu_{k-1})$ where
$\mu_{k-1}$ has the distribution function $\rho_{k-1}(t)$.
More generally, we could start with any distribution function $F(t)$ on
$\bbR_+$ and the corresponding probability measure $\mu$ and define a new
distribution function $\Psi(F)(t)$ as the survival probability
$\rho(\kk_t;\mu)$. This defines a map from the set of distribution functions
(or probability measures) on $\ooo$ into itself, and we have
$\rho_k=\Psi(\rho_{k-1})$. If one could show that $\Psi$ is a contraction for
some complete metric (perhaps on some suitable subset of distribution
functions), then
Banach's fixed point theorem would imply the existence of a unique fixed
point $\rho_\infty$, and convergence of $\rho_k$ to it.
However, the mapping $\Psi$ is quite complicated, and we leave the possible
construction of such a metric as an open problem.

Recall also that $t=\gs_k$ is where $\rho_k(t)$ becomes non-zero, see
\refT{T1}\ref{T1gx}.
Hence, \refConj{Conj-rhoInfinity} suggests also the following, related conjecture.
\begin{conjecture}\label{Conj-gx}
  There exists a real constant $\gsx$ such that as \ktoo,
  \begin{equation}
  \gx_k=2k+\gsx+o(1).
\end{equation}
In particular,
  \begin{equation}
  \gx_k-\gx_{k-1}\to 2.
\end{equation}
\end{conjecture}

\section{Computational results}\label{Snumerical}

\subsection{Naive simulations}
For intuition and as a sanity check on all calculations,
we first directly simulate the problem described in the introduction's Poisson edge-weight model.
Specifically, we take a graph with $n$ vertices and random edge weights
(\iid exponential random variables with mean 1),
find the MST, add fresh exponentials to the weights of the MST edges, and repeat
to get the second and subsequent MSTs.
For each MST, we plot each edge's rank within the MST, divided by $n$
(so, $1/n$ for the first edge, up to $(n-1)/n$ for the last) on the vertical axis,
against the edge's weight (multiplied by $n$ in accordance with our time scaling) on the horizontal axis.
The results are shown in \refF{FigNaive1}.
\begin{figure}[htbp]
  \centering
  \includegraphics[width=4in]{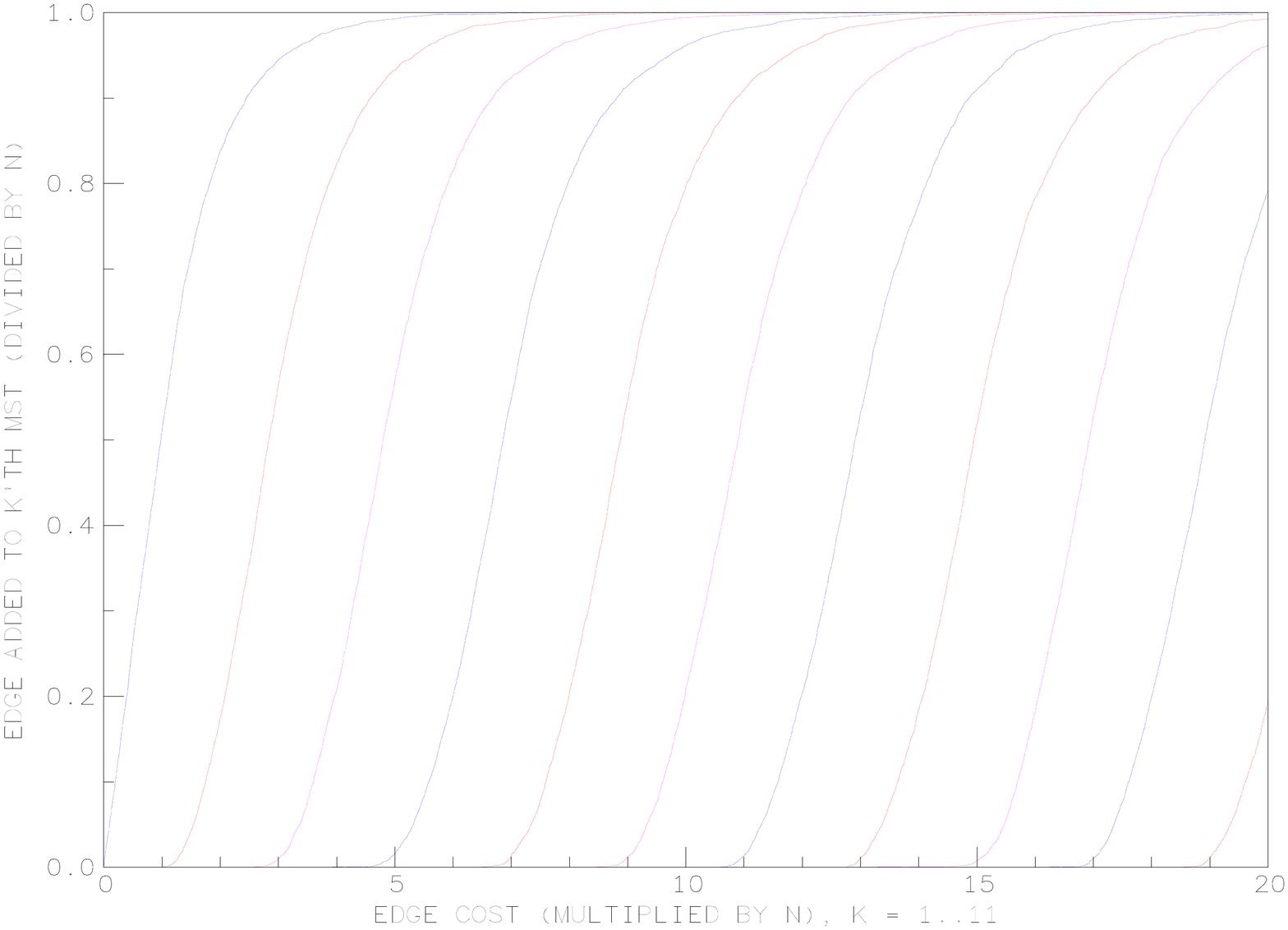}
  \caption{Size of $k$th MST (divided by $n$) plotted against the weight of the
  next edge added (multiplied by $n$), with $n=\num{4000}$,
  for $k=1,\ldots,11$.} \label{FigNaive1}
\end{figure}
The corresponding estimates of $\gamma_k$, for $k$ up to 5, are
$1.197, 3.055, 5.035, 7.086, 9.100$.
% $11.113, 13.130, 15.180, 17.158, 19.132$.
This was done for just a single graph with $n=\num{4000}$, not averaged over several graphs.
For a sense of the limited accuracy of the estimates, remember that $\gamma_1 = \zeta(3) = 1.2020\ldots$.

\subsection{Better simulations}
Better simulations can be done with reference to the model introduced in \refS{Smodelmodel}
and used throughout.
We begin with $k$ empty graphs of order $n$.
At each step we introduce a random edge $e$ and, in the first graph $G_i$
for which $e$ does not lie within a component, we merge the two components given by its endpoints.
(If this does not occur within the $k$ graphs under consideration, we do nothing, just move on to the next edge.)
For each graph we simulate only the components (i.e., the sets of vertices comprised by each component);
there is no need for any more detailed structure.
The edge arrivals should be regarded as occurring as a Poisson process of intensity
$(n-1)/2$
but instead we simply treat them as arriving at times
$2/n$, $4/n$, etc.

Figure \ref{Sim1M} depicts the result of a single such simulation
with $n=\num{1 000 000}$,
showing for each $k$ from 1 to 5 the size of the largest component of $G_k$ (as a fraction of $n$)
against time.
\begin{figure}
  \centering
  \includegraphics[width=4in]{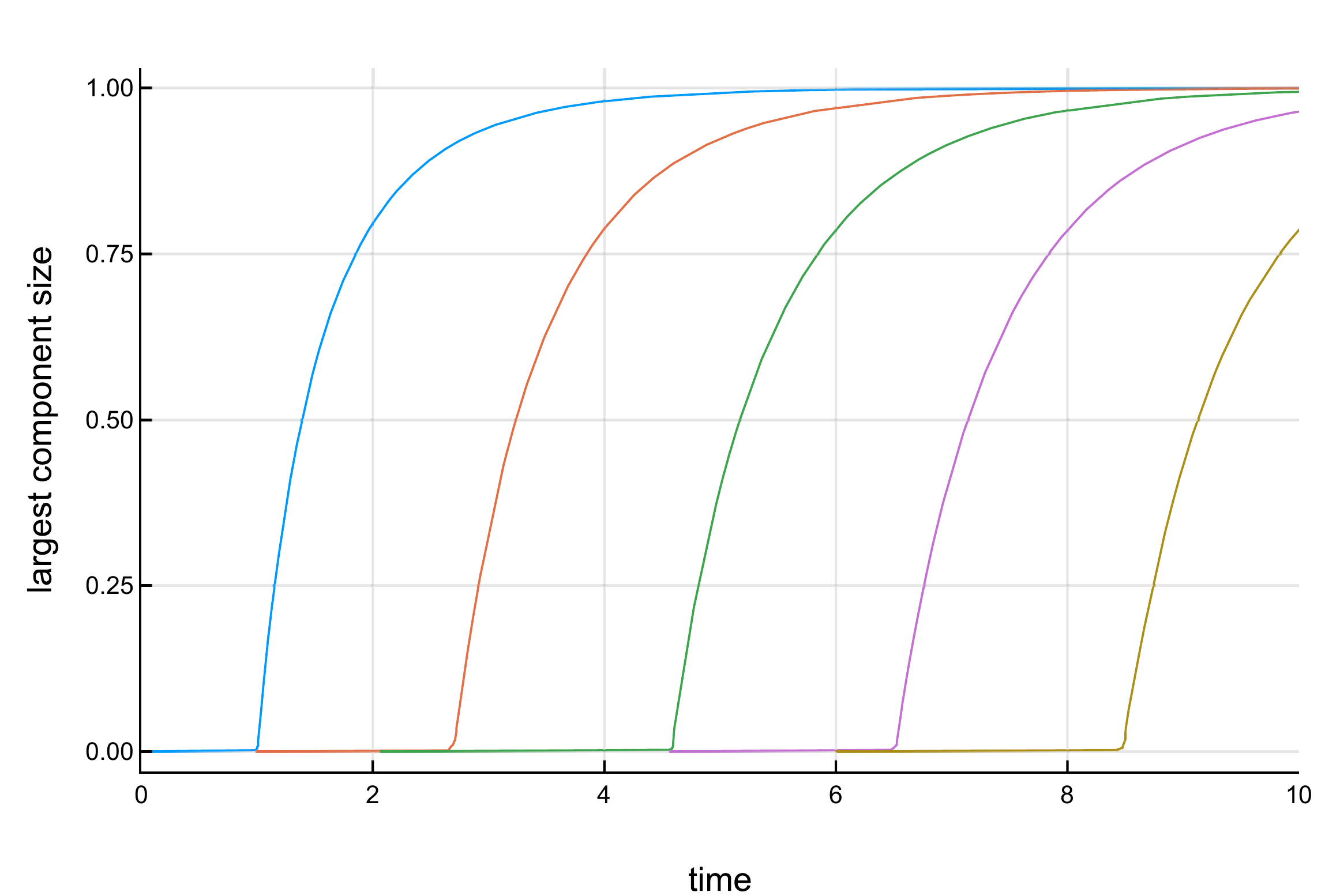}
  \caption{Largest component sizes, as a fraction of $n$, for graphs $G_1,\ldots,G_5$,
   based on a single simulation with $n=\num{1 000 000}$.}\label{Sim1M}
\end{figure}

A similar experiment was done with
10 simulations each with $n=\num{1 000 000}$, and another with
 100 simulations each with $n=\num{100 000}$.
Taking the measured MST costs as estimates of $\gamma_1,\ldots,\gamma_5$,
Table \ref{sim10} shows the sample means and standard errors.
The results of the two experiments are broadly consistent, including in that their estimated means differ by amounts on the order of their estimated standard errors.
\begin{table}
  \centering
\begin{tabular}{|c|c|c|c|c|c|}
  \hline
  \multicolumn{6}{|c|}{10 simulations each with $n=\num{1 000 000}$} \\
  \hline
   & $\gamma_1$ & $\gamma_2$ & $\gamma_4$ & $\gamma_4$ & $\gamma_5$ \\
  mean & 1.2025 & 3.0928 & 5.0457 & 7.0302 & 9.0149 \\
  std err & 0.0008 & 0.0011 & 0.0011 & 0.0019 & 0.0026 \\
  \hline
  \multicolumn{6}{c}{ \mbox{ }} \\
  \hline
  \multicolumn{6}{|c|}{100 simulations each with $n=\num{100 000}$} \\
  \hline
   & $\gamma_1$ & $\gamma_2$ & $\gamma_4$ & $\gamma_4$ & $\gamma_5$ \\
  mean & 1.2026 & 3.0913 & 5.0469 & 7.0299 & 9.0159 \\
  std err & 0.0005 & 0.0009 & 0.0014 & 0.0019 & 0.0020 \\
  \hline
\end{tabular}
\caption{Estimates of $\gamma_1,\ldots,\gamma_5$ from 10 simulations
each with $n=\num{1 000 000}$,
and 100 simulations each with $n=\num{100 000}$.}  \label{sim10}
\end{table}

A larger simulation, using 10 simulations each with $n=$10M, and up to time $t=40$
(i.e., 200M steps), supports
\refConj{Conj2k-1} that $\gamma_k = 2k-1+o(1)$;
see Figure \ref{sim10M}.
\begin{figure}[htbp]
  \centering
  \includegraphics[width=4in]{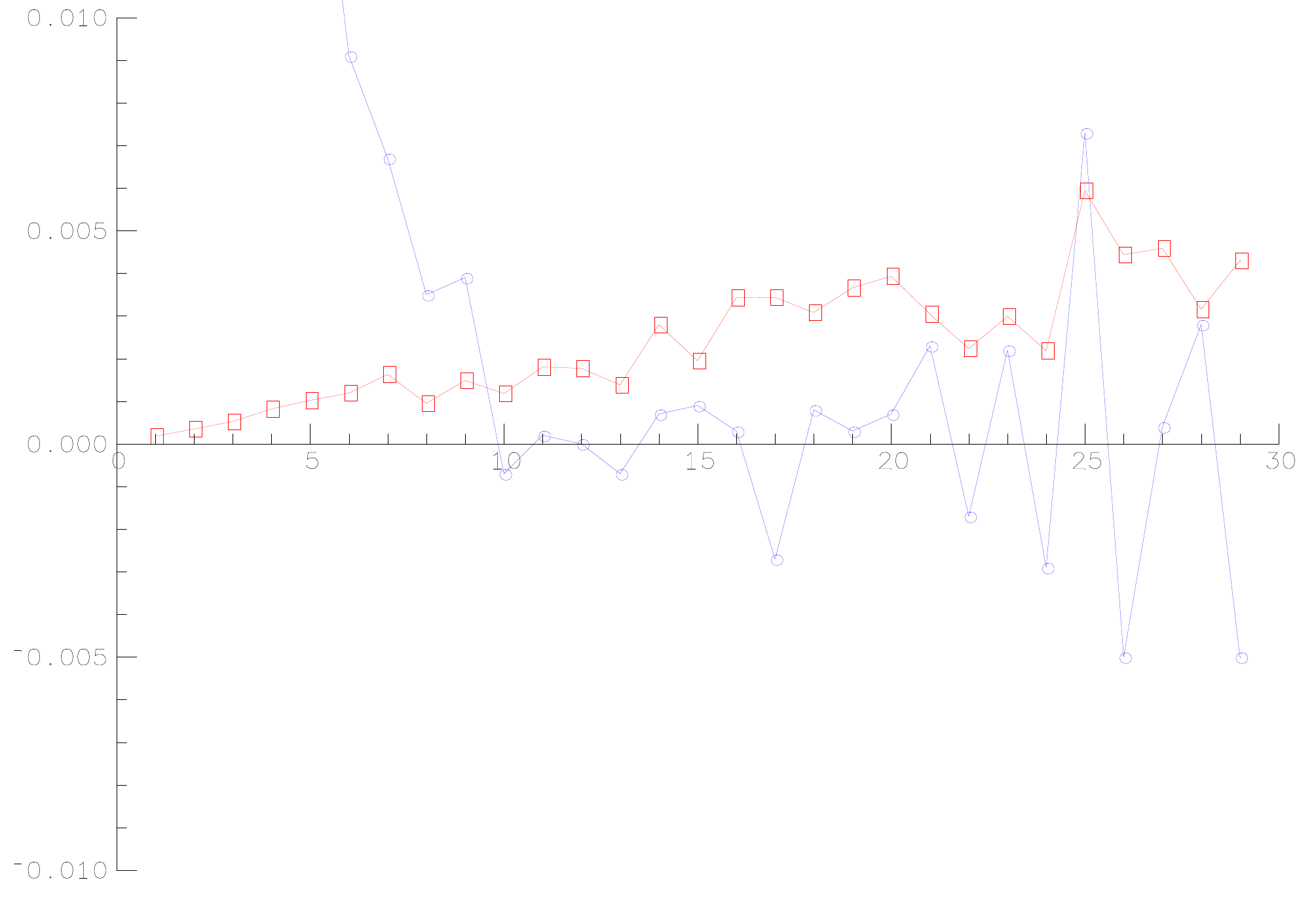}
  % APL2 variable gamma10m in MSTFIXEDPT3B.
  \caption{%
  Sample means of $\gamma_k-(2k-1)$ plotted against $k$ (in blue)
  and (in red) the standard errors for $\gamma_k$,
  supporting \refConj{Conj2k-1} that $\gamma_k = 2k-1+o(1)$.
  The statistics are taken from 10 simulations each with $n=$10M, and up
   to time $t=40$ (i.e., 200M steps).
 (We do not estimate $\gamma_{k}$ for $k>29$ because
 in one of the 10 simulations the giant component in $G_{30}$ did not yet cover all vertices.)
 }
   \label{sim10M}
\end{figure}

\subsection{Estimates of the improved upper bound} \label{Srho2Sim}
The differential equation system \eqref{pg},
giving the improved upper bound of \refS{SUB2},
is easy to solve numerically.
We did so as a discrete-time approximation,
setting $g_k(t+\Delta t) = g_k(t)+\frac12\Delta\bigpar{g_{k-1}(t)^2 - g_k(t)^2}$,
using $\Delta t= \num{0.000 01}$ and considering $k$ up to 50.

\refF{rho2OccSim} shows the results up to time $t=10$.
Because $g_k(t)$ pertains to a model in which all edges of $F_k$
are imagined to be in a single component,
this plot is comparable both to that in \refF{FigNaive1}
(which counts all edges)
and to those in \refF{Sim1M}
(which counts edges in the largest component)
and \refF{FigRho3from1to5}
(the theoretical giant-component size).

\begin{figure}[htbp]
  \centering
  \includegraphics[width=4in]{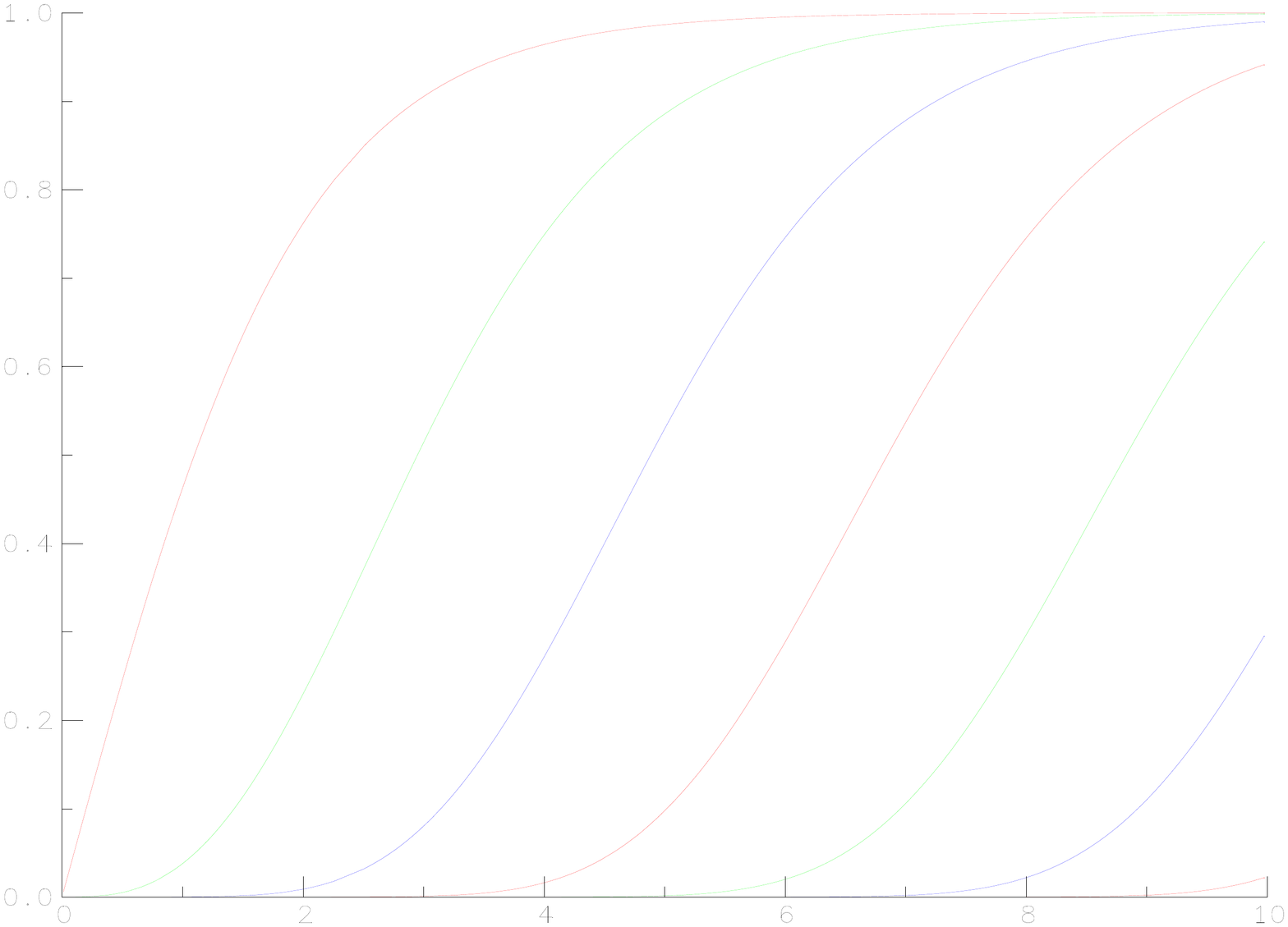}
  % APL2 variable gamma10m in MSTFIXEDPT3B.
  \caption{%
  Values $g_k(t)$ plotted against $t$.
  The function $g_7(t)$ is just rising from 0
  within the plot range;
  the values of $g_k(t)$ for larger $k$
  are too close to 0 to be seen.
 }
   \label{rho2OccSim}
\end{figure}

\refTab{TgGimproved} and
\refF{FGammaKrho2UB} show the corresponding upper bounds on $\gG_k$.
Specifically, the bound on $\gG_k$ from \eqref{gam+},
\newcommand{\gup}{\overline{\gG}}
call it $\gup_k$,
is estimated as
$\gup_k \doteq \tfrac12 \Delta \, \sum_{t \in T} t (1-g_k(t)^2)$
where $T=\set{0, \Delta, 2\Delta,\ldots}$.
Since we cannot sum to infinity, we terminate when the final $g_k$
under consideration is judged sufficiently close to 1,
specifically within \num{0.0000001} of 1.
It appears experimentally that the gap $1-g_k(t)$ decreases exponentially fast
(very plausible in light of \eqref{t1rholim})
so termination should not be a large concern;
see also \eqref{lwin}.

\begin{figure}[htbp]
  \centering
  \includegraphics[width=4in]{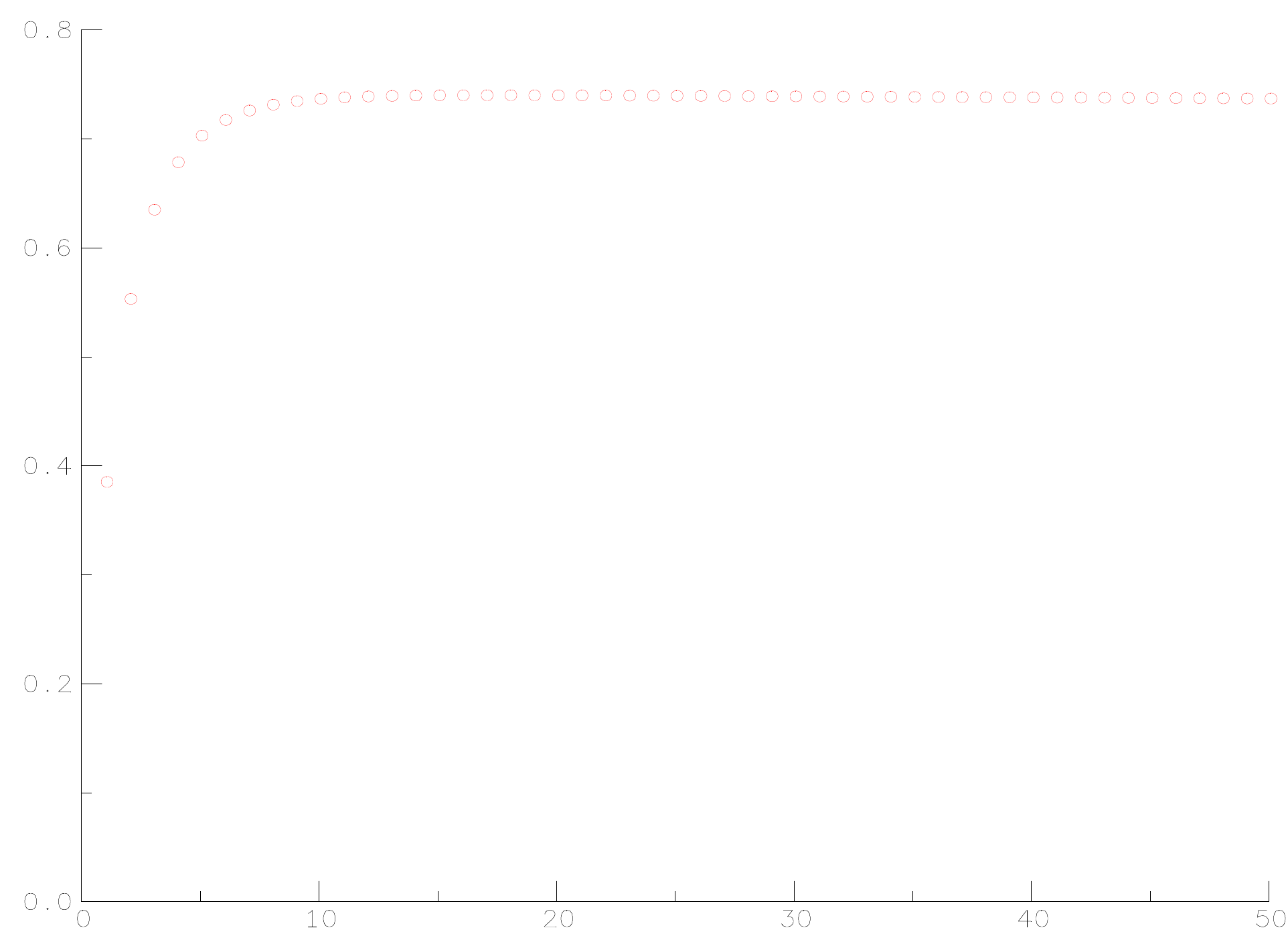}
  \caption{%
  Upper bounds from \eqref{gam+} on $\gG_k-k^2$, for $k$ from 1 to 50,
  based on numerical solution to the differential equations \eqref{pg}.
 }
   \label{FGammaKrho2UB}
\end{figure}

\begin{table}[htbp]
  \centering
  \footnotesize
  \begin{tabular}{|c|ccccccccc|}
    \hline
    $k$ & 1 & 2 & 3 & 4 & 5 & 6 & 7 & 8 & 9 \\
    $\gG_k$ &
    1.3863 & 4.5542 & 9.6362 & 16.6799 & 25.7045
  & 36.7189 & 49.7277 & 64.7331 & 81.7365 \\% & 100.7387 \\
    \hline
  \end{tabular}
  \caption{Upper bounds on $\gG_k$ obtained from numerical solution of \eqref{gam+}.
  }\label{TgGimproved}
\end{table}

\refF{FGammaKrho2UB} suggests that the gaps $\gup_k-k^2$ level off at about $0.743$.
(Beyond about $k=25$ the gaps decrease,
but using $\Delta t = \num{0.000 1}$ they continued to increase,
and in either case the degree of change is comparable with $\Delta t$
and thus numerically unreliable.)
This suggests the following conjecture.
(Recall from \eqref{cbounds} that $\gG_k \geq k^2$.)
\begin{conjecture} \label{ConjImproved}
For every $k \geq 1$, $\gG_k \leq \gup_k \leq k^2+\gdimp$
for some constant $\gdimp$.
\end{conjecture}
We established in \refS{SUB2} that $\gG_k \leq \gup_k$,
so only $\gup_k \leq k^2+\gdimp$ is conjectural.
If the conjecture holds, then it follows, using also \eqref{cbounds},
that
$\gamma_k = \gG_k-\gG_{k-1} \geq (k^2) - ((k-1)^2+\gdimp) = 2k-1-\gdimp$
and
$\gamma_k = \gG_k-\gG_{k-1} \leq (k^2+\gdimp)-(k-1)^2 = 2k-1+\gdimp$.
Hence, the conjecture would imply
\begin{equation}
2k-1-\gdimp \leq \gamma_k \leq 2k-1+\gdimp.
\end{equation}
In particular, if Conjecture \ref{ConjImproved} holds with
$\gdimp \leq 1$ as it appears, then $2k-2 \leq \gamma_k \leq 2k$.

\subsection{Estimates of the fixed-point distributions
 \texorpdfstring{$\rho_k$}{\unichar{"03C1}\unichar{"2096}}}
We also numerically estimated the distributions $\rho_k$;
recall from \refT{T1} that $C_1(G_k(t))/n$ $\pto$ $\rho_k(t)$.
We may begin with either $\rho_0(t)$, which is 0 for $t<0$ and 1 for $t \geq 0$,
or with $\rho_1(t)$, which as described in \refE{Egamma1} is the inverse function of
$-\log(1-\rho)/\rho$.
(Both choices gave similar results, the latter being slightly preferable numerically.)
We use $\rho_{k-1}$ to obtain $\rho_k$,
following the branching process described in \refS{SSbranching}.
The survival probability $\rho_t(x)$ at time $t$ of a particle born at time $x$ in the branching process equivalent of $G_k$
is given by the function $\rho_t = \rho_\kk$
which (see \eqref{rhoPhi})
is the largest fixed point of
\begin{align}\label{Phiksim}
  \rho_t &= \Phi_\kk\rho_t \eqdef 1-e^{-T_\kk\rho_t}
\end{align}
(the time $t$ is implicit in the kernel $\kk =\ kk_t$
thus in the operators $\Phi_\kk$ and $T_\kk$),
where (see \eqref{tk}) $T_\kk$ is given by
$\bigpar{T_\kk f}(x)=\intS \kk(x,y)f(y)\dd\mu(y)$.
With reference to the kernel $\kk$ defined in \eqref{kkt},
$\bigpar{T_\kk f}(x)=0$ for $x>t$, while otherwise,
with $\mu=\rho_{k-1}$ as in \eqref{muk},
\begin{align}
  \bigpar{T_\kk f}(x)
   &=     \int_{0}^{t} (t-x\vee y) f(y)\dd \rho_{k-1}(y) \notag
   \\&=   \int_{0}^{x} (t-x) f(y)\dd \rho_{k-1}(y)
        + \int_{x}^{t} (t-y) f(y)\dd \rho_{k-1}(y) . \label{Tksim}
\end{align}
Given $t$, to find $\rho_t$ numerically we iterate
\eqref{Phiksim},
starting with some $f$ known to be larger than $\rho_t$ and repeatedly
setting $f(x)$ equal to $\bigpar{\Phi_\kk f}(x)$;
this gives a
sequence of functions that converges to
the desired largest fixed point $\rho_t$, cf.\ \cite[Lemma 5.6]{SJ178}.
We will estimate $\rho_t$ for times $i \, \Delta t$,
for $\Delta t$ some small constant and $i=0,\ldots,I$,
with $I \Delta t$ judged to be sufficient time to observe all relevant behavior.
We initialize with $f \equiv 1$ to find $\rho_{I \Delta t}$,
then iteratively initialize with $f = \rho_{i \Delta t}$
to find $\rho_{(i-1) \Delta t}$.
Since the branching process is monotone in $t$
--- each vertex can only have more children by a later time $t$ ---
so is the survival probability, thus
$\rho_{i \Delta t}$ is larger than $\rho_{(i-1) \Delta t}$
and therefore a suitable starting estimate.
In practice we find that the process converges in 20 iterations or so
even for $\rho_{I \Delta t}$, and less for subsequent functions $\rho_{i \Delta t}$,
with convergence defined as two iterates differing
by at most $10^{-8}$ for any $x$.

For each $k$ in turn, we do the above for all times $t$,
whereupon the desired function
 $\rho_k (t) \eqdef \rho(\kk) = \rho(\kk_t) $
is given by  (see \eqref{rhokappa} and \eqref{rhok})
\begin{align}\label{rhokiterate}
 \rho_k(t)= \int_{0}^{\infty} \rho_t(x) \dd\rho_{k-1}(x) .
\end{align}
Do not confuse $\rho_t$ of \eqref{Phiksim} and $\rho_k$ of \eqref{rhokiterate},
respectively the $\rho_\kk$ and $\rho$ of \eqref{rhokdef};
see also \eqref{rhok} and the comment following it.

All the calculations were performed with time ($t$, $x$, and $y$) discretized to multiples of $\Delta t = 0.01$
and restricted to the interval $[0,10]$.
For a fixed $t$, the calculation in \eqref{Tksim} can be done
efficiently for all $x$.
The derivative of \eqref{Tksim} with respect to $x$
is $-\int_{0}^{x} f(y) \dd \rho_{k-1}(y)$
(cf.{} \eqref{ev2} and \eqref{ev'}).
So, given the value of \eqref{Tksim} for some $x$,
that at the next discrete $x$ is the discrete sum corresponding to this integral,
and in one pass we can compute these integrals (discretized to summations)
for all $x$.
Each computed $\rho_k(t)$ is translated by $2k$ to keep the functions' interesting regimes
within the time range $[0,10]$, before doing the computations for $k+1$,
but these translations are reversed before interpreting the results.

The first observation is that the estimates of $\rho_k$ are consistent with \refConj{Conj-rhoInfinity}.
As shown in \refF{FigRho3from1to5}, even the first few functions $\rho_k$ have visually very similar forms.
\begin{figure}[htbp]
  \centering
  \includegraphics[width=4in]{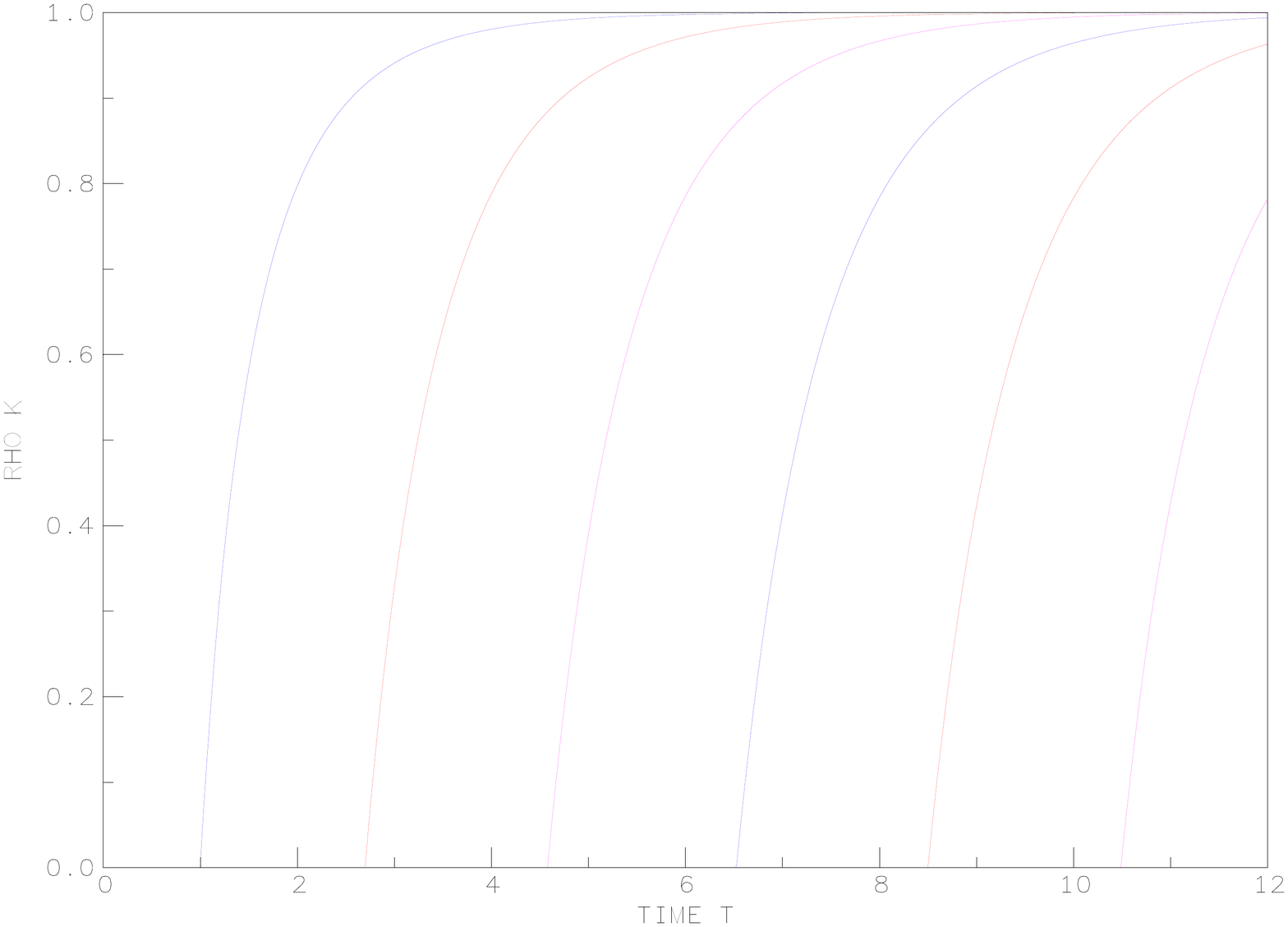}
  \caption{Estimates of $\rho_k$ for $k=1,\ldots,6$,
  already suggesting a limiting $\rho_\infty$ (up to translation).}\label{FigRho3from1to5}
\end{figure}

To make a more precise comparison, we time-shift each function $\rho_k$ so that it
reaches the value $1-e^{-1}$ at time $t=4$ (arbitrarily chosen).
\refF{rho3shift12thou} shows the thus-superposed curves for $\rho_1$, $\rho_2$, and $\rho_{\num{1000}}$;
the curve $\rho_5$ (not shown) is already visually indistinguishable from $\rho_{\num{1 000}}$.
\begin{figure}[htbp]
  \centering
  \includegraphics[width=4in]{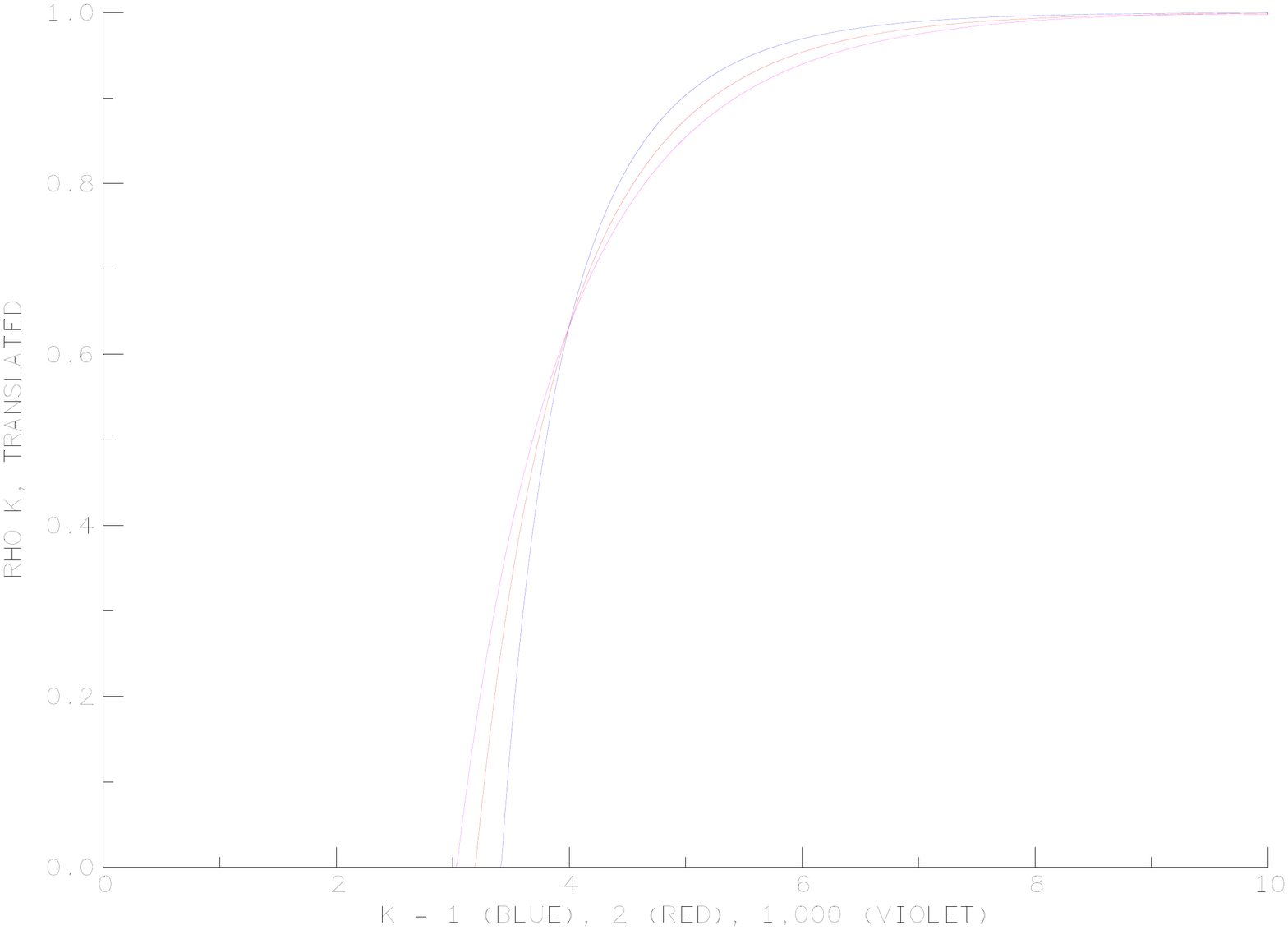}
  \caption{Functions $\rho_k$ time-shifted to coincide as nearly as possible,
  each shifted so as to make $\rho_k(4)=1-e^{-1}$,
  for $k=1$, 2, and \num{1000}.}\label{rho3shift12thou}
\end{figure}

Estimates for $\gamma_k$,
obtained from those for $\rho_k$ via \eqref{wlim},
are shown in \refTab{Tgamma10}.
Estimates of $\gamma_k$ for large $k$ were deemed numerically unreliable for two reasons.
First, discretization of time to intervals of size $\Delta t=0.01$ is problematic:
the timing of $\rho_1$ is uncertain to this order,
that of $\rho_2$ additionally uncertain by the same amount, and so on,
and translation of $\rho_k$ directly affects the corresponding estimate of $\gamma_k$.
Second, the time range $t \in [0,10]$ (translated as appropriate) used in the computations
proved to be too narrow, in that for large $k$ the maximum value of $\rho_k$ observed
only about $0.9975$,
and the gap between this and 1 may be enough to throw off the estimates of $\gamma_k$ perceptibly.

\begin{table}[htbp]
  \centering
  \footnotesize
  \begin{tabular}{|c|cccccccccc|}
    \hline
    $k$ & 1 & 2 & 3 & 4 & 5 & 6 & 7 & 8 & 9 & 10 \\
    $\gamma_k$ &
 1.202 & 3.095 & 5.057 & 7.043 & 9.038 & 11.039 & 13.042 & 15.047 & 17.052 & 19.058 \\
    \hline
  \end{tabular}
  \caption{Estimates of $\gamma_k$ from \eqref{wlim}.
  }\label{Tgamma10}
\end{table}

\section{Open questions}
We would be delighted to confirm
the various conjectures above, in particular
Conjectures \ref{Conj-rhoInfinity}--\ref{Conj2k-1},
and get a better understanding of (and ideally a closed form for) $\rho_\infty$
(provided it exists).

It is also of natural interest to ask this $k$th-minimum question for structures other than spanning trees.
Subsequent to this work, the length of the $k$th shortest $s$--$t$ path in a complete graph
with random edge weights has been studied in \cite{GerkeMS}.
The behavior is quite different:
the first few paths cost nearly identical amounts,
while \cite{GerkeMS} gives results for all $k$ from 1 to $n-1$.

The ``random assignment problem'' is to determine the cost of a
minimum-cost perfect matching in a complete bipartite graph with random edge weights,
and a great deal is known about it, by a variety of methods;
for one relatively recent work, with references to others, see \cite{WastlundAP}.
It would be interesting to understand the $k$th cheapest matching.

It could also be interesting to consider other variants of all these questions.
One, in the vein of \cite{FriezeJ}, is to consider the $k$ disjoint structures which together have the smallest possible total cost.
Another is to consider a second structure not disjoint from the first, but differing in at least one element, either of our choice or as specified by an adversary.

\section*{Acknowledgments}
We thank Oliver Riordan for helpful comments which simplified our proof,
and Bal\'azs Mezei for assistance with Julia programming.

\newcommand\AAP{\emph{Adv. Appl. Probab.} }
\newcommand\JAP{\emph{J. Appl. Probab.} }
\newcommand\JAMS{\emph{J. \AMS} }
\newcommand\MAMS{\emph{Memoirs \AMS} }
\newcommand\PAMS{\emph{Proc. \AMS} }
\newcommand\TAMS{\emph{Trans. \AMS} }
\newcommand\AnnMS{\emph{Ann. Math. Statist.} }
\newcommand\AnnPr{\emph{Ann. Probab.} }
\newcommand\CPC{\emph{Combin. Probab. Comput.} }
\newcommand\JMAA{\emph{J. Math. Anal. Appl.} }
\newcommand\RSA{\emph{Random Struct. Alg.} }
\newcommand\ZW{\emph{Z. Wahrsch. Verw. Gebiete} }
\newcommand\DMTCS{\jour{Discr. Math. Theor. Comput. Sci.} }

\newcommand\AMS{Amer. Math. Soc.}
\newcommand\Springer{Springer-Verlag}
\newcommand\Wiley{Wiley}

\newcommand\vol{\textbf}
\newcommand\jour{\emph}
\newcommand\book{\emph}
\newcommand\inbook{\emph}
\def\no#1#2,{\unskip#2, no. #1,} %(typeset after year)
\newcommand\toappear{\unskip, to appear}

\newcommand\arxiv[1]{\texttt{arXiv:#1}}
\newcommand\arXiv{\arxiv}

\def\nobibitem#1\par{}

\end{document}